\documentclass[reqno]{amsart}

\usepackage{amsmath,amssymb,latexsym,amsthm}
\usepackage{graphicx}
\usepackage{color}
\usepackage{subfigure}
\usepackage{multirow}
\usepackage{mathptmx}
\usepackage{url,bm}
\usepackage{enumerate}
\usepackage{mathtools}
\DeclarePairedDelimiter\ceil{\lceil}{\rceil}

\newtheorem{theorem}{Theorem}[section]
\newtheorem{corollary}{Corollary}[section]
\newtheorem{proposition}{Proposition}[section]
\newtheorem{lemma}{Lemma}[section]
\newtheorem{remark}{Remark}[section]
\newtheorem{Assumption}{Assumption}[section]

\newtheorem{definition}{Definition}[section]

\newtheorem{example}{Example}[section]
\newcommand*{\bigchi}{\mbox{\Large$\chi$}}
\newcommand{\Second}{\textup{I}\!\textup{I}}
\numberwithin{equation}{section}

\title[KDE with general kernel on manifold]{Strong uniform consistency with rates for kernel density estimators with general kernels on manifolds}

\author{Hau-Tieng~Wu}
\address{Hau-Tieng Wu\\
Departments of Mathematics and Department of Statistical Science\\
Duke University, Durham, NC, USA; Mathematics Division, National Center for Theoretical Sciences, Taipei, Taiwan.}
\email{hauwu@math.duke.edu}

\author{Nan~WU}
\address{Nan Wu\\
Department of Mathematics\\
Duke University}
\email{nan.wu@duke.edu}

\begin{document}


\maketitle

\begin{abstract}
{{When analyzing modern machine learning algorithms, we may need to handle kernel density estimation (KDE) with intricate kernels that are not designed by the user and might even be irregular and asymmetric. To handle this emerging challenge,} we provide a strong uniform consistency result with the $L^\infty$ convergence rate for KDE on Riemannian manifolds with Riemann integrable kernels (in the ambient Euclidean space). We also provide an {$L^1$} consistency result for kernel density estimation on Riemannian manifolds with Lebesgue integrable kernels. 
The {isotropic} kernels considered in this paper are different from the kernels in the Vapnik-Chervonenkis class that are frequently considered in statistics society. We illustrate the difference when we apply them to estimate {the} probability density function. 
{Moreover, we elaborate the delicate difference when the kernel is designed on the intrinsic manifold and on the ambient Euclidian space, both might be encountered in practice.} At last, we prove the necessary and sufficient condition for an {isotropic} kernel to be Riemann integrable on a submanifold in the Euclidean space.}
\end{abstract}

\section{Introduction}

Motivated by analyzing modern machine learning algorithms, we provide a strong uniform consistency result with the convergence rate for kernel density estimation (KDE) with general kernels on Riemannian manifolds.
Our main results can be summarized as follows. Let $M$ be a smooth $d$ dimensional compact manifold without boundary isometrically embedded in $\mathbb{R}^p$ through $\iota$. Let $x_1, \cdots, x_n$ be a sequence of points independently and identically (i.i.d) sampled from $M$ following the probability density function $P$ on $M$. 
\begin{enumerate}
\item
Suppose $K(t):\mathbb{R}_{\geq 0} \rightarrow \mathbb{R}$ (not necessary non-negative) is a bounded {\em Riemann integrable} kernel function with a proper decay rate.  Suppose $\int_{ \mathbb{R}^d} K(\|\textbf{v}\|_{\mathbb{R}^d})d\textbf{v}=1$. Consider the kernel density estimator at $x\in M$ as
\begin{align}\label{Definition Kn isotropic}
K_n(x)=\frac{1}{n \epsilon_n^d}\sum_{i=1}^n K\left(\frac{\|\iota(x_i)-\iota(x)\|_{\mathbb{R}^p}}{\epsilon_n}\right)\,.
\end{align}
Under a suitable relationship between $\epsilon_n$ and $n$, we show that $\sup_{x \in M}|K_n(x)-\mathbb{E} K_n(x)| \rightarrow 0$ a.s. and we will provide the convergence rate. If we further assume that the probability density function $P$ is H\"older continuous on $M$, then we have $\sup_{x \in M}|\mathbb{E} K_n(x)-P(x)| \rightarrow 0$ with a convergence rate.  {See Chapter \ref{Section:KDE2}. This setup is a direct generalization of most traditional setups by considering the manifold and possibly irregular function as an isotropic kernel.}

\item{
Suppose $K(\textbf{v}):\mathbb{R}^d \rightarrow \mathbb{R}$ (not necessary non-negative) is a bounded {\em Riemann integrable} kernel function with compact support.  Suppose $\int_{ \mathbb{R}^d} K(\textbf{v})d\textbf{v}=1$. For any $x$, we have a {diffeomorphism (or chart)} $\Phi_x$ of $\iota(M)$ such that $\Phi_x(\textbf{v}): B^{\mathbb{R}^d}_r(0) \rightarrow \iota(M)$ with $\Phi_x(0)=\iota(x)$, where $B^{\mathbb{R}^d}_r(0)$ is the open Euclidean ball of radius $r$ centered at $0$. Let $U_x(0)$ be the volume form of $\iota(M)$ in the chart $\Phi_x$ at $\textbf{v}=0$. For $x \in M$, we define the kernel density estimator at $x$ to be
\begin{align}
K_{n}(x)=\frac{1}{n \epsilon_n^d U_x(0)}\sum_{i=1}^n K_{\epsilon_n}(x,x_i)\,,
\end{align}
where  $K_{\epsilon_n}(x,x_i)=K(\frac{\Phi^{-1}_x(\iota(x_i))}{\epsilon_n})$ if $\iota(x_i) \in \Phi_x(B^{\mathbb{R}^d}_r(0))$ and $K_{\epsilon_n}(x,x_i)=0$ otherwise. We assume that for all $x$, we have a uniform control of the bi-Lipschitz constant of $\Phi_x$ and the $C^1$ norm of the volume form of $\iota(M)$ in the chart $\Phi_x$ at $\textbf{v}=0$. We further assume that the probability density function $P$ is H\"older continuous on $M$. Under a suitable relationship between $\epsilon_n$ and $n$, we show that $\sup_{x \in M}|\mathbb{E} K_n(x)-P(x)| \rightarrow 0$ with a convergence rate.  
}{See Chapter \ref{Section:KDE3}. This setup is less considered in the traditional setup, but an emerging one inspired by analyzing modern manifold learning algorithms. Here, the kernel $K_{\epsilon}(x,y)$ is in general not translational invariant and irregular.}

\item Suppose $K_{\epsilon}: \iota(M) \times \iota(M) \rightarrow \mathbb{R}_{\geq 0}$ is a sequence of bounded {\em Lebesgue integrable} kernel with compact support indexed by $\epsilon>0$.  Moreover, $\int_{M} K_{\epsilon}(\iota(x),\iota(y))dV(y)=1$, where $dV$ is the volume form on $M$. Let
\begin{align}
K_{n,\epsilon}(x)=\frac{1}{n}\sum_{i=1}^n K_{\epsilon}(\iota(x),\iota(x_i))\,.
\end{align}
Under a suitable relationship between $\epsilon$ and $n$, we show that $K_{n,\epsilon}(x) \rightarrow P(x)$ a.s. in the $L^1$ sense under different conditions on $P$. {See Chapter \ref{Section:KDE4}. This is the most general setup, where the kernel is in general not translational invariant and only Lebesgue integrable, and only $L^1$ convergence is shown.}
\end{enumerate}

KDE on manifolds shows up{, sometimes obviously and sometimes occultly, in manifold learning algorithms. Take the} nonlinear dimension reduction (NDR) {as an example}. 
%
The purpose of an NDR algorithm is to reconstruct the dataset on a manifold of a high dimensional space in some low dimensional space, while {preserving} the geometric or topological structure of the dataset.  However, the {non-uniform} distribution of the dataset on the manifold is one of the major obstacles to the exploration of the geometric or topological structure of the dataset. Hence, {it is intuitive that} KDE plays an essential role to eliminate the impact of the {non-uniform} distribution of the dataset when one develops an NDR algorithm. 

{There are two main ways KDE might appear in a manifold learning algorithm.
The first one comes from a direct design aiming to eliminating the impact of non-uniform distribution.}
For instance, the $\alpha=1$ renormalization in the diffusion map \cite{coifman2006diffusion} is the step to estimate the density of the dataset based on the kernel chosen by the user. {For KDE of this kind of appearance, since the manifold is usually unknown, the kernel is usually designed in the ambient Euclidean space, and we need to handle the intricate interaction between the intrinsic and extrinsic geometries when we analyze its behavior. Usually, researchers consider isotopic (that is, the kernel is defined as $K(x,y):=f(\|x-y\|)$ for a ``proper'' function $f$ defined on $\mathbb{R}_{\geq 0}$ for all $x,y\in \iota(M)$), or sometimes anisotropic (that is, the kernel is defined as $K(x,y):=k(x-y)$ for a ``proper'' function $k$ defined on $\mathbb{R}^p$ for all $x,y\in \iota(M)$), kernels that are symmetric with some nice properties, but rarely consider consider non-translational invariant (that is, the kernel $K$ is defined on $\mathbb{R}^p\times \mathbb{R}^p$ and cannot be defined on $x-y$ for all $x,y\in \mathbb{R}^p$) kernels with undesired properties like irregularity. We will discuss the isotropic case in Section \ref{Section:KDE2}.}

{The second one is not obvious from directly reading the algorithms, and the involvement of KDE was probably not the focus when designing the algorithm. However, it appears naturally when we analyze these algorithms under the manifold setup. The associated kernel of this kind is not designed by the user but appears intrinsically when the algorithm is applied, is usually defined on the manifold, and is probably not translational invariant, dependent on $n$ and other quantities, and irregular}. The widely applied locally linear embedding (LLE) algorithm \cite{Roweis_Saul:2003} is {a good} example. It is shown in \cite[Corollary 3.1]{wu2018think} that under the manifold setup, the LLE kernel, {and hence KDE,} is adaptive to the geometry of the dataset and depends on the regularization chosen in the barycentric coordinate evaluation. Therefore, the kernel is {non-translational invariant} and irregular. {Moreover,} the kernel's behavior near the boundary of a manifold is dramatically different from that away from the boundary \cite{wu2018locally}. As shown in \cite{wu2018think} and \cite{wu2018locally}, KDE is a crucial step towards the asymptotic analysis of LLE. {Getting a deeper theoretical justification of LLE relies on studying the spectral behavior of the integral operator corresponding to the LLE kernel, where a $L^\infty$ sense convergence of KDE with a general kernel is involved. {We will discuss this case in Sections \ref{Section:KDE3} and \ref{Section:KDE4}.}

Therefore, although the kernels considered in this work might seem not practical from the traditional perspective, it emerges naturally when we analyze the modern machine learning algorithms, which evolves fast in the past decades. While nowadays LLE might be considered a ``traditional'' NDR algorithm, its non-obvious involvement of KDE and intricate kernel structure suggests the possibility that we might encounter more challenging kernels when we analyze ``modern'' machine learning algorithms.}

Our results could be viewed as a generalization of various existing results.
The traditional KDE on the Euclidean space was first considered by M. Rosenblatt \cite{rosenblatt1956remarks} and E. Parzen \cite{parzen1962estimation}.  Let $P$ be a probability density function on $\mathbb{R}^d$. Let $x_1, \cdots, x_n$ be a sequence of i.i.d samples from $\mathbb{R}^d$ based on $P$. Let $\epsilon_n$ be a sequence of numbers such that $\epsilon:=\epsilon_n \rightarrow 0$ and $n\epsilon^d_n \rightarrow \infty$ as $n \rightarrow \infty$. 
For $x \in \mathbb{R}^d$, the kernel density estimator at $x$ is defined as
\begin{align}
K_{n}(x):=\frac{1}{n \epsilon^d}\sum_{i=1}^n K\left(\frac{x_i-x}{\epsilon}\right)\,, 
\end{align}
where $K(x) \geq 0$ is a bounded real valued function on $\mathbb{R}^d$. After \cite{rosenblatt1956remarks,parzen1962estimation}, a lot of results under various conditions were proposed. We summarize those results that are directly related to our work.

\begin{enumerate}
\item In \cite{devroye1979l1}, the authors consider the KDE on the Euclidean space. It is shown that if $P$ is bounded and $K$ is non-negative and its Lebesgue integration is $1$, then $K_n(x) \rightarrow P(x)$ in $L_1(\mathbb{R}^d)$ a.s. when $n \rightarrow \infty$. It is also shown that if there is no condition on $P$ and $K$ further satisfies certain decay rate, then $K_n(x) \rightarrow P(x)$ in $L_1(\mathbb{R}^d)$ a.s. when $n \rightarrow \infty$.
In \cite{devroye1980strong}, it is shown that if $P$ is uniformly continuous and $K$ is non-negative with certain decay rate and Riemann integration $1$, then $K_n(x) \rightarrow P(x)$ for all $x$, a.s. when $n \rightarrow \infty$. {Density estimation from the Bayesian approach could be found in, for example, \cite{wong2010optional, ma2017adaptive} and many others. }

\item In \cite{gine2002rates} and \cite{einmahl2005uniform}, the authors consider the KDE in the Euclidean space. The authors provide the  {$L^\infty$ convergence rate} from $K_n(x)$ to $\mathbb{E}K_n(x)$ for bounded $P$ and $K$ in the VC class. Compared with the convergence of $K_n(x)$ to $\mathbb{E}K_n(x)$, the regularity of $P$ is necessary for the convergence rate from $\mathbb{E}K_n(x)$ to $P(x)$. The result is applied to the density clustering problem in \cite{rinaldo2010generalized}. The proof of the convergence of $K_n(x)$ to $\mathbb{E}K_n(x)$ relies on the {{\em Rademacher process}}. In this process, it is crucial that the $\epsilon$-covering number of the kernel in the $L^2$ norm is of order $\frac{1}{\epsilon^v}$ for some $v>0$ so that the log of the covering number is an integrable function of $\epsilon$. We mention that the VC class is {\em in some sense} stronger than Riemann integrable in that there exists a bounded  kernel with compact support which is continuous except at one point but not in the VC class. We will discuss this difference more extensively below.

\item There have been several KDE results on closed Riemannian manifold (i.e., compact {\em without} boundary).  In \cite{pelletier2005kernel}, the author works on a closed Riemannian manifold and assumes that $P$ is $C^2$ and $K$ is a non-negative Lebesgue integrable with compact support. Through the Taylor expansion of $P$, the author provides the convergence rate of $\mathbb{E}_P\|K_n-P\|_{L^2(M)}$. In this work, the author uses the intrinsic information, the {\em geodesic distance}, rather than the extrinsic information, the ambient Euclidean distance, to construct the kernel. However, in practice, the geodesic distance information is usually not accessible to researchers. Similarly, the authors in \cite{kim2013geometric} also construct a kernel density estimator by using the intrinsic information of the manifold. Through the exponential map, they pull back the sample points on the manifold to the tangent space and apply KDE for the tangent space (Euclidean space). When the kernel function is regular and the probability density function is $C^4$, they provide an $L^2$ convergence rate result. In contrast, the authors in \cite{ozakin2009submanifold} only use the ambient Euclidean distance to construct a density estimator on the manifold and prove an $L^2$ convergence rate for the {non-negative} $C^1$ kernel functions and $C^2$ probability density function.   In \cite{berenfeld2019density}, the authors study the $L^p$ convergence rate for smooth kernel functions with compact support. The convergence rate  is estimated in terms of the H\"older's exponent of the probability density function and the regularity of the manifold (the H\"older's exponents of the derivatives of the exponential map). In the case when the manifold is compact {\em with} boundary, under a similar assumption as in \cite{pelletier2005kernel}, the authors in \cite{berry2017density} prove an $L^2$ convergence result.   In \cite{chen2019generalized}, the author studies the singular measures on the manifold. The author proves that the KDE is not uniformly consistent but pointwise consistent after rescaling.

\item We mention that the asymptotic convergence of KDE in the {\em pointwise} sense under the manifold setup is widely considered {\em implicitly} in several asymptotic analyses of manifold learning algorithms. In addition to those in the LLE analysis mentioned above \cite{wu2018think,wu2018locally},  see, for example, \cite{belkin2007convergence} and many others. To our knowledge, in most work, since the KDE is not the focus, usually the probability density function and the kernel, if needed, are assumed to behave nicely.  
\end{enumerate}

To our knowledge, our work is the first detailed analysis of the $L^\infty$ convergence rate of KDE with such general conditions on the kernel and the probability density function under the manifold setup. A related but different result under the different kernel assumption is shown in \cite{rinaldo2010generalized}. In that paper, when the dataset is sampled from the manifold, the authors directly generalize [7, Corollary 2.2] to estimate the {\em mollified density function} {that is defined in the ambient space}, where the kernel is in the VC class. The uniform convergence in \cite{rinaldo2010generalized} is in the ambient Euclidean space but not intrinsic to the manifold. As a result, the convergence in \cite{rinaldo2010generalized} might suffer when the ambient space dimension is high. Moreover, it is not clear how to directly ``deconvolve'' the estimated mollified density function to recover the density function on the manifold. Thus, the uniform convergence result under the manifold setup shown in \cite{rinaldo2010generalized} is not comparable with ours. On the other hand, we found that the technical challenge to produce the convergence rate of the KDE on manifolds with kernels classified by their integrability is different from that with kernels in the VC class, and the developed analysis scheme could be useful in other scenarios.

Through the whole paper, we consider the following manifold model. Let $M$ be a smooth $d$ dimensional compact manifold without boundary isometrically embedded in $\mathbb{R}^p$ through $\iota$. Let $P$ be a probability density function on $M$. Let $x_1, \cdots, x_n$ be i.i.d sampled from $M$ based on $P$. We will discuss the KDE with {three} types of kernels: isotropic Riemann integrable kernels, {non-translational invariant Riemann integrable kernels,} and Lebesgue integrable kernels. 
We impose different conditions on the probability density function $P$ and study the convergence behavior of the kernel density estimator.
We use $\textbf{u}=(u_1, \cdots, u_p)$ to denote a vector in $\mathbb{R}^p$ and $\textbf{v}=(v_1, \cdots , v_d)$ to denote a vector in $\mathbb{R}^d$.

\section{Kernel density estimation with isotropic Riemann integrable kernels}\label{Section:KDE2}

We start from stating our assumptions on the kernel.

\begin{Assumption}\label{assumptions on kernel 1}
Suppose $K(t):\mathbb{R}_{\geq 0} \rightarrow \mathbb{R}$ is the kernel function satisfying the following conditions:
\begin{enumerate}[(i)]
\item
$K$ is a bounded function on $\mathbb{R}_{\geq 0}$; that is, $\sup_{t \in \mathbb{R}_{\geq 0}} |K(t)| = K_{\sup}$ for some $K_{\sup}>0$.
\item
$K$ is Riemann integrable on any compact subset $[0,a]$, where $a>0$.
\item
There exists a $\rho>0$, such that if $t \geq \rho$, then $|K(t)| \leq \frac{1}{t^\alpha}$, where $\alpha>d$.
\item
$\int_{ \mathbb{R}^d} K(\|\textbf{v}\|_{\mathbb{R}^d})d\textbf{v}=1$.
\end{enumerate}
\end{Assumption} 

We have a few comments about the above assumptions. First, we do not need the kernel to be non-negative. 
Second, since $K(\|\textbf{u}\|_{\mathbb{R}^p})$ can be regraded as an {\em isotropic} kernel function on the ambient space $\mathbb{R}^p$, we still call the kernel an isotropic one for the density estimation on $M$ despite the fact that it {might not be} isotropic with respect to the intrinsic geometry of the manifold. {The relationship between the intrinsic and extrinsic geometries will be further elaborated in Section \ref{Relationship between intrinsic and extrinsic}.}
Third, note that if $K(t):\mathbb{R}_{\geq 0} \rightarrow \mathbb{R}$ is a bounded function that is Riemann integrable on the compact subsets  $[0,a]$ for any $a$ and $f: A \subset \mathbb{R}^q \rightarrow \mathbb{R}_{\geq 0}$ is a continuous function on a closed rectangle $A$, then it is in general not true that $K \circ f$ is  Riemann integrable on $A$. However, if $f$ is the Euclidean distance function; that is, $f(x)=\|x\|$, we know that $K(\|\textbf{v}\|_{\mathbb{R}^q})$ is Riemann integrable on any closed rectangle $A$ in $\mathbb{R}^q$.
Hence, $\int_{ \mathbb{R}^d} K(\|\textbf{v}\|_{\mathbb{R}^d})d\textbf{v}=1$ in the assumptions can be either understood as an improper Riemann integral or as a Lebesgue integral.

Let $\epsilon_n$ be a sequence of numbers such that $\epsilon_n \rightarrow 0$, as $n \rightarrow \infty$. For $x \in M$, consider the kernel density estimator at $x$ defined in \eqref{Definition Kn isotropic}.
Note that
\begin{align}
\mathbb{E}K_n(x)=\int_{M} \frac{1}{\epsilon_n^d}K\left(\frac{\|\iota(y)-\iota(x)\|_{\mathbb{R}^p}}{\epsilon_n}\right) P(y) dV(y),
\end{align}
where $dV$ is the volume form on $M$. 
We need to emphasize that the integration in the above equation should be understood as {\em Lebesgue integral} on $M$ even though we assume that $K(t)$ is Riemann integrable. In fact, we will see in the later section that under the manifold setup we consider here, for an arbitrary Riemann integrable kernel $K(t)$, $K\left(\frac{\|\iota(y)-\iota(x)\|_{\mathbb{R}^p}}{\epsilon_n}\right)$ {\em may not} be a Riemann integrable function on $M$. 

{In this section,} we expect to show that when $n \rightarrow \infty$, 
\begin{align}
K_n(x) \rightarrow P(x)
\end{align}
a.s. for {\em all} $x$ on $M$ under some convergence rate; that is, $L^\infty$ convergence.
We split this result in two steps, including the variance analysis and the bias analysis. In the bias analysis, under the regularity assumption of $P$, we control $\|\mathbb{E}K_n(x)-P(x)\|_{L^\infty(M)}$ in the deterministic way.
In the variance analysis, we control the variance of $\|\mathbb{E}K_n(x)-K_n(x)\|_{L^\infty(M)}$ in the probabilistic sense, where the regularity assumption of $P$ is not needed. 

\subsection{\textnormal{\textbf{Variance analysis.}}}

We introduce the {\em partition number} for the kernel $K(t)$ satisfying Assumption \ref{assumptions on kernel 1}. 
The partition number plays an essential role in calculating the convergence rate. 

\begin{definition}(Partition number)\label{Definition:Partition number}
Consider the kernel function $K(t)$ satisfying Assumption \ref{assumptions on kernel 1}. Let $D_{lip}$ be a constant that depends only on the $d$-dimensional smooth and compact manifold without boundary embedded in $\mathbb{R}^p$. See \eqref{map J_x} for a precise definition. For any $\gamma>0$, the partition number $N(\gamma)$ is the smallest integer so that the following condition is satisfied. If we partition $[-D_{lip}\gamma^{-\frac{1}{\alpha}}, D_{lip}\gamma^{-\frac{1}{\alpha}}]^d$ uniformly into $N(\gamma)$  cubes $\{\mathcal{Q}_i\}_{i=1}^{N(\gamma)}$, where each $\mathcal{Q}_i$ is of the form $[v_1,v_1+a] \times \cdots \times [v_d,v_d+a]$ for $a>0$, we have 
\begin{align}\label{equation def partition number}
\sum_{i =1}^{N(\gamma)}[M_i (K)-m_i (K)]Vol(\mathcal{Q}_i) <\gamma^2.
\end{align} 
Here, $M_i(K)$ and $m_i(K)$ are  the supremum and infimum of  $K(\|\textbf{v}\|_{\mathbb{R}^d})$ respectively over the cube $\mathcal{Q}_i$. 
\end{definition}

Recall the above-mentioned fact that $K(\|\textbf{v}\|_{\mathbb{R}^d})$ is Riemann integrable over $[-D_{lip}\gamma^{-\frac{1}{\alpha}}, D_{lip}\gamma^{-\frac{1}{\alpha}}]^d$. Hence, the existence of $N(\gamma)$ follows from the Riemann integrability.


\begin{theorem} \label{EK_n-K_n} 
Under Assumptions \ref{assumptions on kernel 1}, we further assume $\|P\|_\infty = P_{Max}$ for $P_{Max}>0$. Suppose $0<\gamma<\min\{\rho^{-\alpha},1\}$ and $\epsilon_n \rightarrow 0$ as $n \rightarrow \infty$. When $\epsilon_n \leq \mathcal{D}_3 \gamma^{\frac{1}{\alpha-d}}$, we have 
\begin{align}\label{Bound variance most general isotropic one}
P\{\|\mathbb{E}K_n-K_n\|_\infty \leq \mathcal{D}_1 \gamma^{1-\frac{d}{\alpha}}\} \geq 1-8(2n)^{2p^2+2p}\exp\left\{-\mathcal{D}_2 n \epsilon_n^{d+\frac{d^2}{\alpha}}\frac{\gamma^{2-\frac{d}{\alpha}}}{N(\gamma)^2}\right\}\,,
\end{align}
where $\mathcal{D}_1$ and $\mathcal{D}_2$ depend on $p$, $d$, $\alpha$, $P_{Max}$, $K_{sup}$, and the second fundamental form of $\iota(M)$, and $\mathcal{D}_3$ depends on $p$, $d$, $\alpha$, $P_{Max}$ and the second fundamental form of $\iota(M)$. 
\end{theorem}

The definitions of $\mathcal{D}_1$, $\mathcal{D}_2$ and $\mathcal{D}_3$ can be found in the proof, which is postponed to Appendix \ref{Appendix A}. The proof involves the methods in \cite{devroye1980strong} developed for the convergence analysis in the Euclidean space setup{, while we need to handle the interaction between the intrinsic and extrinsic geometry}. {We comment that the dependence on the ambient space dimension, that is, $(2n)^{2p^2+2p}$ in the right hand side of \eqref{Bound variance most general isotropic one}, comes from using rectangles in the ambient space to construct a collection of regular sets on the manifold. }

\begin{remark} To appreciate the challenge of generalizing this result to Lebesgue integrable kernels, for simplicity, we assume that $M$ is a rectangle $A \subset \mathbb{R}^d$. Since the function $K(\|\cdot\|_{\mathbb{R}^d})$ is Riemann integrable on $A$, we can approximate $K(\|\cdot\|_{\mathbb{R}^d})$ by a step function $K^*(\cdot)$ uniformly, except on a bad set $A_{\epsilon}$ which can be covered by {\em finitely} many rectangles with the total volume less than $\epsilon>0$. Hence, instead of controlling $|\mathbb{E}K_n(x)-K_n(x)|$ for an arbitrary kernel function, we study the variance of the $0-1$ kernel over rectangles. In contrast, if the function $K(\|\cdot\|_{\mathbb{R}^d})$ is Lebesgue measurable on $A$, we can still approximate $K(\|\cdot\|_{\mathbb{R}^d})$ by a step function $K^*(\cdot)$ uniformly, except on a bad set $B_{\epsilon}$. Here, $B_{\epsilon}$ can be covered by {\em countably}, but may not finitely, many rectangles with the total volume less than $\epsilon$. Therefore, to control the variance $|\mathbb{E}K_n(x)-K_n(x)|$, we have to deal with a bad set $B_{\epsilon}$ associated with each $x$. The major difficulty in generalizing the argument from Riemann integrable kernels to Lebesgue integrable kernels is that if $A_{\epsilon}$ cannot be covered by finite rectangles, then the variance over $B_{\epsilon}$ can not be controlled uniformly for all $x$.  The reader may refer to the last two steps in \eqref{Riemann approx step 3} in the Appendix for details.
\end{remark}

In the following corollary, we state a special case of Theorem \ref{EK_n-K_n} when the kernel $K(t)$ has a compact support. We see that the convergence rate is improved from the general case.

\begin{corollary}\label{compact support variance}
Under Assumption \ref{assumptions on kernel 1}, we further assume that $K(t)$ has a compact support on $[0, \rho]$ and $\|P\|_\infty = P_{Max}$ for $P_{Max}>0$. Suppose $0<\gamma<1$ and $\epsilon_n \rightarrow 0$ as $n \rightarrow \infty$. When $\epsilon_n \leq \mathcal{D}_3$, we have
\begin{align}
P\left\{\sup_{x \in M} |\mathbb{E}K_n(x)-K_n(x)| \leq \mathcal{D}_1 \gamma \right\} \geq 1-8(2n)^{2p^2+2p}\exp\left\{-\mathcal{D}_2n \epsilon_n^{d}\frac{ \gamma^2}{N(\gamma)^2}\right\}\,,
\end{align}
where $\mathcal{D}_1$ and $\mathcal{D}_2$ depend on $p$, $d$, $P_{Max}$, $K_{sup}$, and the second fundamental form of $\iota(M)$, and $\mathcal{D}_3$ depends on $\rho$, $P_{Max}$ and the second fundamental form of $\iota(M)$. 
\end{corollary}

The proof of the corollary follows directly by choosing $\gamma<1$ and taking $\alpha \rightarrow \infty$ in Theorem \ref{EK_n-K_n}.  
We explain why the convergence rate is slower when $K(t)$ does not have compact support. As we show in Lemma \ref{rectangle covering lemma}, we have a good control of the variance over any region centered at $x$ that is not too large on the manifold $M$, and this control is uniform. To control the variance outside the region, we have to make sure that the {\em tail} is small enough; that is, for $\eta>0$, $\int_{M \setminus B_\eta(x)} \frac{1}{\epsilon_n^d}K\left(\frac{\|\iota(y)-\iota(x)\|_{\mathbb{R}^p}}{\epsilon_n}\right) P(y)dV(y) $ is small enough when $\epsilon_n$ is small, where $B_\eta(x)$ is a geodesic ball of radius $\eta$ at $x \in M$. When $M$ is an Euclidean space, due to the polynomial decay assumption, the larger $\|y-x\|$ is, the smaller the kernel value $K(\|y-x\|)$ is. So, for $\eta>0$, we can choose $\epsilon_n>0$ sufficiently small so that $\int_{\mathbb{R}^d \setminus B^{\mathbb{R}^d}_{\eta}(x)} \frac{1}{\epsilon_n^d}|K(\frac{\|y-x\|_{\mathbb{R}^d}}{\epsilon_n})| P(y) dy$ is small. 
However, in general, due to the geometry of the manifold, it is not guaranteed that the larger the geodesic distance between $y$ and $x$ is, the smaller the kernel value $K(\|\iota(y)-\iota(x)\|)$. Indeed, two far away points in the sense of geodesic distance might have short Euclidean distance in the ambient space. 
Thus, the only way to make sure that $\int_{M \setminus B_\eta(x)} \frac{1}{\epsilon_n^d}K\left(\frac{\|\iota(y)-\iota(x)\|_{\mathbb{R}^p}}{\epsilon_n}\right) P(u) dV(y) $ is small is to choose $\epsilon_n$ {\em further smaller} so that $ \frac{1}{\epsilon_n^d}K\left(\frac{\|\iota(y)-\iota(x)\|_{\mathbb{R}^p}}{\epsilon_n}\right) $ is small outside $B_\eta(x)$. In other words, the slower convergence rate in the non compact kernel case is a result of the geometry of the manifold.

The following Corollary is a special case when the kernel $K(t)$ is a step function. The proof is in  Appendix \ref{Appendix A}.

\begin{corollary}\label{0-1 kernel}
Suppose $K(t)=\sum_{j=1}^J c_j \bigchi_{[a_j, b_j]}(t)$, where $c_j \leq K_{\sup}$ and $0 \leq a_j \leq b_j \leq \rho$. Suppose $\|P\|_\infty = P_{Max}$ for $P_{Max}>0$. Let $0<\gamma<1$ and $\epsilon_n \rightarrow 0$ as $n \rightarrow \infty$.  If $\epsilon_n \leq \mathcal{D}_3$, then
\begin{align}
P\left\{\sup_{x \in M} |\mathbb{E}K_n(x)-K_n(x)| \leq \mathcal{D}_1  \gamma\right\} \geq 1-8(2n)^{p+2}\exp\left\{-\mathcal{D}_2  n \epsilon_n^d \gamma^2\right\}\,,
\end{align}
where $\mathcal{D}_1$ depends on $d$, $P_{Max}$, $J$, $K_{\sup}$ and the second fundamental form of $\iota(M)$,  and $\mathcal{D}_2$ depends on $d$, $\rho$, $P_{Max}$ and the second fundamental form of $\iota(M)$ and $\mathcal{D}_3$ depends on $\rho$, $P_{Max}$ and the second fundamental form of $\iota(M)$. 

Hence, if $\epsilon_n \rightarrow 0$ as $n \rightarrow \infty$ and $\epsilon_n \leq \mathcal{D}_3$, then with probability greater then $1-\frac{1}{n^2}$, we have 
\begin{align}
\sup_{x \in M} |\mathbb{E}K_n(x)-K_n(x)| \leq \mathcal{D}_4 \sqrt{\frac{\log n}{n \epsilon^d}},
\end{align}
where  $\mathcal{D}_4$ depends on $p$, $d$, $P_{Max}$, $\rho$, $J$, $K_{\sup}$ and the second fundamental form of $\iota(M)$
\end{corollary}

Since $K(t)$ is a step function, we do not need to approximate it anymore. Hence, in contrast to Corollary \ref{compact support variance}, there is no partition number and step function approximation involved and the convergence rate can be further improved.

\subsection{\textnormal{\textbf{Bias analysis}}}

To get the bias analysis with a convergence rate, we need to further assume the H\"older continuity condition on the probability density function $P$. In fact, if we only need to show that $\|\mathbb{E}K_n(x)-P(x)\|_{L^\infty(M)} \rightarrow 0$ as $\epsilon_n \rightarrow 0$, then it is sufficient to assume $P$ is continuous. However, if we want to control the convergence rate of  $\|\mathbb{E}K_n(x)-P(x)\|_{L^\infty(M)}$, we need to know how $P(x) \rightarrow P(y)$ when $x \rightarrow y$ for $x, y \in M$. 

\begin{Assumption}\label{holder assumption of P}
Suppose the density function $P$ satisfies $\|P\|_\infty = P_{Max}$ for some $P_{Max}>0$ on $M$.  Moreover, $P$ is H\"older continuous so that 
\begin{align}
|P(x)-P(y)| \leq C_P d(x,y)^\kappa\,,
\end{align} 
where $d(x,y)$ is $x,y$ is the geodesic distance between $x$ and $y$ on $M$, $0<\kappa \leq 1$ and $C_P>0$. 
\end{Assumption}

The main theorem in the bias analysis is discussed when $K(t)$ is compactly supported or not. The proof of the theorem is in Appendix \ref{Appendix B}.

\begin{theorem}\label{bias analysis main theorem}
\begin{enumerate}
\item Under Assumptions \ref{assumptions on kernel 1} and \ref{holder assumption of P}, when $0<\gamma<1$ and $\epsilon \leq \omega_1 \gamma^{\frac{2\alpha}{(\alpha-d)}}$, we have
\begin{align}
\sup_{x \in M}\left|\int_{M}\frac{1}{\epsilon^d} K\left(\frac{\|\iota(y)-\iota(x)\|_{\mathbb{R}^p}}{\epsilon}\right)P(y)dV(y)-P(x)\right|\leq \omega_2 \gamma^\kappa\,,
\end{align}
where $\omega_1$ depends $\rho$, $\alpha$ and $d$ and $\omega_2$ depends on $\rho$, $\alpha$, $d$, $\kappa$, $C_P$, $P_{Max}$, $K_{sup}$, the curvature of $M$ and the second fundamental form of $\iota(M)$.

\item Under Assumptions \ref{assumptions on kernel 1} and \ref{holder assumption of P}, we further assume that $K(t)$ is compactly support on $[0, \rho]$, where $\rho>0$. When $0<\gamma<1$ and $\epsilon \leq \omega_1 \gamma^{2}$, we have
\begin{align}
\sup_{x \in M}\left|\int_{M}\frac{1}{\epsilon^d} K\left(\frac{\|\iota(y)-\iota(x)\|_{\mathbb{R}^p}}{\epsilon}\right)P(y)dV(y)-P(x)\right|\leq \omega_2 \gamma^\kappa\,,
\end{align}
where $\omega_1$   depends on $\rho$ and $\omega_2$  depends on $\rho$, $d$, $\kappa$, $C_P$, $P_{Max}$, $K_{sup}$, the curvature of $M$ and the second fundamental form of $\iota(M)$.
\end{enumerate}
\end{theorem}

\subsection{\textnormal{\textbf{Put everything together}}}

We concatenate the variance and the bias analysis to derive the main result. 
In the following theorem, we discuss three cases corresponding to those three cases in the variance analysis. 
The proof is in Appendix \ref{Appendix B}.

\begin{theorem}\label{density estimation}
\begin{enumerate}
\item Under Assumptions \ref{assumptions on kernel 1} and \ref{holder assumption of P}, assume $0<\gamma<\min\{\rho^{-\alpha},1\}$ and $\epsilon_n \rightarrow 0$ as $n \rightarrow \infty$. When $\epsilon_n \leq \Omega_3 \gamma^{\frac{2\alpha}{\alpha-d}}$, we have 
\begin{align}
&P\left\{\sup_{x \in M}|K_n(x)-P(x)| \leq \Omega_1 (\gamma^{1-\frac{d}{\alpha}}+\gamma^\kappa)\right\} \nonumber \\
\geq \,&1-8(2n)^{2p^2+2p}\exp\left\{-\Omega_2 n \epsilon_n^{d+\frac{d^2}{\alpha}}\left(\frac{\gamma^{2-\frac{d}{\alpha}}}{N(\gamma)^2}\right)\right\}\,,
\end{align}
where $\Omega_1$ depends on $p$, $\rho$, $d$, $\alpha$, $\kappa$, $C_P$, $P_{Max}$, $K_{sup}$, the curvature of $M$ and the second fundamental form of $\iota(M)$. $\Omega_2$  depends on $p$, $d$, $\alpha$, $P_{Max}$, $K_{sup}$ and the second fundamental form of $\iota(M)$. $\Omega_3$  depends on $p$, $\rho$, $\alpha$, $d$, $P_{Max}$ and the second fundamental form of $\iota(M)$. 

\item Under Assumptions \ref{assumptions on kernel 1} and \ref{holder assumption of P}, we further assume that $K(t)$ is compactly supported on $[0, \rho]$. Suppose $0<\gamma<1$ and $\epsilon_n \rightarrow 0$ as $n \rightarrow \infty$.  If $\epsilon_n \leq \Omega_3\gamma^{2}$, then
\begin{align}
P\left\{\sup_{x \in M}|K_n(x)-P(x)| \leq \Omega_1 \gamma^\kappa \right\} \geq 1-8(2n)^{2p^2+2p}\exp\left\{-\Omega_2\frac{n \epsilon_n^{d} \gamma^2}{N(\gamma)^2}\right\}\,,
\end{align}
where $\Omega_1$  depends on $p$, $\rho$, $d$, $\kappa$, $C_P$, $P_{Max}$, $K_{sup}$, the curvature of $M$ and the second fundamental form of $\iota(M)$. $\Omega_2$  depends on $p$, $d$, $P_{Max}$, $K_{sup}$, and the second fundamental form of $\iota(M)$. $\Omega_3$  depends on $\rho$,  $P_{Max}$ and the second fundamental form of $\iota(M)$. 
\item
Suppose $K(t)=\sum_{j=1}^J c_j \bigchi_{[a_j, b_j]}(t)$, where $c_j \leq K_{\sup}$ and $0 \leq a_j \leq b_j \leq \rho$. Moreover, $K(t)$ satiesfies (iv) in Assumption \ref{assumptions on kernel 1} . Suppose $0<\gamma<1$ and $\epsilon_n \rightarrow 0$, as $n \rightarrow \infty$. Under Assumption \ref{holder assumption of P}, if $\epsilon_n \leq \Omega_3\gamma^{2}$, then
\begin{align}
P\left\{\sup_{x \in M}\left|K_n(x)-P(x)\right| \leq \Omega_1 \gamma^\kappa \right\} \geq 1-8(2n)^{p+2}\exp\left\{-\Omega_2  n \epsilon_n^d \gamma^2\right\}\,,
\end{align}
where $\Omega_1$ depends on $\rho$, $d$, $J$, $\kappa$, $C_P$, $P_{Max}$, $K_{\sup}$, the curvature of $M$ and the second fundamental form of $\iota(M)$. $\Omega_2$ which depends on $d$, $\rho$, $P_{Max}$ and the second fundamental form of $\iota(M)$. $\Omega_3$ which depends on $\rho$,  $P_{Max}$ and the second fundamental form of $\iota(M)$. 

Hence, if $\epsilon_n \rightarrow 0$ as $n \rightarrow \infty$ and $\epsilon_n \leq \Omega_3$, then with probability greater then $1-\frac{1}{n^2}$, we have 
\begin{align}
\sup_{x \in M} |K_n(x)-P(x)| \leq \Omega_4 (\frac{\log n}{n \epsilon^d})^{\frac{\kappa}{2}},
\end{align}
where  $\Omega_4$ depends on $p$, $\rho$, $d$, $J$, $\kappa$, $C_P$, $P_{Max}$, $K_{\sup}$, the curvature of $M$ and the second fundamental form of $\iota(M)$.
\end{enumerate}
\end{theorem}

\subsection{Riemann integrability of an isotropic kernel on a manifold}\label{Relationship between intrinsic and extrinsic}

In this section, {we discuss the delicate difference when the kernel is defined on the ambient Euclidean space and on the manifold, particularly when the kernel is irregular. The interaction between intrinsic and extrinsic geometry plays a role.}
We show a necessary and sufficient condition for an isotropic kernel to be Riemann integrable on a manifold when it is Riemann integrable in the ambient space. We will provide an {explicit example of a kernel which} is Riemann integrable in the ambient space but not integrable on the manifold. {In the example}, although the kernel is not Riemann integrable on the manifold, by our main theorem, such kernel can still be used for density estimation on the manifold and the convergence rate can still be estimated. Our theorem in this section {suggests} that the results in the Euclidean space may not be easily generalized to the manifold case when the extrinsic properties of the manifold are involved.

For any function $f$ which is Riemann integrable on any closed ball in $\mathbb{R}^p$, it is not necessary that $f$ is a Riemann integrable function over any embedded submanifold in $\mathbb{R}^p$. A trivial example is as follows.

\begin{example}
$f(x,y)$ is defined on $\mathbb{R}^2$. $f(x,y)=1$ when $(x,y) \in \mathbb{Q}\cap (0,1) \times \{0\}$. And $f(x,y)=0$ otherwise. Then $f(x,y)$ is Riemann integrable over any  ball in $\mathbb{R}^p$. But, if $(0,1) \times \{0\}$ is the embedded submanifold, then $f(x,y)$ is not Riemann integrable over the submanifold.
\end{example}

For the theoretical purpose, one may ask that whether $K(\frac{\|\iota(y)-\iota(x)\|_{\mathbb{R}^p}}{\epsilon})$ is a Riemann integrable function on the manifold $M$. 
The next theorem provides the necessary and sufficient condition such that $K(\frac{\|\iota(y)-\iota(x)\|_{\mathbb{R}^p}}{\epsilon})$ is Riemann integrable on the manifold $M$ for any $\epsilon$. 

\begin{theorem}\label{Riemann integrable main theorem}
Suppose $M$ is a $d$-dimensional compact smooth manifold without boundary isometrically embedded in $\mathbb{R}^p$ through $\iota$. Fix $x \in M$. Let $D_x(\textbf{u})=\|\textbf{u}-\iota(x)\|_{\mathbb{R}^p}$. The set of critical points of $D_x(\textbf{u})$ on $\iota(M)$ is Jordan measurable if and only if for any $\epsilon>0$, $K(\frac{\|\iota(y)-\iota(x)\|_{\mathbb{R}^p}}{\epsilon})$ is a Riemann integrable function of $y$ on the manifold for all bounded kernel $K(t):\mathbb{R}_{\geq0}\to \mathbb{R}$ that is Riemann integrable on $[0,a]$ for all $a>0$. 
\end{theorem}

The proof of Theorem \ref{Riemann integrable main theorem} is in Appendix \ref{Appendix D}. When $M$ is an analytic manifold, the set of critical points of $D_x(\textbf{u})$ on $\iota(M)$ has measure $0$. Because the set of critical points of $D_x(\textbf{u})$ on $\iota(M)$ is a closed subset of $\iota(M)$ containing $\iota(x)$, it contains all its boundary points. Hence, the set of critical points of $D_x(\textbf{u})$ on $\iota(M)$ is Jordan measurable and we have the following corollary.

\begin{corollary}
Suppose $M$ is a $d$-dimensional compact {\em analytic} manifold without boundary isometrically embedded in $\mathbb{R}^p$ through $\iota$. Take a bounded kernel $K(t):\mathbb{R}_{\geq0}\to \mathbb{R}$ that is Riemann integrable on $[0,a]$ for all $a>0$.  Fix $x \in M$. For any $\epsilon>0$, $K(\frac{\|\iota(y)-\iota(x)\|_{\mathbb{R}^p}}{\epsilon})$ is a Riemann integrable function of $y$ on the manifold.
\end{corollary}

Next, we construct explicitly an embedded manifold $\iota(M)$ so that for some $\iota(x)$ on $\iota(M)$, the set of critical points of $D_x(\textbf{u})$ on $\iota(M)$ is not Jordan measurable. Moreover, we construct a function $K(t)$ satisfies Assumption \ref{assumptions on kernel 1}, but $K(\frac{\|\iota(y)-\iota(x)\|_{\mathbb{R}^p}}{\epsilon})$  is not Riemann integrable on the manifold for infinity many choices of $\epsilon$.

\begin{example} 
We construct a fat Cantor set $C$ in $[0, \frac{\pi}{2}]$ with non-zero measure. Then there is a non-negative smooth function $f(\theta)$ that only vanishes on $C$.  $f(\theta)$ can be constructed in the following way. The complement of $C$ in $[0, \frac{\pi}{2}]$ is the union of countable open intervals. Then on each of the open interval, $f(\theta)$ is equal to an everywhere positive bump function subject to  the open interval and vanishes on the boundary of the open interval. Moreover, $f(\theta)$ is $0$ on the fat cantor set. We construct  the curve $\gamma(\theta) \subset \mathbb{R}^2$ for $\theta \in  [0, 2\pi]$ with the following conditions. 
\begin{enumerate}
\item $\gamma(\theta)$ is smooth with $\gamma(\pi)=(0,0)$.
\item $\gamma(\theta)=((f(\theta)+1)cos(\theta), (f(\theta)+1)sin(\theta))$, for $\theta \in [0, \frac{\pi}{2}]$.  
\item $\gamma(\theta)$ for $\theta \in (\frac{\pi}{2}, 2\pi)$ is contained in the open unit disc centered at $(0,0)$.
\end{enumerate}
Note that the fat Cantor set $C$ is embedded into the closed curve as a fat Cantor set. In other words, let $C'=\gamma(C)$. Then, the unit circle in $\mathbb{R}^2$ is tangent to $\gamma(\theta)$ along the fat cantor set $C'$.  Hence, they are the critical points of the function $D_x(\textbf{u})=\|\textbf{u}\|_{\mathbb{R}^2}$. The other critical points on $\gamma(\theta)$ for $\theta \in [0, \frac{\pi}{2}]$ are the image of the critical points of $f(\theta)$ in the open intervals under $\gamma$. Thus, they are countable. Therefore, the set of critical points of $D_x(\textbf{u})=\|\textbf{u}\|_{\mathbb{R}^2}$ is not Jordan measurable.
\begin{align}
K(t)=\left\{
\begin{array}{lll} 
\frac{1}{3}& \mbox{if $t \in [0, 1)\cup (1,\frac{3}{2}]$};\\ \frac{1}{t^2} & \mbox{if $t$ is a positive integer};\\ 
0 & \mbox{everywhere else}. \end{array} \right.
\end{align}
Note that $K(t)$ satisfies Assumption \ref{assumptions on kernel 1} and $K(t)$ is not continuous on positive integers and $\frac{3}{2}$. Fix $x=0\in \mathbb{R}^2$. We can easily see that $K(\frac{\|y-0\|_{\mathbb{R}^2}}{\epsilon})=K(\frac{\|y\|_{\mathbb{R}^2}}{\epsilon})$ is not a Riemann integrable function  on $\gamma(\theta)$, for any $\epsilon=\frac{1}{k}$ where $k$ can be any positive integer.   In other word, given the above manifold and kernel, there are arbitrarily small bad choices of $\epsilon$ so that the kernel fails to be Riemann integrable. 
\end{example}

 \subsection{A comparison between Riemann integrable kernels and kernels in the VC class}

We recall the VC class of kernels subject to a set $A$ in $\mathbb{R}^p$, and refer readers with interest to \cite{shorack2009empirical} and \cite{pollard2012convergence} for more general definition. Suppose $K(\textbf{u})$ is a bounded function on $\mathbb{R}^p$ and $K(\textbf{u}) \in L^1 (\mathbb{R}^p)$. Note that $K(\textbf{u}) $ is also in  $ L^2 (\mathbb{R}^p)$.   For any $A\subset\mathbb{R}^p$, we consider the {\em space of kernels} over $A$, denoted as $\mathcal{F}(A)$, by 
\begin{align}
\mathcal{F}(A)=\{K(x - \cdot ), \, x\in A \}.
\end{align}
Suppose $\mathcal{P}$ is any probability measure defined on the $\sigma$-algebra of the Borel sets in $\mathbb{R}^p$. Then the $L^2(\mathcal{P})$ metric over $\mathcal{F}(A)$ is defined as 
\begin{align}\label{general L2 metric}
d_{L^2(\mathcal{P})}(K(x - \cdot ), K(y - \cdot ))=(\int_{\mathbb{R}^p} (K(x - z)-K(y -z))^2 d \mathcal{P}(z) ) ^{\frac{1}{2}}.
\end{align}
Let $N_{cov}(\epsilon, \mathcal{F}(A), d_{L^2(\mathcal{P})})$ be the $\epsilon$-covering number of $\mathcal{F}(A)$ with respect to the metric $d_{L^2(\mathcal{P})}$. Then $\mathcal{F}(A)$ is a VC class, whenever there exist constants $C>0$ and $b>0$ such that for all $0<\epsilon<1$,
\begin{align}
\sup_{\mathcal{P}} N_{cov}(\epsilon, \mathcal{F}(A), d_{L^2(\mathcal{P})}) \leq C \epsilon^{-b},
\end{align}
where the supremum is taken over all the probability measure $\mathcal{P}$ defined on the $\sigma$ algebra of the Borel sets in $\mathbb{R}^p$.
The constants $C$ and $b$ are called the VC characteristics. In \cite{shorack2009empirical}, \cite{pollard2012convergence}, \cite{nolan1987u}, \cite{van2000applications}, the authors discuss several sufficient conditions for $\mathcal{F}(A)$ to be a VC class. For example, if $g_1$ is a bounded real function with bounded  variation, $g_2(\textbf{u})$ is a polynomial on $\mathbb{R}^p$ and $K(\textbf{u})=g_1(g_2(\textbf{u}))$, then $\mathcal{F}(A)$ is a VC class.

Next, we show an example of kernel which is Riemann integrable but not in VC class. 
\begin{example} \label{not VC class}
Let $K(t)=\sin(\exp(\exp (\frac{1}{|t|})))$ for $|t| \leq 1$ and $K(t)=0$ for $|t|>1$. $K(t)$ is discontinuous at $t=-1, 0, 1$. Hence, $K(t)$ is Riemann integrable. It is also trivial to modify $K(t)$ so it is discontinuous only at  $t=0$. Suppose $A=[0,1] \subset \mathbb{R}$, then $\mathcal{F}(A)$ is not a VC class. The proof is in Appendix \ref{proof of not VC class}. 
\end{example}

\begin{example} \label{not RI class}
There are kernels in the VC class that are not Riemann integrable. For example, suppose $K_1$ is the characteristic function on the irrational numbers in $[0,1]$. $K_1$ is clearly in the VC class but not Riemann integrable. 
To further explore $K_1$, suppose $K_2$ is the characteristic function on $[0,1]$. Clearly, $K_2$ differs from $K_1$ by a measure $0$ set, and $K_1$ and $K_2$ have the same $\epsilon$ covering number of any closed interval $A$ on $\mathbb{R}$.  However,  it can be shown that $K_1$ and $K_2$ have the same $L^\infty$ convergence rate when they are applied for KDE. In other words, the KDE behavior of the non-Riemann integrable kernel $K_1$ can be studied through the Riemann integrable kernel $K_2$. 
\end{example}

Before closing this subsection, we have a comparison of our results with those shown in \cite{gine2002rates}. The authors in \cite{gine2002rates} proved the following variance analysis result for the density estimation on the Euclidean space $\mathbb{R}^d$ by using the kernel in the VC class.
\begin{theorem}\label{Gine Guillou}[Gin\'e and Guillou]
Let $P$ be a bounded probability density function on $\mathbb{R}^d$. Let $\textbf{v}_1, \cdots, \textbf{v}_n$ be a sequence of i.i.d samples from $\mathbb{R}^d$ based on $P$. Suppose $K(\textbf{v})$ is in the VC class of $\mathbb{R}^d$ with the VC characteristics $C$ and $b$. Suppose
\begin{align}\label{condition 1 in Gine}
 \|P\|_{\infty}  \int_{\mathbb{R}^d} |K(\textbf{v})|^2d\textbf{v} \leq D.
\end{align}
For $\textbf{v} \in \mathbb{R}^d$, define the kernel density estimator at $\textbf{v}$ to be
\begin{align}
K_{n,\epsilon}(\textbf{v})=\frac{1}{n \epsilon^d}\sum_{i=1}^n K\left(\frac{\textbf{v}_i-\textbf{v}}{\epsilon}\right)\,. 
\end{align}
Suppose $C_1$ and $C_2$ are constants depending on the VC characteristics. For any $c_1>C_1$ and $0<\gamma<\frac{c_1D}{\|K\|_\infty}$, there is a $n_0$ depending on $\gamma$, $D$, $\|K\|_\infty$ and the VC characteristics, such that if $n>n_0$, then
\begin{align}\label{GiniMainresult}
P\left\{\sup_{x \in \mathbb{R}^d} |\mathbb{E}K_n(\textbf{v})-K_n(\textbf{v})| \geq 2 \gamma\right\} \leq C_2 \exp\left\{-\frac{1}{D}\frac{\log(1+\frac{4c_1}{C_2})}{c_1C_2}  n \epsilon^d \gamma^2\right\}\,,
\end{align}
\end{theorem}

It is reasonable to compare Theorem \ref{Gine Guillou} with  Corollary \ref{compact support variance}. In Corollary \ref{compact support variance}, we prove that if $\epsilon_n \leq \mathcal{D}_3$, then
\begin{align}\label{OurCorollary1Statement}
P\left\{\sup_{x \in M} |\mathbb{E}K_n(x)-K_n(x)| \geq \mathcal{D}_1 \gamma \right\} \leq 8(2n)^{2p}\exp\left\{-\mathcal{D}_2n \epsilon_n^{d}\frac{ \gamma^2}{N(\gamma)^2}\right\}\,,
\end{align}
which has an extra $\log(n)$ term compared with \eqref{GiniMainresult}.
The term $\frac{\log(1+\frac{4c_1}{C_2})}{c_1C_2}$ in \eqref{GiniMainresult} and the term $\frac{1}{N(\gamma)^2}$ in \eqref{OurCorollary1Statement} both characterize the regularity of the kernel. Besides this difference, the convergence rate in Theorem \ref{Gine Guillou} and Corollary \ref{compact support variance} are the same.  However, in the case when the kernel is not in the VC class,  $\frac{\log(1+\frac{4c_1}{C_2})}{c_1C_2}$ does not exist.  
{
It would be an interesting future direction to generalize the $L^\infty$ convergence result with a kernel in the VC class and provide the rate in the manifold setup. 

}

{
\section{Kernel density estimation with non-translational invariant Riemann integrable kernels}\label{Section:KDE3} 

We switch our discussion to more general Riemann integrable kernels that might be non-translational invariant. As we discussed in the introduction, this is the case that we may encounter when we analyze modern machine learning algorithms. Note that when the kernel function is defined on the Euclidean space, the kernel's property might be different when it is restricted on the manifold, as is shown in Section \ref{Relationship between intrinsic and extrinsic}. In this case, we thus need to further know the kernel's behavior on the manifold so that we can guarantee if the density estimator works.
In the field of manifold learning, sometimes the kernel is implicitly defined on the manifold via the algorithm. In this case, we could focus on the kernel analysis. 

We focus on a special non-translational invariant kernel that is inspired by the LLE algorithm. 
Suppose a manifold learning algorithm implicitly provides us with kernel via a chart of $\iota(M)$, where the kernel might be irregular. We can use this kernel to construct a kernel density estimator. 
We need the following two assumptions to describe the general kernel function.  

\begin{Assumption}\label{assumptions on non isotropic kernel}
\begin{enumerate}[(i)]
\item
$K(\textbf{v}):\mathbb{R}^d \rightarrow \mathbb{R}$ is a bounded function such that $\sup|K(\textbf{v})| = K_{\sup}$ for some $K_{\sup}>0$.
\item
$K$ is Riemann integrable on $\mathbb{R}^d$ with the support contained in the cube $[-R,R]^d$.
\item
$\int_{ \mathbb{R}^d} K(\textbf{v})d\textbf{v}=1$.
\end{enumerate}
\end{Assumption}

Next, suppose the chart satisfies the following assumption.  Recall that with a chart, for each point $x\in M$, there is a local diffeormorphism from $M$ to some Euclidean space.
\begin{Assumption}\label{assumptions on charts of M}
Suppose there are constants $r$, $D_1 \geq 1$, and $D_2\geq0$ such that for any $x \in M$ there is a diffeomorphism $\Phi_x$ satisfying the following conditions: 
\begin{enumerate}[(i)]
\item
$\Phi_x(\textbf{v}): B^{\mathbb{R}^d}_r(0) \rightarrow \iota(M)$ with $\Phi_x(0)=\iota(x)$.  $B^{\mathbb{R}^d}_r(0)$ is the open Euclidean ball of radius $r$ centered at $0$.
\item
For $\textbf{v}_1, \textbf{v}_2 \in B^{\mathbb{R}^d}_r(0)$, we have  $\frac{1}{D_1} \|\textbf{v}_1- \textbf{v}_2\|_{\mathbb{R}^d} \leq d_{\iota(M)} (\Phi_x(\textbf{v}_1),\Phi_x(\textbf{v}_2))  \leq  D_1 \|\textbf{v}_1- \textbf{v}_2\|_{\mathbb{R}^d} $, where $d_{\iota(M)}$ is the geodesic distance on $\iota(M)$. 
\item
Suppose the volume form of $\iota(M)$ in the chart $\Phi_x$ can be expressed as $dV=U_x(\textbf{v}) d \textbf{v}$ for $\textbf{v} \in B^{\mathbb{R}^d}_r(0)$.  Then, $|\frac{U_x(\textbf{v})}{U_x(0)}-1| \leq D_2 \|\textbf{v}\|_{\mathbb{R}^d}$.
\end{enumerate}
\end{Assumption}
Note that (ii) in the above assumption implies that $|U_x(\textbf{v})|\leq D^d_1$ for any $\textbf{v} \in B^{\mathbb{R}^d}_r(0)$. Hence, (iii) is equivalent to the fact that we have a $C^1$ bound of the volume form at $0$. 
Now, we define the kernel density estimator. For $x \in M$, define the kernel density estimator at $x$ with the bandwidth $\epsilon_n>0$ as
\begin{align}
K_{n}(x)=\frac{1}{n \epsilon_n^d U_x(0)}\sum_{i=1}^n K_{\epsilon_n}(x,x_i)\,,
\end{align}
where 
\[
K_{\epsilon_n}(x,x_i):=K\left(\frac{\Phi^{-1}_x(\iota(x_i))}{\epsilon_n}\right)
\] 
if $\iota(x_i) \in \Phi_x(B^{\mathbb{R}^d}_r(0))$ and $K_{\epsilon_n}(x,x_i)=0$ otherwise. Note that the kernel function $K$ is defined on $\mathbb{R}^d$, which can be understood as the tangent space. Also note that the kernel is not translational invariant since in general $K_{\epsilon_n}(x,x_i)$ cannot be written as $k(x-x_i)$ for some function $k$ defined on $\mathbb{R}^p$. In other words, at different points, we use different weights to estimate the density function. Clearly, this is a special non-translational invariant kernel. In general, even if we have the chart, we still need to estimate $U_x(0)$. In the special case when $\Phi_x (\textbf{v})= \iota \circ \exp_x(\textbf{v})$ where $\exp_x$ is the exponential map at $x$, an expansion of $U_x(\textbf{v})$ in terms of $\textbf{v}$ is shown in Lemma \ref{volume form in normal coordinate} in Appendix. Especially,  $U_x(0)=1$ in this case.

Now we claim the strong uniform consistency of this kernel density estimator. Define the partition number $N(\gamma)$ of $K(\textbf{v})$ over $[-R, R]^d$ in the same way as the partition number for an isotropic kernel in \eqref{equation def partition number}. We have the following result.
\begin{theorem}\label{Theorem nontranslational invariant Riemann}
Suppose Assumptions  \ref{holder assumption of P}, \ref{assumptions on non isotropic kernel}, and \ref{assumptions on charts of M} hold. Take $0<\gamma<1$ and assume $\epsilon_n \rightarrow 0$ as $n \rightarrow \infty$.  If $\epsilon_n \leq \Omega_3\gamma^{2}$, then
\begin{align}
P\left\{\sup_{x \in M}|K_n(x)-P(x)| \leq \Omega_1 \gamma^\kappa \right\} \geq 1-8(2n)^{2p^2+2p}\exp\left\{-\Omega_2\frac{n \epsilon_n^{d} \gamma^2}{N(\gamma)^2}\right\}\,,
\end{align}
where $\Omega_1$  depends on $p$, $R$, $d$, $\kappa$, $C_P$, $P_{Max}$, $K_{sup}$, $D_1$, $D_2$, and the second fundamental form of $\iota(M)$. $\Omega_2$  depends on $p$, $d$, $P_{Max}$, $K_{sup}$, $D_1$, and the second fundamental form of $\iota(M)$. $\Omega_3$  depends on $r$, $R$, $D_1$, $P_{Max}$, and the second fundamental form of $\iota(M)$. 
\end{theorem}

Thanks to the chart, the proof of the above theorem is a direct extension of the variance and the bias analysis for the isotropic kernel case shown in Section \ref{Section:KDE2}. Since the proof is similar, we omit details but only indicate key steps. Specifically, by (ii)  in Assumption \ref{assumptions on charts of M} and the argument in Lemma \ref{approximation to K} in Appendix, we can find an approximation of $K(\frac{\Phi^{-1}_x(\iota(y))}{\epsilon_n})$ by a simple function. The rest of the proof in the variance analysis is same as the proof of Theorem \ref{EK_n-K_n}. The bias analysis follows the same argument as in the proof of Theorem \ref{bias analysis main theorem} by applying the change of variable formula in the chart $\Phi_x$.

The kernel function associated with LLE \cite{wu2018think} that motivates the above result deserves a discussion. 
Take $\rho\in\mathbb{R}$ as the regularization order in the LLE algorithm \cite[See (2.9) and (3.5)]{wu2018think} and $\epsilon>0$ as a bandwidth. The kernel associated with {LLE} is  
\begin{equation}
K_{\texttt{LLE}}(x,y)=[1- \mathbf{T}_{\iota(x)}^\top(\iota(y)-\iota(x))]\chi_{B_{\epsilon}^{\mathbb{R}^p}(\iota(x)) \cap \iota(M)}(\iota(y)),
\end{equation}
where $x,y\in M$, 
\begin{equation}
\mathbf{T}_{\iota(x)}:= \mathcal{I}_{\epsilon^{d+\rho}}(C_{x})\big[\mathbb{E}(X-\iota(x))\chi_{B_{\epsilon}^{\mathbb{R}^p}(x)}\big]    \in \mathbb{R}^p\,,\label{Definition:Tx:ContinuousCase}
\end{equation}
\[
C_x:=\mathbb{E}[(X-\iota(x))(X-\iota(x))^{\top}\chi_{B_{\epsilon}^{\mathbb{R}^p}(\iota(x))}(X)]\,
\]
and
$\mathcal{I}_{\epsilon^{d+\rho}}(C_{x})$ is the regularized pseudoinverse of $C_x$ with the regularization $\epsilon^{d+\rho}$.   
Here, $C_x$ is understood as the {\em local covariance matrix}, which depends on the embedding and hence the extrinsic geometry of the manifold. This quantify is commonly considered in the local principal component analysis. Moreover, $\mathbf{T}_{\iota(x)}$ depends on the manifold curvature, mainly the extrinsic curvature information.

Note that the LLE algorithm is based on the barycentric coordinate, and its relationship with KDE by using $K_{\texttt{LLE}}$ is not obvious by reading the algorithm. In \cite[Sections 2 and 3]{wu2018think}, the barycentric coordinate is carefully rewritten so that locally points are represented as coordinates associated with a section of the associated frame bundle. Such conversion naturally leads to the KDE analysis. The density estimation by using $K_{\texttt{LLE}}$ appears naturally in the algorithm as is shown in \cite[(3.15)]{wu2018think}.
Clearly, in general the kernel function $K_{\texttt{LLE}}$ is non-translational invariant, and it is not designed by the user but emerges implicitly in the algorithm. Moreover, when the scaling parameter $\epsilon$ changes and the density function changes, the kernel function itself changes correspondingly. Next, we consider a special case.

\begin{example}
Suppose $M=S^{p-1}$, the canonical $(p-1)$-dim sphere, and it is isometrically embedded in $\mathbb{R}^p$ through $\iota$ with the center at the origin. To simplify the discussion, we also assume the data is uniformly sampled. By a direct calculation in \cite[Appendix G.1]{wu2018think} and the cosine law, we can describe $K_{\texttt{LLE}}$ as follows. For $\textbf{v} \in \mathbb{R}^{p-1} $, let 
\begin{align}
K_{\epsilon}(\textbf{v}) = [1-a(\epsilon, \rho)+a(\epsilon, \rho)\cos(\|\textbf{v}\|_{\mathbb{R}^{p-1}})]\chi_{B_{r(\epsilon)}^{\mathbb{R}^{p-1}}}(\textbf{v}), 
\end{align}
where $r(\epsilon)=\arccos(1-\frac{\epsilon^2}{2})$ and $a>0$ is a function depending on $\epsilon$ and $\rho$.  Then, $K_{\texttt{LLE}}(x,y)=K_{\epsilon}(\exp^{-1}_x(y))$ for $y$ in the geodesic ball of radius $r(\epsilon)$ around $x$ and $K_{\texttt{LLE}}(x,y)=0$ otherwise. 
Based on the above discussion, we see that $K_{\texttt{LLE}}(x,y)$ is a kernel depending on the intrinsic structure (the normal coordinates) of the manifold. Due to the symmetry of $S^{p-1}$, $K_{\texttt{LLE}}(x,y)$ can be further simplified. In this special case, $\|\exp^{-1}_x(y)\|_{\mathbb{R}^{p-1}}=\arccos(1-\frac{\|\iota(x)-\iota(y)\|_{\mathbb{R}^p}^2}{2})$. Hence, if we take
\begin{align}
\tilde{K}_{\epsilon}(t) = \left[1-\frac{a(\epsilon, \rho)\epsilon^2}{2} t^2\right]\chi_{[0,\epsilon]}(t), 
\end{align}
$K_{\texttt{LLE}}(x,y)=\tilde{K}_{\epsilon}(\frac{\|\iota(x)-\iota(y)\|_{\mathbb{R}^p}}{\epsilon})$.
When $\rho \leq 4$, there is a constant $C$ such that $0<a(\epsilon, \rho)\epsilon^2<C$ for all $\epsilon$ small enough. Therefore, if $\rho \leq 4$ and $\epsilon$ is small enough, $K_{\epsilon}(\|\textbf{v}\|_{\mathbb{R}^p})$ has a uniform $C^1$ bound except for $\textbf{v}$ on the boundary of $B_{\epsilon}^{\mathbb{R}^p}(0)$ and $K_{\epsilon}(\|\textbf{v}\|_{\mathbb{R}^p})$ has a $L^\infty$ bound for $\textbf{v}$ on the boundary of $B_{\epsilon}^{\mathbb{R}^p}(0)$. The volume of the boundary of $B_{\epsilon}^{\mathbb{R}^p}(0)$ becomes smaller when $\epsilon$ is smaller. Hence, there is a uniform partition so that \eqref{equation def partition number} holds for all $K_{\epsilon}(\|\textbf{v}\|_{\mathbb{R}^p})$ with $\epsilon$ small enough. We conclude that we have an upper bound of the partition numbers for all  $K_{\epsilon}(\|\textbf{v}\|_{\mathbb{R}^p})$ that only depends on $\gamma$ and is independent of $\epsilon$. Hence, the $L^\infty$ convergence and its rate of the density estimator $\frac{p-1}{|S^{p-2}| n \epsilon_n^{p-1}} \sum_{i=1}^n K_{\texttt{LLE}}(x, x_i)$ for $\rho \leq 4$ are the same as those shown in Part 2, Theorem \ref{density estimation} and can be proved by the same method.
\end{example}

 In this simple case, due to the symmetry, the kernel is isotropic. In general, the LLE kernel is far from being so nice, and its strong uniform consistency cannot be inferred by  Theorem \ref{density estimation} or Theorem \ref{Theorem nontranslational invariant Riemann}. While the LLE kernel is in general not too irregular, the main challenge is that the kernel, and hence the partition number, depends on $\epsilon$.
In general, the convergence analysis should depend on the algorithm, where the kernel might be not only non-translational invariant, but also dependent on $\epsilon$ and irregular. We will report such analyses in our future research.
}

\section{Kernel density estimation with Lebesgue integrable kernels}\label{Section:KDE4}

We consider the following assumption on the Lebesgue integrable kernel.

\begin{Assumption}\label{assumptions on kernel}
$K_{\epsilon}:\iota(M) \times \iota(M) \rightarrow \mathbb{R}_{\geq 0}$, where $\epsilon>0$, is a sequence of kernel functions so that the following conditions are satisfied for all $\epsilon$:
\begin{enumerate}[(i)]
\item There exists a constant $K_{\sup}>0$, such that $0 \leq K_\epsilon(\iota(x),\iota(y)) \leq \frac{K_{\sup}}{\epsilon^\alpha}$, for all $x,y \in M$ and $\alpha \geq d$.

\item $K_{\epsilon}$ is a Lebesgue measurable function on $\iota(M) \times \iota(M)$.

\item $K_{\epsilon}(\iota(x),\iota(y))=0$ if $\|\iota(x)-\iota(y)\|_{\mathbb{R}^p}>\epsilon$.

\item $\int_{M} K_{\epsilon}(\iota(x),\iota(y))dV(y)=1$, where $dV$ is the volume form on $M$.
\end{enumerate}
\end{Assumption} 

\begin{remark}
For the kernel $K(t)$ satisfies Assumption \ref{assumptions on kernel 1}, if we apply the same notation as in Assumption  \ref{assumptions on kernel}, we have $K_{\epsilon}(\iota(x), \iota(y))=\frac{1}{\epsilon^d}K\left(\frac{\|\iota(y)-\iota(x)\|_{\mathbb{R}^p}}{\epsilon}\right)$. Hence, for the isotropic kernel,  $\epsilon^d K_{\epsilon}(\iota(x), \iota(y)) \leq K_{\sup}$  for all $\epsilon$. In contrast, if $K_{\epsilon}$ is a kernel defined as in  Assumption  \ref{assumptions on kernel}, then $\epsilon^d K_{\epsilon}(\iota(x), \iota(y)) \leq \frac{K_{\sup}}{\epsilon^{\alpha-d}} $. In other words, Assumption  \ref{assumptions on kernel} is more general than Assumption \ref{assumptions on kernel 1} in the sense that we allow the kernel to blow up as $\epsilon \rightarrow 0$.
\end{remark}

Let $x_1, \cdots, x_n$ be a sequence of i.i.d samples from $M$ based on the probability density function $P$ on $M$.  For $x \in M$, we define the kernel density estimator at $x$ to be
\begin{align}
K_{n,\epsilon}(x)=\frac{1}{n}\sum_{i=1}^n K_{\epsilon}(\iota(x),\iota(x_i))\,.
\end{align}
Clearly, we have
\begin{align}
\mathbb{E}K_{\epsilon}(x)=\int_{M}K_{\epsilon}(\iota(x),\iota(y))P(y)dV(y).
\end{align}
Then we have the following result for the variance analysis. 
\begin{proposition}\label{variance L1 case}
Suppose $P$ is measurable on $M$ with $\|P\|_\infty = P_{max}$ for some $P_{max}>0$. Under Assumption \ref{assumptions on kernel}, if $\epsilon \rightarrow 0$ and  $\frac{\log n}{n \epsilon^{2\alpha-d}} \rightarrow 0$ as $n \rightarrow \infty$, then, a.s.
\begin{align}
\int_{M} \left|\frac{1}{n}\sum_{i=1}^n K_{\epsilon}(\iota(x),\iota(x_i))-\mathbb{E}K_{\epsilon}(x)\right| dV(x) \rightarrow 0. 
\end{align}
\end{proposition}

The proof of the proposition is in Appendix \ref{Appendix C}, which generalizes the method developed in \cite{devroye1979l1}. Note that since we apply Scheffe's Lemma, we require the kernel $K_{\epsilon}$ to be non-negative. The bias analysis includes the following two cases. The proof of the proposition is in Appendix \ref{Appendix C}.

\begin{proposition}\label{bias L1 case}
\begin{enumerate}
\item
 Under Assumption \ref{assumptions on kernel}, suppose $P$ is continuous on $M$ with $\|P\|_\infty = P_{max}$  for some $P_{max}>0$,  then $\mathbb{E}K_{\epsilon}(x) \rightarrow P(x)$ for all $x \in M$ as $\epsilon \rightarrow 0$.
\item
 Under Assumption \ref{assumptions on kernel}, suppose $\alpha=d$ and $P$ is measurable on $M$ with $\|P\|_\infty = P_{max}$ for some $P_{max}>0$,  then $\mathbb{E}K_{\epsilon}(x) \rightarrow P(x)$ for  almost every $x \in M$ as $\epsilon \rightarrow 0$.
\end{enumerate}
\end{proposition}

Since we have
\begin{align}
&\int_{M}\left |\frac{1}{n}\sum_{i=1}^n K_{\epsilon}(\iota(x),\iota(x_i))-P(x)\right| dV(x)\nonumber \\
\leq &\,\int_{M} \left|\frac{1}{n}\sum_{i=1}^n K_{\epsilon}(\iota(x),\iota(x_i))-\mathbb{E}K_{\epsilon}(x)\right| dV(x) + \int_{M} |\mathbb{E}K_{\epsilon}(x)-P(x)| dV(x)\,, \nonumber
\end{align}
we can concatenate the previous two propositions and conclude the convergence of $K_{n,\epsilon}$ to $P$ in the $L^1$ sense.

\begin{theorem}\label{KDE L1 convergence}
Under Assumption \ref{assumptions on kernel}, suppose $\epsilon \rightarrow 0$ and $\frac{\log n}{n \epsilon^{2\alpha-d}} \rightarrow 0$ as $n \rightarrow \infty$. Then, a.s.
\begin{align}
\int_{M}|K_{n,\epsilon}(x)-P(x)|dV(x) \rightarrow 0
\end{align}
whenever one of the following conditions holds.
\begin{enumerate}
\item
$P$ is continuous on $M$ with $\|P\|_\infty = P_{max}$ for some $P_{max}>0$.
\item
$\alpha=d$ and $P$ is measurable on $M$ with $\|P\|_\infty = P_{max}$ for some $P_{max}>0$.
\end{enumerate}
\end{theorem}

In a specific case, by using the same argument as in Theorem \ref{KDE L1 convergence}, we can prove the convergence of the density estimation in $L^1$ sense for an isotropic nonnegative Lebesgue integrable kernel.

\begin{theorem}
Suppose $K(t):\mathbb{R}_{\geq 0} \rightarrow \mathbb{R}_{\geq 0}$ is a non-negative Lebesgue integrable function satisfying conditions (i), (iii) and (iv) in Assumption \ref{assumptions on kernel 1}, and $P$ is measurable on $M$ with $\|P\|_\infty = P_{max}$ for some $P_{max}>0$.  Suppose $\epsilon \rightarrow 0$ and $\frac{\log n}{n \epsilon^{d}} \rightarrow 0$ as $n \rightarrow \infty$. Then, for the kernel defined in \eqref{Definition Kn isotropic}, we have a.s.
\begin{align}
\int_{M}|K_{n}(x)-P(x)|dV(x) \rightarrow 0
\end{align}
\end{theorem}

\section*{Acknowledgements}
The authors would like to thank Professor Tyrus Berry and Professor Yen-Chi Chen about the helpful suggestions of the literatures. 
\

\bibliographystyle{plain}
\bibliography{bib}

\begin{thebibliography}{10}

\bibitem{belkin2007convergence}
M.~Belkin and P.~Niyogi.
\newblock Convergence of laplacian eigenmaps.
\newblock In {\em Advances in Neural Information Processing Systems}, pages
  129--136, 2007.

\bibitem{berenfeld2019density}
Cl{\'e}ment Berenfeld and Marc Hoffmann.
\newblock Density estimation on an unknown submanifold.
\newblock {\em arXiv preprint arXiv:1910.08477}, 2019.

\bibitem{berry2017density}
Tyrus Berry and Timothy Sauer.
\newblock Density estimation on manifolds with boundary.
\newblock {\em Computational Statistics \& Data Analysis}, 107:1--17, 2017.

\bibitem{chen2019generalized}
Yen-Chi Chen et~al.
\newblock Generalized cluster trees and singular measures.
\newblock {\em Annals of Statistics}, 47(4):2174--2203, 2019.

\bibitem{coifman2006diffusion}
Ronald~R Coifman and St{\'e}phane Lafon.
\newblock Diffusion maps.
\newblock {\em Applied and computational harmonic analysis}, 21(1):5--30, 2006.

\bibitem{cover1965geometrical}
Thomas~M Cover.
\newblock Geometrical and statistical properties of systems of linear
  inequalities with applications in pattern recognition.
\newblock {\em IEEE transactions on electronic computers}, (3):326--334, 1965.

\bibitem{devroye1979l1}
LP~Devroye and TJ~Wagner.
\newblock The l1 convergence of kernel density estimates.
\newblock {\em The Annals of Statistics}, pages 1136--1139, 1979.

\bibitem{devroye1980strong}
Luc~P Devroye and Terry~J Wagner.
\newblock The strong uniform consistency of kernel density estimates.
\newblock In {\em Multivariate Analysis V: Proceedings of the fifth
  International Symposium on Multivariate Analysis}, volume~5, pages 59--77,
  1980.

\bibitem{einmahl2005uniform}
Uwe Einmahl and David~M Mason.
\newblock Uniform in bandwidth consistency of kernel-type function estimators.
\newblock {\em The Annals of Statistics}, 33(3):1380--1403, 2005.

\bibitem{gine2002rates}
Evarist Gin{\'e} and Armelle Guillou.
\newblock Rates of strong uniform consistency for multivariate kernel density
  estimators.
\newblock In {\em Annales de l'Institut Henri Poincare (B) Probability and
  Statistics}, volume~38, pages 907--921. Elsevier, 2002.

\bibitem{kim2013geometric}
Yoon~Tae Kim and Hyun~Suk Park.
\newblock Geometric structures arising from kernel density estimation on
  riemannian manifolds.
\newblock {\em Journal of Multivariate Analysis}, 114:112--126, 2013.

\bibitem{ma2017adaptive}
Li~Ma et~al.
\newblock Adaptive shrinkage in p{\'o}lya tree type models.
\newblock {\em Bayesian Analysis}, 12(3):779--805, 2017.

\bibitem{nolan1987u}
Deborah Nolan and David Pollard.
\newblock U-processes: rates of convergence.
\newblock {\em The Annals of Statistics}, pages 780--799, 1987.

\bibitem{ozakin2009submanifold}
Arkadas Ozakin and Alexander~G Gray.
\newblock Submanifold density estimation.
\newblock In {\em Advances in Neural Information Processing Systems}, pages
  1375--1382, 2009.

\bibitem{parzen1962estimation}
Emanuel Parzen.
\newblock On estimation of a probability density function and mode.
\newblock {\em The annals of mathematical statistics}, 33(3):1065--1076, 1962.

\bibitem{pelletier2005kernel}
Bruno Pelletier.
\newblock Kernel density estimation on riemannian manifolds.
\newblock {\em Statistics \& probability letters}, 73(3):297--304, 2005.

\bibitem{pollard2012convergence}
David Pollard.
\newblock {\em Convergence of stochastic processes}.
\newblock Springer Science \& Business Media, 2012.

\bibitem{rinaldo2010generalized}
Alessandro Rinaldo and Larry Wasserman.
\newblock Generalized density clustering.
\newblock {\em The Annals of Statistics}, 38(5):2678--2722, 2010.

\bibitem{rosenblatt1956remarks}
Murray Rosenblatt.
\newblock Remarks on some nonparametric estimates of a density function.
\newblock {\em The Annals of Mathematical Statistics}, pages 832--837, 1956.

\bibitem{sauer1972density}
Norbert Sauer.
\newblock On the density of families of sets.
\newblock {\em Journal of Combinatorial Theory, Series A}, 13(1):145--147,
  1972.

\bibitem{Roweis_Saul:2003}
L.~K. Saul and S.~T. Roweis.
\newblock Think globally, fit locally: unsupervised learning of low dimensional
  manifolds.
\newblock {\em Journal of Machine Learning Research}, 4(Jun):119--155, 2003.

\bibitem{shorack2009empirical}
Galen~R Shorack and Jon~A Wellner.
\newblock {\em Empirical processes with applications to statistics}.
\newblock SIAM, 2009.

\bibitem{van2000applications}
Sara~A Van~de Geer.
\newblock {\em Applications of empirical process theory}, volume~91.
\newblock Cambridge University Press Cambridge, 2000.

\bibitem{vapnik2015uniform}
Vladimir~N Vapnik and A~Ya Chervonenkis.
\newblock On the uniform convergence of relative frequencies of events to their
  probabilities.
\newblock In {\em Measures of complexity}, pages 11--30. Springer, 2015.

\bibitem{wong2010optional}
Wing~H Wong, Li~Ma, et~al.
\newblock Optional p{\'o}lya tree and bayesian inference.
\newblock {\em The Annals of Statistics}, 38(3):1433--1459, 2010.

\bibitem{wu2018think}
Hau-Tieng Wu and Nan Wu.
\newblock Think globally, fit locally under the manifold setup: Asymptotic
  analysis of locally linear embedding.
\newblock {\em The Annals of Statistics}, 46(6B):3805--3837, 2018.

\bibitem{wu2018locally}
Hau-tieng Wu and Nan Wu.
\newblock When locally linear embedding hits boundary.
\newblock {\em arXiv preprint arXiv:1811.04423}, 2018.

\bibitem{zygmund1989certain}
A~Zygmund.
\newblock On certain lemmas of marcinkiewicz and carleson.
\newblock In {\em Selected Papers of Antoni Zygmund}, pages 432--440. Springer,
  1989.

\end{thebibliography}

\clearpage

\appendix

\section{Proof of  Theorem \ref{EK_n-K_n} and Corollary \ref{0-1 kernel}}\label{Appendix A}

\subsection{\textbf{Variance analysis over the regular subsets of the manifold}}\label{variance on regular subsets}
{We first define the rectangles and the balls as follows.}
\begin{definition}\label{def of rectangle, ball and cube}
We define a rectangle $R$ in $\mathbb{R}^p$ as any set that is isometric to $(a_1, b_1] \times \cdots \times (a_p, b_p]$ in $\mathbb{R}^p$, i.e. $R$ is equal to $(a_1, b_1] \times \cdots \times (a_p, b_p]$ up to a rotation.  Denote $B^{\mathbb{R}^p}$ to be a closed ball in $\mathbb{R}^p$. Unless necessary, we do not specify the centers of the rectangles and the balls. 
We define a half open cube centered at the origin of side length $2a$ in $\mathbb{R}^p$ as
\begin{align}
Q_a=(-a,a]^p.
\end{align}
For $\Omega \subset \mathbb{R}^p$, we define the following sets related to rectangles:
\begin{align}
\mathcal{R}_r(\Omega)=\{R \cap \Omega |\,R \cap \Omega \not= \emptyset , diam(R) \leq r\}\,
\end{align}
and the following sets related to balls:
\begin{align}
\mathcal{B}_r(\Omega)=\left\{B^{\mathbb{R}^p} \cap \Omega \Big|\,B^{\mathbb{R}^p} \cap \Omega \not= \emptyset , \mbox{radius of } B^{\mathbb{R}^p} \leq r\right\}\,.
\end{align}
\end{definition}

{Based on the above definition, the collection of the regular subsets that we discuss are the intersections of rectangles and balls in $\mathbb{R}^p$ with $\iota(M)$, namely, $\mathcal{R}_r(\iota(M))$ and $\mathcal{B}_r(\iota(M))$.}

Suppose the points $\{x_i\}_{i=1}^n$ are i.i.d. samples from the random variables with the density function $P$ supported on the manifold $M$. Let $\iota^{-1}$ be the inverse of $\iota$. Then $P\circ \iota^{-1}$ is the corresponding probability density function on $\iota(M)$. And $\iota_*dV$ is the volume density associated with the metric on $\iota(M)$. Then  $\iota(x_1), \cdots, \iota(x_n)$ can be regarded as a sequence of i.i.d samples from $\iota(M)$ based on the density function $P\circ \iota^{-1}$. We denote the empirical measure associated with the measure $PdV$ as
\begin{equation}
P_n:=\frac{1}{n}\sum_{i=1}^n\delta_{x_i}\,,
\end{equation}
where $\delta_{x_i}$ is the delta measure supported on $x_i$. Similarly, we denote the empirical measure associated with the measure $P \circ \iota^{-1}\iota_*dV$:
\begin{equation}
(P\circ \iota^{-1})_n:=\frac{1}{n}\sum_{i=1}^n\delta_{\iota(x_i)}\,.
\end{equation}
For any Lebesgue measurable subset $A$ of $\iota(M)$, denote 
\begin{equation}
\mu(A):=\int_A P \circ \iota^{-1}\iota_*dV\mbox{ and }\mu_n(A):=\int_A(P\circ \iota^{-1})_n\,.
\end{equation}

Recall the following definitions of the VC dimension and the growth function \cite{vapnik2015uniform}. 
\begin{definition} 
Let $H$ be a family of subsets of $\mathbb{R}^p$. For any finite subset $C \subset \mathbb{R}^p$, the intersection  $H\cap C=\{h \cap C , h\in H\}$ is a family of subset of $C$. Obviously $|H\cap C|\leq 2^{|C|}$. We say that $C$ is shattered by $H$ if $|H\cap C|= 2^{|C|}$, i.e. $H\cap C$ contains all subsets of $C$. The VC dimension $H$ is the largest cardinality of $C$ that can be shattered by $H$.

Suppose  $C=\{\textbf{u}_1 ,\cdots, \textbf{u}_n\} \subset \mathbb{R}^p$. Then, $|H\cap C|\leq 2^n$. Hence, we define the growth function as follows,
\begin{align}
G(H,n)=\max_{C} |H\cap C|,
\end{align}
where $\max$ is taken over all possible $C$, i.e. all possible sets of $n$ points in $\mathbb{R}^p$.

\end{definition}

The Sauer's lemma \cite{sauer1972density} relates the growth function and the VC dimension of $H$.
\begin{lemma}[Sauer] \label{Sauer's lemma}
\begin{align}
G(H,n) \leq (\frac{n e}{d_{VC}})^{d_{VC}},
\end{align}
where $d_{VC}$ is the VC dimension of $H$.
\end{lemma}

Suppose $H_1$ is the set of all $p$ dimensional rectangles in $\mathbb{R}^p$ defined as in Definition \ref{def of rectangle, ball and cube}, then $G(H_1, 2n)  \leq  (2n)^{2p^2+2p}$, \cite{cover1965geometrical}. Let $H_2$ be the set of all closed balls in $\mathbb{R}^p$.  It is also known that the VC dimension of $H_2$ is $p+2$. Hence, by Sauer's lemma, $G(H_2, 2n) \leq  (\frac{2n e}{p+2})^{p+2} \leq  (2n)^{p+2}$. By using the same argument of in Lemma 1 and Lemma 2 in \cite{devroye1980strong}, we have the following Lemma: 
\begin{lemma}\label{rectangle covering lemma}
Let $H_1$ be the set of all $p$ dimensional rectangles in $\mathbb{R}^p$ defined as in Definition \ref{def of rectangle, ball and cube}. Suppose we can find $r>0$  so that  $\sup_{A \in \mathcal{R}_{2r}(\iota(M))} \mu(A) \leq  b \leq \frac{1}{4}$. For any $\delta>0$ , if $n \geq \max\big(\frac{1}{b},\frac{8b}{\delta^2}\big)$, then we have
\begin{align}
P\left\{\sup_{A \in \mathcal{R}_r(\iota(M))} |\mu_n(A)-\mu(A)|\geq \delta \right\} & \leq 4 G(H_1, 2n) \exp\left\{-\frac{n\delta^2}{64b+4\delta}\right\}+8n\exp\left\{-\frac{nb}{10}\right\}\, \\
&\leq 4 (2n)^{2p^2+2p} \exp\left\{-\frac{n\delta^2}{64b+4\delta}\right\}+8n\exp\left\{-\frac{nb}{10}\right\} \, .
\end{align}
\end{lemma}

We can use the same method to prove the following lemma.

\begin{lemma}\label{ball covering lemma}
 Let $H_2$ be the set of all closed balls in $\mathbb{R}^p$. Suppose we can find $r$  so that  $\sup_{A \in \mathcal{B}_{2r}(\iota(M))} \mu(A) \leq  b \leq \frac{1}{4}$. For any $\delta>0$ , if $n \geq \max\big(\frac{1}{b},\frac{8b}{\delta^2}\big)$, then 
\begin{align}
P\left\{\sup_{A \in  \mathcal{B}_r(\iota(M))} |\mu_n(A)-\mu(A)|\geq \delta \right\} & \leq 4G(H_2 ,2n) \exp\left\{-\frac{n\delta^2}{64b+4\delta}\right\}+8n\exp\left\{-\frac{nb}{10}\right\}\, \\
& \leq 4 (2n)^{p+2} \exp\left\{-\frac{n\delta^2}{64b+4\delta}\right\}+8n\exp\left\{-\frac{nb}{10}\right\}.
\end{align}
\end{lemma}

\begin{remark}
The arguments in Lemma 1 and Lemma 2 in \cite{devroye1980strong} do not rely on the structure of $\iota(M)$ at all. In fact, the same conclusions hold for $\sup_{A \in \mathcal{R}_r(\Omega)} |\mu_n(A)-\mu(A)|$ or $\sup_{A \in  \mathcal{B}_r(\Omega)} |\mu_n(A)-\mu(A)|$ for any set $\Omega \subset \mathbb{R}^p$, as long as the sets in $\mathcal{R}_r(\Omega)$ or $\mathcal{B}_r(\Omega)$ are measurable.
\end{remark}

\subsection{\textbf{Local approximation to the Riemann integrable kernel}}

For any $x\in M$, we introduce the following coordinates for $\iota(M)$ which represents $\iota(M)$ locally as the graph of some function. Generally speaking, there are several canonical ways to build up a coordinate on a manifold. For example, there are normal coordinates and harmonic coordinates. {However, we introduce the following one since it is closely related to the construction of the regular sets in the previous subsection and is useful in the local approximation to the restriction of the kernel function on the manifold.}

Let $e_1, \cdots, e_p$ be the standard orthonormal basis of $\mathbb{R}^p$. Without loss of generality, we assume that $\iota$ is an isometric embedding so that $\iota(x)=0 \in \mathbb{R}^p$ and $\iota_*T_x M$ is the subspace $\mathbb{R}^d \subset \mathbb{R}^p$ generated by $e_1, \cdots, e_d$.  Let $\textbf{u}=(u_1, \cdots, u_p) \in \mathbb{R}^p$, we define the projection: $I_x :\mathbb{R}^p \rightarrow \mathbb{R}^d $,
\begin{align}
I_x(\textbf{u})=(u_1, \cdots, u_d).
\end{align}
If $\lambda_1$ is small enough, then $I_x$ is a diffeomorphism from $ [-\lambda_1,\lambda_1]^p \cap \iota(M)$ onto $ [-\lambda_1,\lambda_1]^d\subset \mathbb{R}^d$. Let $I_x^{-1}$ be the inverse map from $ [-\lambda_1,\lambda_1]^d$ to $ [-\lambda_1,\lambda_1]^p  \cap \iota(M)$, then for $\textbf{v}\in \mathbb{R}^d$, we have
\begin{align}
I_x^{-1}(\textbf{v})=(\textbf{v},g_1(\textbf{v}),\cdots, g_{p-d}(\textbf{v}))\,,
\end{align}
where for $k=1, \cdots, p-d$, $g_k$ is a smooth function of $\textbf{v}$, and $g_k(\textbf{v})=\sum_{i,j}a^k_{ij}(x)v_iv_j+O(\|\textbf{v}\|^3_{\mathbb{R}^d})$.
Note that $I_x^{-1}$ is bi-Lipschitz.  Since $M$ is compact, we can choose $\lambda_1$ and $C_{lip}$ that satisfy the following conditions:
\begin{enumerate}
\item
$I_x^{-1}$ is a diffeomorphism from $[-\lambda_1, \lambda_1]^d$ to $ [-\lambda_1,\lambda_1]^p  \cap \iota(M)$, for all $x \in M$.
\item
For any $x \in M$, $I_x^{-1}$ is bi-Lipschitz with a bi-Lipschitz constant bounded by $C_{lip} \geq 1$ so that for any $\textbf{v}_1,\textbf{v}_2 \in [-\lambda_1, \lambda_1]^d$,
\begin{align}
\|\textbf{v}_1-\textbf{v}_2\|_{\mathbb{R}^d} \leq \|I_x^{-1}(\textbf{v}_1)-I_x^{-1}(\textbf{v}_2)\|_{\mathbb{R}^p}  \leq d_M(I_x^{-1}(\textbf{v}_1),I_x^{-1}(\textbf{v}_2)) \leq C_{lip}  \|\textbf{v}_1-\textbf{v}_2\|_{\mathbb{R}^d},
\end{align}
where $d_M(\cdot,\cdot)$ denotes the geodesic distance on $\iota(M)$.
\end{enumerate}
Note that $C_{lip}$ depends on $a^k_{ij}(x)$, and hence depends on the second fundamental form of $\iota(M)$ in $\mathbb{R}^p$.

Next, take $\lambda \leq \lambda_1$ and define a map $J_x$ from $[-\lambda, \lambda]^d$ to $\mathbb{R}^d$ by
\begin{align}\label{map J_x}
J_x(\textbf{v})=\sqrt{\|\textbf{v}\|^2_{\mathbb{R}^d}+g^2_1(\textbf{v})+\cdots+g^2_{p-d}(\textbf{v})}\frac{\textbf{v}}{\|\textbf{v}\|_{\mathbb{R}^d}}\,.
\end{align} 
For a fixed direction $\frac{\textbf{v}}{\|\textbf{v}\|_{\mathbb{R}^d}}$, denote $t:=\|\textbf{v}\|_{\mathbb{R}^d}$ so that $\textbf{v}=\frac{\textbf{v}}{\|\textbf{v}\|_{\mathbb{R}^d}} t $. Since $g_k(\textbf{v})=\sum_{i,j}a^k_{ij}(x)v_iv_j+O(\|\textbf{v}\|^3_{\mathbb{R}^d})$ for $1\leq k \leq p-d$, we have  
\begin{align}
g^2_1(\textbf{v})+\cdots+g^2_{p-d}(\textbf{v})=\left[A\left(x,\frac{\textbf{v}}{\|\textbf{v}\|_{\mathbb{R}^d}} \right)+O(t)\right]^2 t^4
\end{align}
for a smooth function $A:M\times S^{d-1}\to \mathbb{R}$.
Taking the derivative with respect to $t$, we can show that if $\lambda$ is small enough, then $g^2_1(\textbf{v})+\cdots+g^2_{p-d}(\textbf{v})$ is a non decreasing function of $t$. Hence, $J_x(\textbf{v})$ is bijective on $[-\lambda, \lambda]^d$. Since $M$ is compact, we can choose $\lambda$ and $D_{lip}$ that satisfy the following conditions:
\begin{enumerate}
\item
$\lambda \leq \lambda_1$.
\item
$J_x$ is a homeomorphism from $[-\lambda, \lambda]^d$ onto its image for all $x \in M$ and $J_x$ is smooth except at $0$.
\item
For any $x \in M$, $J_x$ is bi-Lipschitz with a bi-Lipschitz constant bounded by $D_{lip} \geq 1$ so that for any $\textbf{v}_1,\textbf{v}_2 \in [-\lambda, \lambda]^d$,
\begin{align}\label{definition of D_lip}
 \frac{1}{D_{lip}}\|\textbf{v}_1-\textbf{v}_2\|_{\mathbb{R}^d} \leq  \|J_x(\textbf{v}_1)-J_x(\textbf{v}_2)\|_{\mathbb{R}^d}  \leq D_{lip}  \|\textbf{v}_1-\textbf{v}_2\|_{\mathbb{R}^d} .
\end{align}
\end{enumerate}

{Note that $D_{lip}$ depends on $a^k_{ij}(x)$.  In other words, $D_{lip}$ depends on the bound of the second fundamental form of $\iota(M)$ in $\mathbb{R}^p$. Hence, when the manifold is compact, the choice of $D_{lip}$ is independent of $x$.  Moreover, as a consequence of \eqref{map J_x}, we have $[-\lambda, \lambda]^d \subset J_x([-\lambda, \lambda]^d) \subset [-D_{lip}\lambda, D_{lip}\lambda]^d$.

By using the maps $J_x$ and $I_x$ constructed above, we have the following local approximation result of the kernel function. }

\begin{lemma}\label{approximation to K}
{For any  $x \in M$, we rotate and translate $\iota(M)$ so that $\iota(x)=0 \in \mathbb{R}^p$ and $\iota_*T_x M$ is the subpace $\mathbb{R}^d \subset \mathbb{R}^p$ generated by $e_1, \cdots, e_d$. Let $\lambda$ and $D_{lip}$ be the constants described as in the construction of the map $J_x$ in \eqref{map J_x}  and \eqref{definition of D_lip}.} Suppose $K(t):\mathbb{R}_{\geq 0} \rightarrow \mathbb{R}$ is a bounded Riemann integrable function. For any $\eta>0$, choose $\epsilon_n$ so that $ \epsilon_n \eta <\lambda$. We partite $[-D_{lip}\eta, D_{lip}\eta]^d$ uniformly into $\mathcal{N}(\eta,\gamma)$ cubes $\{\mathcal{Q}_i\}_{i=1}^{\mathcal{N}(\eta,\gamma)}$, where each $\mathcal{Q}_i$ is of the form $[v_1,v_1+a] \times \cdots \times [v_d,v_d+a]$ for $a>0$ so that 
\begin{align}
\sum_{i =1}^{\mathcal{N}(\eta,\gamma)}(M_i (K)-m_i (K))Vol(\mathcal{Q}_i) <\gamma^2\,,
\end{align} 
where $M_i (K)$ and $m_i (K)$ are  the supremum and infimum of $K(\|\textbf{v}\|_{\mathbb{R}^d})$ over the cube $\mathcal{Q}_i$  respectively.

For $\textbf{u} \in Q_{\epsilon_n\eta} \cap \iota(M)$, there is an approximation of $K\left(\frac{\|\textbf{u}\|_{\mathbb{R}^p}}{\epsilon_n}\right)$ by a function $K^*_{\epsilon_n}(\textbf{u})=\sum_{i=1}^{N_1} a_i \bigchi_{A_i}(\textbf{u})$ with the following properties:
\begin{enumerate}
\item
$0 \leq |a_1|, \cdots, |a_{N_1}| \leq K_{sup}$ and $N_1 \leq \mathcal{N}(\eta,\gamma)$. 
\item
$\{A_i\}$ are disjoint sets and $A_i=R_i \cap \iota(M)$, where $\{R_i\}$ are  rectangles in $Q_{\epsilon_n \eta}$ defined as in Definition \ref{def of rectangle, ball and cube}.
\item For $\textbf{u} \in Q_{\epsilon_n\eta} \cap \iota(M)$, we have
\[
\left|K^*_{\epsilon_n}(\textbf{u})-K\left(\frac{\|\textbf{u}\|_{\mathbb{R}^p}}{\epsilon_n}\right)\right|<(\ceil{D_{lip}}+1)^d \gamma,
\]  
except on a set $D \subset  Q_{\epsilon_n\eta} \cap \iota(M)$.

\item 
$D \subset \sqcup_{i=1}^{N_2} A'_i $, $\{A'_i\}$ are disjoint sets and $A'_i=R'_i \cap \iota(M)$  and $\{R'_i\}$ are  rectangles in $Q_{\epsilon_n \eta}$ defined as in Definition \ref{def of rectangle, ball and cube}.
\item
$N_2 \leq \mathcal{N}(\eta,\gamma)$ and $\sum_{i=1}^{N_2} Vol(A'_i) <  [C_{lip}\epsilon_n (\ceil{D_{lip}}+1)] ^d \gamma$.
\end{enumerate}
\end{lemma}

\begin{proof}
For $\textbf{v} \in (-\eta,\eta]^d$, let
\begin{align}
J_{x,\epsilon_n}(\textbf{v}):=\frac{1}{\epsilon_n}J_x(\epsilon_n\textbf{v}).
\end{align}
Note that $J_{x,\epsilon_n}$ is well defined by the choice of $\eta$ and $\epsilon_n$.
A key observation is that for $\textbf{u} \in  Q_{\epsilon_n\eta} \cap \iota(M)$, 
\begin{align}
K\left(\frac{\|\textbf{u}\|_{\mathbb{R}^p}}{\epsilon_n}\right)=K\left(\left\|J_{x,\epsilon_n}\left(\frac{I_x(\textbf{u})}{\epsilon_n}\right)\right\|_{\mathbb{R}^d}\right)\,.
\end{align}
Moreover, both $J_{x,\epsilon_n}$ and $I_x$ are  bi-Lipschitz homeomorphism. Since $J_{x,\epsilon_n}\left(\frac{I_x(\textbf{u})}{\epsilon_n}\right)$ is smooth except at one point, we are able to apply the change of variables formula later.

Let us start with approximating $K(\|J_{x,\epsilon_n}(\textbf{v})\|_{\mathbb{R}^d})$ for $\textbf{v} \in (-\eta, \eta]^d$. 
The construction is based on a development of the proof of the well known statement ``a function is Riemann integrable if and only if the set where the function is discontinuous has measure 0'', while taking the map $J_{x,\epsilon_n}$ into account. 
Based on the scaling argument and the fact that  $J_{x}$ is bi-Lipschitz homeomorphism with a bi-Lipschitz constant bounded by $D_{lip} \geq 1$, we have the follwing observations about $J_{x,\epsilon_n}$,
\begin{enumerate}
\item
$J_{x,\epsilon_n}$ is a bi-Lipschitz homeomorphism with the same bi-Lipschitz constant  $D_{lip}$ as $J_x$.
\item
$[-\eta, \eta]^d \subset J_{x,\epsilon_n}([-\eta, \eta]^d) \subset [-D_{lip}\eta, D_{lip}\eta]^d$.
\end{enumerate}
Note that $K(\|\textbf{v}\|_{\mathbb{R}^d})$ is integrable on $[-D_{lip}\eta, D_{lip}\eta]^d$. We partite $[-D_{lip}\eta, D_{lip}\eta]^d$ uniformly into $\mathcal{N}(\eta,\gamma)$ disjoint cubes $\{\mathcal{Q}_i\}_{i=1}^{\mathcal{N}(\eta,\gamma)}$, where each $\mathcal{Q}_i$ is of the form $[v_1,v_1+a] \times \cdots \times [v_d,v_d+a]$ for $a>0$ so that 
\begin{align}
\sum_{i =1}^{\mathcal{N}(\eta,\gamma)}(M_i (K)-m_i (K))Vol(\mathcal{Q}_i) <\gamma^2\,,
\end{align} 
where $M_i (K)$ and $m_i (K)$ are  the supremum and infimum of $K(\|\textbf{v}\|_{\mathbb{R}^d})$ over the cube $\mathcal{Q}_i$  respectively. Let $S_1$ denote the set of cubes $\mathcal{Q}_i$  where $M_i (K)-m_i (K) <\gamma$ and let $S_2$ denote the set of cubes $\mathcal{Q}_j$  where  where $M_j (K)-m_j (K) \geq \gamma$. Clearly, $\max\{|S_1|, |S_2|\}\leq \mathcal{N}(\eta,\gamma)$ and
\begin{align}
\sum_{\mathcal{Q}_i \in S_1 }(M_i (K)-m_i (K))Vol(\mathcal{Q}_i) + \sum_{\mathcal{Q}_j \in S_2 }(M_j (K)-m_j (K))Vol(\mathcal{Q}_j)<\gamma^2.
\end{align} 
Hence, $ \sum_{\mathcal{Q}_j \in S_2} Vol(\mathcal{Q}_j) <\gamma$. 

Denote $S_3:=\{\mathcal{Q}_j \in S_2 |\, \mathcal{Q}_j \cap J_{x,\epsilon_n}([-\eta, \eta]^d) \not= \emptyset\}$. Denote $\mathcal{S}:=\cup_{\mathcal{Q}_j \in S_3 }\mathcal{Q}_j \cap J_{x,\epsilon_n}((-\eta, \eta]^d)$. Since $diam(\mathcal{Q}_j \cap J_{x,\epsilon_n}((-\eta, \eta]^d)) \leq diam(\mathcal{Q}_j)$, we have $diam\big( J^{-1}_{x,\epsilon_n}(\mathcal{Q}_j \cap J_{x,\epsilon_n}((-\eta, \eta]^d))\big) \leq D_{lip} diam (\mathcal{Q}_j)$. Note that the diameters of the above sets are measured by the canonical metric of $\mathbb{R}^d$.  Therefore, each $J^{-1}_{x,\epsilon_n}(\mathcal{Q}_j \cap J_{x,\epsilon_n}([-\eta, \eta]^d))\big)$ can be covered by no more than $(\ceil{D_{lip}}+1)^d$ cubes $\mathcal{Q}_j$.

Let $S_4$ denote the set of cubes that cover $ J^{-1}_{x,\epsilon_n}(\mathcal{S})$. Then $|S_4| \leq (\ceil{D_{lip}}+1)^d |S_3| \leq (\ceil{D_{lip}}+1)^d |S_2|$, and 
\begin{align}
 \sum_{\mathcal{Q}_j \in S_4} Vol(\mathcal{Q}_j) < (\ceil{D_{lip}}+1)^d \gamma.
\end{align}

Let $\mathcal{L}:=\{L_i:=\mathcal{Q}_i \cap [-\eta, \eta]^d\}$. Then, each $L_i$ is a closed rectangle in $\mathbb{R}^d$. Set $\mathcal{L}=\mathcal{L}_1 \cup \mathcal{L}_2$, where $ \mathcal{L}_2=\{L_i=\mathcal{Q}_i \cap [-\eta, \eta]^d|\, \mathcal{Q}_i \in S_4\}$ and $\mathcal{L}_1:=\mathcal{L}\backslash \mathcal{L}_2$. By construction, the union of $L_i\in \mathcal{L}_2$ covers $ J^{-1}_{x,\epsilon_n}(\mathcal{S})$. Moreover,
\begin{align}
 \sum_{L_i \in \mathcal{L}_2} Vol(L_i) \leq  \sum_{\mathcal{Q}_j \in S_4} Vol(\mathcal{Q}_j) < (\ceil{D_{lip}}+1)^d \gamma.
\end{align}
For each $L_i \in \mathcal{L}_1$, $ J_{x,\epsilon_n}(L_i) \subset \mathcal{S}^c \cap J_{x,\epsilon_n}([-\eta, \eta]^d)$. Hence $J_{x,\epsilon_n}(L_i)$ can be covered by $\mathcal{Q}_j$ in $S_1$, and hence $diam(J_{x,\epsilon_n}(L_i))  \leq D_{lip} diam (L_i)  \leq D_{lip} diam (\mathcal{Q}_j)$. Again, the diameters of the above sets are measured by the canonical metric of $\mathbb{R}^d$. Therefore, $J_{x,\epsilon_n}(L_i)$ can be covered by no more than $(\ceil{D_{lip}}+1)^d$ cubes in $S_1$.  Denote $M_i(K,J)$ and $m_i(K,J)$ be the supremum and infimum of  $K(\|J_{x,\epsilon_n}(\textbf{v})\|_{\mathbb{R}^d})$ over $L_i \in \mathcal{L}_1$  respectively. Since $ J_{x,\epsilon_n}(L_i)$ is a path connected set, by the triangle inequality, we have
\begin{align}
M_i(K,J)-m_i(K,J)< (\ceil{D_{lip}}+1)^d \gamma,
\end{align}
for $L_i \in \mathcal{L}_1$. 

Note that each $L_i$ is of the form $[v^i_1,v_1+a^i_1] \times \cdots \times [v^i_d,v^i_d+a^i_d]$. Define $\tilde{L}_i=(v^i_1,v_1+a^i_1] \times \cdots \times (v^i_d,v^i_d+a^i_d]$. For $\textbf{v} \in (-\eta, \eta]^d$, define
\begin{align}
K^*_{J_{x,\epsilon_n}}(\textbf{v})=\sum m_i(K,J)  \bigchi_{\tilde{L}_i}(\textbf{v}),
\end{align}
where the summation is over all the $\tilde{L}_i$ such that $L_i \in \mathcal{L}_1$ and $\tilde{L}_i \cap (-\eta, \eta]^d \not=\emptyset$.
Then, $|K^*_{J_{x,\epsilon_n}}(\textbf{v})-K(\|J_{x,\epsilon_n}(\textbf{v})\|_{\mathbb{R}^d})|< (\ceil{D_{lip}}+1)^d \gamma$ on $(-\eta,\eta]^d$, except on the set 
$\cup_{L_i \in  \mathcal{L}_2} \tilde{L}_i$. And $ \sum_{L_i \in \mathcal{L}_2} Vol(\tilde{L}_i) < (\ceil{D_{lip}}+1)^d \gamma$.

We define 
\begin{align}
K^*_{\epsilon_n}(\textbf{u})=K^*_{J_{x,\epsilon_n}}\left(\frac{I_x(\textbf{u})}{\epsilon_n}\right)\,,
\end{align}
where $\textbf{u} \in Q_{\epsilon_n\eta} \cap \iota(M)$. If we use a scaling argument, then we can show that $K^*_{\epsilon_n}(\textbf{u})=\sum_{i=1}^{N_1} a_i \bigchi_{A_i}(\textbf{u})$ and it satisfies the following properties:

\begin{enumerate}
\item
$0 \leq |a_1|, \cdots, |a_{N_1}| \leq K_{sup}$ and $N_1 \leq \mathcal{N}(\eta,\gamma)$. 
\item
$\{A_i\}$ are disjoint sets and $A_i=R_i \cap \iota(M)$, where $\{R_i\}$ are  rectangles in $Q_{\epsilon_n \eta}$ defined as in Definition \ref{def of rectangle, ball and cube}.
\item For $\textbf{u} \in Q_{\epsilon_n\eta} \cap \iota(M)$, we have
\[
\left|K^*_{\epsilon_n}(\textbf{u})-K\left(\frac{\|\textbf{u}\|_{\mathbb{R}^p}}{\epsilon_n}\right)\right|<(\ceil{D_{lip}}+1)^d \gamma,
\]  
except on a set $D \subset  Q_{\epsilon_n\eta} \cap \iota(M)$.
\item 
$D \subset \sqcup_{i=1}^{N_2} A'_i $, $\{A'_i\}$ are disjoint sets and $A'_i=R'_i \cap \iota(M)$  and $\{R'_i\}$ are  rectangles in $Q_{\epsilon_n \eta}$ defined as in Definition \ref{def of rectangle, ball and cube}.
\item
$N_2 \leq \mathcal{N}(\eta,\gamma)$ and $\sum_{i=1}^{N_2} Vol(A'_i) <  [C_{lip}\epsilon_n (\ceil{D_{lip}}+1)] ^d \gamma$.
\end{enumerate}
Note that we also use the fact that $I_x$ is $C_{lip}$ bi-Lipschitz to prove (5).
\end{proof}

\subsection{\textbf{Proof of Theorem \ref{EK_n-K_n}}}

We fix $x \in M$ and show the conclusion holds for any $x$. Without loss of generality, we assume that $\iota(M)$ is rotated and translated, so that $\iota(x)=0 \in \mathbb{R}^p$ and $\iota_*T_x M$ is the subpace $\mathbb{R}^d \subset \mathbb{R}^p$ generated by $e_1, \cdots, e_d$.

\underline{\textbf{Choice of $\lambda$}}

{Choose $\lambda$ as described in the previous construction of $J_x$ \eqref{map J_x}.} Moreover, it also satisfies the following bound
\begin{align}\label{relation between P and lambda}
P_{Max}[2C_{lip}\lambda]^d \leq \frac{1}{4}\,.
\end{align}

\underline{\textbf{Choice of $\eta$ based on $\gamma$}}

Suppose $0<\gamma<\min\{\rho^{-\alpha},1\}$.  We  choose $\eta=\gamma^{-\frac{1}{\alpha}}\geq \rho$. If $K(t)$ has a compact support on $[0, \rho]$, we simply require $\eta = \rho$. Let $T(t)=\frac{1}{t^\alpha}$ for $t>0$. Then, by the assumption we immediately have the following statements:
\begin{enumerate}
\item $|K(t)| \leq  \gamma$ whenever $t \geq \eta$.
\item 
\begin{align}\label{property of T(t)}
\int_{\big(B^{\mathbb{R}^d}_{\eta}(0)\big)^c}T(\|\textbf{v}\|_{\mathbb{R}^d})d\textbf{v} \leq \frac{|S^{d-1}|}{\alpha-d}\gamma^{1-\frac{d}{\alpha}},
\end{align}
where $B^{\mathbb{R}^d}_{\eta}(0)$ is ball of radius $\eta$ {and $|S^{d-1}|$ is the $d-1$ volume of the sphere $S^{d-1}$.}
\end{enumerate}

\underline{\textbf{Choice of $\epsilon_n$ based on $\gamma$}}

We choose $\epsilon_n$ small enough, so that  
\begin{align}\label{relation between gamma and epsilon}
8 \sqrt{p} \epsilon_n^{1-\frac{d}{\alpha}} \gamma^{-\frac{1}{\alpha}} <\lambda. 
\end{align}
Note that this automatically implies $\epsilon_n \eta \leq \lambda$. If $K(t)$ has compact support on $[0, \rho]$, we simply require $\epsilon_n \eta =\epsilon_n \rho \leq \lambda$.
In other words, by letting $\mathcal{D}_3= \left(\frac{\lambda}{8 \sqrt{p}}\right)^{\frac{\alpha}{\alpha-d}}$, we have
\begin{align}
\epsilon_n \leq \mathcal{D}_3 \gamma^{\frac{1}{\alpha-d}}.
\end{align}
Similarly, we take $\mathcal{D}_3=\frac{\lambda}{\rho}$ when $K(t)$ has compact support, and hence $\epsilon_n \leq \mathcal{D}_3$.

\
We start to analyze the term $|\mathbb{E}K_n(x)-K_n(x)|$ for the previously fixed $x$. Note that we will use the fact that $\iota(x)=0$.
\begin{align} 
& |\mathbb{E}K_n(x)-K_n(x)|  \label{inside and outside the cube} \\
=& \left|\int_M\frac{1}{\epsilon_n^d} K(\frac{\|\iota(y)-\iota(x)\|_{\mathbb{R}^p}}{\epsilon_n})P(y)dV(y)-\int_M \frac{1}{\epsilon_n^d}K(\frac{\|\iota(y)-\iota(x)\|_{\mathbb{R}^p}}{\epsilon_n})dP_n(y)\right| \nonumber \\
=& \left|\int_{\iota(M)} \frac{1}{\epsilon_n^d} K(\frac{\|\textbf{u}\|_{\mathbb{R}^p}}{\epsilon_n})P\circ \iota^{-1}(\textbf{u})\iota_*dV(\textbf{u})-\int_{\iota(M)} \frac{1}{\epsilon_n^d} K(\frac{\|\textbf{u}\|_{\mathbb{R}^p}}{\epsilon_n})d(P\circ \iota^{-1})_n(\textbf{u})\right| \nonumber\\
\leq & \left|\int_{Q_{\epsilon_n\eta} \cap \iota(M)}\frac{1}{\epsilon_n^d} K(\frac{\|\textbf{u}\|_{\mathbb{R}^p}}{\epsilon_n})P\circ \iota^{-1}(\textbf{u})\iota_*dV(\textbf{u})-\int_{Q_{\epsilon_n\eta} \cap \iota(M)}\frac{1}{\epsilon_n^d} K(\frac{\|\textbf{u}\|_{\mathbb{R}^p}}{\epsilon_n})d(P\circ \iota^{-1})_n(\textbf{u})\right| \nonumber\\
&+\left|\int_{(Q_{\epsilon_n\eta})^c \cap \iota(M)}\frac{1}{\epsilon_n^d} K(\frac{\|\textbf{u}\|_{\mathbb{R}^p}}{\epsilon_n})P\circ \iota^{-1}(\textbf{u})\iota_*dV(\textbf{u})-\int_{(Q_{\epsilon_n\eta})^c \cap \iota(M)} \frac{1}{\epsilon_n^d}K(\frac{\|\textbf{u}\|_{\mathbb{R}^p}}{\epsilon_n})d(P\circ \iota^{-1})_n(\textbf{u})\right|\,. \nonumber
\end{align}

{In other words, in order to bound the variance $|\mathbb{E}K_n(x)-K_n(x)|$, we need to bound two parts, the local part
$$\left|\int_{Q_{\epsilon_n\eta} \cap \iota(M)}\frac{1}{\epsilon_n^d} K(\frac{\|\textbf{u}\|_{\mathbb{R}^p}}{\epsilon_n})P\circ \iota^{-1}(\textbf{u})\iota_*dV(\textbf{u})-\int_{Q_{\epsilon_n\eta} \cap \iota(M)}\frac{1}{\epsilon_n^d} K(\frac{\|\textbf{u}\|_{\mathbb{R}^p}}{\epsilon_n})d(P\circ \iota^{-1})_n(\textbf{u})\right|,$$
and the tail part
$$\left|\int_{(Q_{\epsilon_n\eta})^c \cap \iota(M)}\frac{1}{\epsilon_n^d} K(\frac{\|\textbf{u}\|_{\mathbb{R}^p}}{\epsilon_n})P\circ \iota^{-1}(\textbf{u})\iota_*dV(\textbf{u})-\int_{(Q_{\epsilon_n\eta})^c \cap \iota(M)} \frac{1}{\epsilon_n^d}K(\frac{\|\textbf{u}\|_{\mathbb{R}^p}}{\epsilon_n})d(P\circ \iota^{-1})_n(\textbf{u})\right|.$$

We summarize the the bound of the local part in the following lemma. 

\begin{lemma}[Bound of the local part in \eqref{inside and outside the cube}]\label{bound of the local part}
\begin{align}
&  \left|\int_{Q_{\eta} \cap \iota(M)} K(\frac{\|\textbf{u}\|_{\mathbb{R}^p}}{\epsilon_n})P\circ \iota^{-1}(\textbf{u})\iota_*dV(\textbf{u})-\int_{Q_{\eta} \cap \iota(M)} K(\frac{\|\textbf{u}\|_{\mathbb{R}^p}}{\epsilon_n})d(P\circ \iota^{-1})_n(\textbf{u})\right| \\
\leq & \mathcal{C}_1  \gamma^{1-\frac{d}{\alpha}}+\frac{3N(\gamma)+(\ceil{D_{lip}}+1)^d \gamma}{\epsilon_n^d} K_{\sup} \sup_{A \in \mathcal{R}_{2\sqrt{p}\epsilon_n \eta}(\iota(M))} |\mu_n(A)-\mu(A)|,\nonumber 
\end{align}
where $\mathcal{C}_1= 2^{d+2} C_{lip}^d (\ceil{D_{lip}}+1)^d P_{Max}(K_{sup}+1)$.
\end{lemma}}

\begin{proof}
To bound the local part,  we use $K^*_{\epsilon_n}(\textbf{u})$ in the Lemma \ref{approximation to K} to approximate $K(\frac{\|\textbf{u}\|_{\mathbb{R}^p}}{\epsilon_n})$. Specifically, we find the smallest integer $N(\gamma)$ so that we partitie $[-D_{lip}\gamma^{-\frac{1}{\alpha}}, D_{lip}\gamma^{-\frac{1}{\alpha}}]^d$ uniformly into $N(\gamma)$ cubes $\{\mathcal{Q}_i\}_{i=1}^{N(\gamma)}$, where each $\mathcal{Q}_i$ is of the form $[v_1,v_1+a] \times \cdots \times [v_d,v_d+a]$ for $a>0$ and
\begin{align}
\sum_{i =1}^{N(\gamma)}(M_i (K)-m_i (K))Vol(\mathcal{Q}_i) <\gamma^2\,,
\end{align} 
where $M_i (K)$ and $m_i (K)$ are  the supremum and infimum of $K(\|\textbf{v}\|_{\mathbb{R}^d})$ over the cube $\mathcal{Q}_i$  respectively. By Lemma \ref{approximation to K}, for $\textbf{u} \in Q_{\epsilon_n\eta} \cap \iota(M)$, there is an approximation of $K\left(\frac{\|\textbf{u}\|_{\mathbb{R}^p}}{\epsilon_n}\right)$ by a function $K^*_{\epsilon_n}(\textbf{u})=\sum_{i=1}^{N_1} a_i \bigchi_{A_i}(\textbf{u})$ with the following properties:
\begin{enumerate}
\item
$0 \leq |a_1|, \cdots, |a_{N_1}| \leq K_{sup}$ and $N_1 \leq N(\gamma)$. 
\item
$\{A_i\}$ are disjoint sets and $A_i=R_i \cap \iota(M)$, where $\{R_i\}$ are  rectangles in $Q_{\epsilon_n \eta}$ defined as in Definition \ref{def of rectangle, ball and cube}.
\item For $\textbf{u} \in Q_{\epsilon_n\eta} \cap \iota(M)$, we have
\[
\left|K^*_{\epsilon_n}(\textbf{u})-K\left(\frac{\|\textbf{u}\|_{\mathbb{R}^p}}{\epsilon_n}\right)\right|<(\ceil{D_{lip}}+1)^d \gamma,
\]  
except on a set $D \subset  Q_{\epsilon_n\eta} \cap \iota(M)$.
\item 
$D \subset \sqcup_{i=1}^{N_2} A'_i $, $\{A'_i\}$ are disjoint sets and $A'_i=R'_i \cap \iota(M)$  and $\{R'_i\}$ are  rectangles in $Q_{\epsilon_n \eta}$ defined as in Definition \ref{def of rectangle, ball and cube}.
\item
$N_2 \leq N(\gamma)$ and $\sum_{i=1}^{N_2} Vol(A'_i) <  [C_{lip}\epsilon_n (\ceil{D_{lip}}+1)] ^d \gamma$.
\end{enumerate}

By using $K^*_{\epsilon_n}(\textbf{u})$,  the local part can be further decomposed into the following three parts:
\begin{align}
& \left|\int_{Q_{\epsilon_n\eta} \cap \iota(M)} \frac{1}{\epsilon_n^d} K(\frac{\|\textbf{u}\|_{\mathbb{R}^p}}{\epsilon_n})P\circ \iota^{-1}(\textbf{u})\iota_*dV(\textbf{u})-\int_{Q_{\epsilon_n\eta} \cap \iota(M)} \frac{1}{\epsilon_n^d} K(\frac{\|\textbf{u}\|_{\mathbb{R}^p}}{\epsilon_n})d(P\circ \iota^{-1})_n(\textbf{u})\right|\nonumber \\
\leq &\left|  \int_{Q_{\epsilon_n\eta} \cap \iota(M)}\bigg[\frac{1}{\epsilon_n^d}K(\frac{\|\textbf{u}\|_{\mathbb{R}^p}}{\epsilon_n})-\frac{1}{\epsilon_n^d}K^*_{\epsilon_n}(\textbf{u})\bigg]P\circ \iota^{-1}(\textbf{u})\iota_*dV(\textbf{u})\right|\label{Riemann approx step} \\
&+\left|\int_{Q_{\epsilon_n\eta} \cap \iota(M)}\frac{1}{\epsilon_n^d} K^*_{\epsilon_n}(\textbf{u}) P\circ \iota^{-1}(\textbf{u})\iota_*dV(\textbf{u}) -\int_{Q_{\epsilon_n\eta} \cap \iota(M)}\frac{1}{\epsilon_n^d}K^*_{\epsilon_n}(\textbf{u})d(P\circ \iota^{-1})_n(\textbf{u})\right| \nonumber \\
&+\left| \int_{Q_{\epsilon_n\eta} \cap \iota(M)}\frac{1}{\epsilon_n^d}\bigg[ K(\frac{\|\textbf{u}\|_{\mathbb{R}^p}}{\epsilon_n})-K^*_{\epsilon_n}(\textbf{u})\bigg]d(P\circ \iota^{-1})_n(\textbf{u})\right|\,.\nonumber 
\end{align}
We bound the first  term  in \eqref{Riemann approx step}: 
\begin{align}\label{Riemann approx step 1}
&\left|\int_{Q_{\epsilon_n \eta} \cap \iota(M)}\frac{1}{\epsilon_n^d} \bigg[ K(\frac{\|\textbf{u}\|_{\mathbb{R}^p}}{\epsilon_n})-K^*_{\epsilon_n}(\textbf{u})\bigg] P\circ \iota^{-1}(\textbf{u})\iota_*dV(\textbf{u}) \right|\\
\leq & \int_{Q_{\epsilon_n \eta} \cap \iota(M) \setminus D}\frac{1}{\epsilon_n^d} \left| K(\frac{\|\textbf{u}\|_{\mathbb{R}^p}}{\epsilon_n})-K^*_{\epsilon_n}(\textbf{u})\right| P\circ \iota^{-1}(\textbf{u})\iota_*dV(\textbf{u}) + \int_{D}\frac{1}{\epsilon_n^d} \left| K(\frac{\|\textbf{u}\|_{\mathbb{R}^p}}{\epsilon_n})-K^*_{\epsilon_n}(\textbf{u}) \right| P\circ \iota^{-1}(\textbf{u})\iota_*dV(\textbf{u})  \nonumber  \\
\leq &  Vol(Q_{\epsilon_n \eta} \cap \iota(M)) \frac{1}{\epsilon_n^d} (\ceil{D_{lip}}+1)^d \gamma P_{Max}+Vol(D) \frac{1}{\epsilon_n^d}  2K_{sup} P_{Max} \nonumber \\
\leq & C_{lip}^d(2\eta)^d  (\ceil{D_{lip}}+1)^d \gamma P_{Max}+ 2 C_{lip}^d (\ceil{D_{lip}}+1)^d \gamma K_{sup} P_{Max} \nonumber \\
\leq &   2^{d+1} C_{lip}^d (\ceil{D_{lip}}+1)^d P_{Max}(K_{sup}+1) \gamma^{1-\frac{d}{\alpha}}. \nonumber
\end{align} 
Note that in the second last step of the above equation, we use the bi-Lipschitz property of $I_x^{-1}$ to estimate the volume. Hence, $Vol(Q_{\epsilon_n \eta} \cap \iota(M) ) \leq C_{lip}^d (2\epsilon_n\eta)^d$ and $Vol(D) \leq \sum_{i=1}^{N_2} Vol(A'_i) <  [C_{lip}\epsilon_n (\ceil{D_{lip}}+1)] ^d \gamma$ . 
For the second term in \eqref{Riemann approx step},
\begin{align}\label{Riemann approx step 2}
& \left|\int_{Q_{\epsilon_n\eta} \cap \iota(M)}\frac{1}{\epsilon_n^d} K^*_{\epsilon_n}(\textbf{u}) P\circ \iota^{-1}(\textbf{u})\iota_*dV(\textbf{u}) -\int_{Q_{\epsilon_n\eta} \cap \iota(M)}\frac{1}{\epsilon_n^d}K^*_{\epsilon_n}(\textbf{u})d(P\circ \iota^{-1})_n(\textbf{u})\right| \\
\leq &\frac{1}{\epsilon_n^d} \left|\sum_{i=1}^{N_1}a_i(\mu_n(A_i)-\mu(A_i))\right|
\leq  \frac{N (\gamma)}{\epsilon_n^d}  K_{sup} \sup_{A \in \mathcal{R}_{2\sqrt{p}\epsilon_n \eta}(\iota(M))} |\mu_n(A)-\mu(A)|.  \nonumber
\end{align}
Note that $A_i= R_i \cap \iota(M)$, where $R_i$ is a rectangle in $ Q_{\epsilon_n\eta}$ and $diam(R_i) \leq diam( Q_{\epsilon_n\eta})=2\sqrt{p}\epsilon_n\eta$. Therefore, the supremum in the last step is taken over $A \in \mathcal{R}_{2\sqrt{p}\epsilon_n \eta}(\iota(M))$.

For the last term in \eqref{Riemann approx step}, again, we use the bi-Lipschitz property of $I_x^{-1}$ to estimate the volume. So, we have
\begin{align}\label{Riemann approx step 3}
& \left| \int_{Q_{\epsilon_n\eta} \cap \iota(M)}\frac{1}{\epsilon_n^d}\bigg[ K(\frac{\|\textbf{u}\|_{\mathbb{R}^p}}{\epsilon_n})-K^*_{\epsilon_n}(\textbf{u})\bigg]d(P\circ \iota^{-1})_n(\textbf{u})\right| \\
\leq &  \int_{Q_{\epsilon_n \eta} \cap \iota(M) \setminus D}\frac{1}{\epsilon_n^d} \Big| K(\frac{\|\textbf{u}\|_{\mathbb{R}^p}}{\epsilon_n})-K^*_{\epsilon_n}(\textbf{u})\Big| d(P\circ \iota^{-1})_n(\textbf{u}) + \int_{D}\frac{1}{\epsilon_n^d} \Big| K(\frac{\|\textbf{u}\|_{\mathbb{R}^p}}{\epsilon_n})-K^*_{\epsilon_n}(\textbf{u}) \Big| d(P\circ \iota^{-1})_n(\textbf{u})  \nonumber  \\
\leq & \frac{1}{\epsilon_n^d} (\ceil{D_{lip}}+1)^d \gamma \mu_n(Q_{\epsilon_n \eta} \cap \iota(M)) +2 K_{Sup}\frac{1}{\epsilon_n^d} \mu_n(D)  \nonumber \\
\leq &   \frac{1}{\epsilon_n^d}  (\ceil{D_{lip}}+1)^d \gamma| \mu_n(Q_{\epsilon_n \eta} \cap \iota(M))-\mu(Q_{\epsilon_n \eta} \cap \iota(M)) |+\frac{1}{\epsilon_n^d}  (\ceil{D_{lip}}+1)^d \gamma \mu(Q_{\epsilon_n \eta} \cap \iota(M)) \nonumber \\
& +2 K_{Sup}\frac{1}{\epsilon_n^d}| \mu_n(D)-\mu(D)|+2 K_{Sup}\frac{1}{\epsilon_n^d}\mu(D)  \nonumber \\
\leq & \frac{  (\ceil{D_{lip}}+1)^d \gamma}{\epsilon_n^d}\sup_{A \in \mathcal{R}_{2\sqrt{p}\epsilon_n \eta}(\iota(M))} |\mu_n(A)-\mu(A)|+\frac{1}{\epsilon_n^d}  (\ceil{D_{lip}}+1)^d \gamma P_{Max} C_{lip}^d (2\epsilon_n \eta)^d\nonumber \\
&+ \frac{2}{\epsilon_n^d} K_{\sup}|\sum_{i=1}^{N_2}\mu_n(A'_i)-\sum_{i=1}^{N_2}\mu(A'_i)|+2K_{\sup}P_{Max} [C_{lip} (\ceil{D_{lip}}+1)] ^d \gamma \nonumber \\
\leq & \frac{  (\ceil{D_{lip}}+1)^d \gamma}{\epsilon_n^d}\sup_{A \in \mathcal{R}_{2\sqrt{p}\epsilon_n \eta}(\iota(M))} |\mu_n(A)-\mu(A)|+ [2C_{lip}(\ceil{D_{lip}}+1)]^d P_{Max}  \gamma^{1-\frac{d}{\alpha}}\nonumber \\
&+ \frac{2N(\gamma)}{\epsilon_n^d} K_{\sup}\sup_{A \in \mathcal{R}_{2\sqrt{p}\epsilon_n \eta}(\iota(M))} |\mu_n(A)-\mu(A)|+2K_{\sup}P_{Max} [C_{lip} (\ceil{D_{lip}}+1)] ^d \gamma. \nonumber
\end{align}

If we plug the bounds \eqref{Riemann approx step 1}, \eqref{Riemann approx step 2}, and \eqref{Riemann approx step 3} into \eqref{Riemann approx step}, we have
\begin{align}\label{inside the cube}
&  \left|\int_{Q_{\eta} \cap \iota(M)} K(\frac{\|\textbf{u}\|_{\mathbb{R}^p}}{\epsilon_n})P\circ \iota^{-1}(\textbf{u})\iota_*dV(\textbf{u})-\int_{Q_{\eta} \cap \iota(M)} K(\frac{\|\textbf{u}\|_{\mathbb{R}^p}}{\epsilon_n})d(P\circ \iota^{-1})_n(\textbf{u})\right| \\
\leq & \Big( 2^{d+1} C_{lip}^d (\ceil{D_{lip}}+1)^d P_{Max}(K_{sup}+1) + [2C_{lip}(\ceil{D_{lip}}+1)]^d P_{Max}  +2K_{\sup}P_{Max} [C_{lip} (\ceil{D_{lip}}+1)] ^d \Big) \gamma^{1-\frac{d}{\alpha}} \nonumber \\
& +\frac{3N(\gamma)+(\ceil{D_{lip}}+1)^d \gamma}{\epsilon_n^d} K_{\sup}\sup_{A \in \mathcal{R}_{2\sqrt{p}\epsilon_n \eta}(\iota(M))} |\mu_n(A)-\mu(A)|  \nonumber \\
\leq & \mathcal{C}_1  \gamma^{1-\frac{d}{\alpha}}+\frac{3N(\gamma)+(\ceil{D_{lip}}+1)^d \gamma}{\epsilon_n^d} K_{\sup} \sup_{A \in \mathcal{R}_{2\sqrt{p}\epsilon_n \eta}(\iota(M))} |\mu_n(A)-\mu(A)|,\nonumber 
\end{align}
where $\mathcal{C}_1= 2^{d+2} C_{lip}^d (\ceil{D_{lip}}+1)^d P_{Max}(K_{sup}+1)$.
\end{proof}

{Next, we bound the second part, the tail part, in \eqref{inside and outside the cube}.  We summarize the result as the following lemma.

\begin{lemma}[Bound of the tail part in \eqref{inside and outside the cube}]\label{bound of the tail part}
Let $q$ be an integer so that $\frac{1}{\epsilon_n^d} < q <\frac{2}{\epsilon_n^d}$. Let  $\eta_i=[(1-\frac{i}{q})\gamma]^{-\frac{1}{\alpha}}$. 
\begin{align}
& \left|\int_{(Q_{\epsilon_n\eta})^c \cap \iota(M)}\frac{1}{\epsilon_n^d} K(\frac{\|\textbf{u}\|_{\mathbb{R}^p}}{\epsilon_n})P\circ \iota^{-1}(\textbf{u})\iota_*dV(\textbf{u})-\int_{(Q_{\epsilon_n\eta})^c \cap \iota(M)} \frac{1}{\epsilon_n^d}K(\frac{\|\textbf{u}\|_{\mathbb{R}^p}}{\epsilon_n})d(P\circ \iota^{-1})_n(\textbf{u})\right|\nonumber \\
\leq & \mathcal{C}_2 \gamma^{1-\frac{d}{\alpha}}+\frac{2\gamma}{\epsilon_n^d} \sup_{A \in \mathcal{R}_{2\sqrt{p}\epsilon_n \eta_{q-1}}(\iota(M))} |\mu_n(A)-\mu(A)|\,,
\end{align}
where $\mathcal{C}_2:= 2p^{\frac{\alpha}{2}} P_{Max}  (C_{lip} D_{lip})^d \frac{|S^{d-1}|}{\alpha-d}+4$.
\end{lemma}
Note that when $K(t)$ has a compact support on $[0, \rho]$, our choice $\eta = \rho$ implies that the tail part vanishes.}

\begin{proof}
Recall that $T(t)=\frac{1}{t^\alpha}$. For $\textbf{u}=(u_1,\cdots,u_p)$, let $\tilde{u}=\max\{|u_1|,\cdots,|u_p|\}$, define $\tilde{T}(\textbf{u})=\frac{1}{\tilde{u}^\alpha}$. Since  $p^{-\frac{1}{2}}\|\textbf{u}\|_{\mathbb{R}^p} \leq \tilde{u} \leq \|\textbf{u}\|_{\mathbb{R}^p}$, we have
\begin{align}
|K(\|\textbf{u}\|_{\mathbb{R}^p}) |\leq T(\|\textbf{u}\|_{\mathbb{R}^p}) \leq \tilde{T}(\textbf{u}) \leq p^{\frac{\alpha}{2}}T(\|\textbf{u}\|_{\mathbb{R}^p}),
\end{align}
for all $\|\textbf{u}\|_{\mathbb{R}^p} \geq \eta \geq \rho$.
Thus, 
the second part in \eqref{inside and outside the cube} is
\begin{align}
& \left|\int_{(Q_{\epsilon_n\eta})^c \cap \iota(M)}\frac{1}{\epsilon_n^d} K(\frac{\|\textbf{u}\|_{\mathbb{R}^p}}{\epsilon_n})P\circ \iota^{-1}(\textbf{u})\iota_*dV(\textbf{u})-\int_{(Q_{\epsilon_n\eta})^c \cap \iota(M)} \frac{1}{\epsilon_n^d}K(\frac{\|\textbf{u}\|_{\mathbb{R}^p}}{\epsilon_n})d(P\circ \iota^{-1})_n(\textbf{u})\right|  \nonumber \\
\leq & \int_{(Q_{\epsilon_n\eta})^c \cap \iota(M)}\frac{1}{\epsilon_n^d} | K(\frac{\|\textbf{u}\|_{\mathbb{R}^p}}{\epsilon_n}) |P\circ \iota^{-1}(\textbf{u})\iota_*dV(\textbf{u})+ \int_{(Q_{\epsilon_n\eta})^c \cap \iota(M)} \frac{1}{\epsilon_n^d}|K(\frac{\|\textbf{u}\|_{\mathbb{R}^p}}{\epsilon_n})| d(P\circ \iota^{-1})_n(\textbf{u})   \nonumber \\
\leq & \int_{(Q_{\epsilon_n\eta})^c \cap \iota(M)}\frac{1}{\epsilon_n^d} \tilde{T}(\frac{\textbf{u}}{\epsilon_n})P\circ \iota^{-1}(\textbf{u})\iota_*dV(\textbf{u})+\int_{(Q_{\epsilon_n\eta})^c \cap \iota(M)} \frac{1}{\epsilon_n^d}\tilde{T}(\frac{\textbf{u}}{\epsilon_n})d(P\circ \iota^{-1})_n(\textbf{u}) \nonumber \\
\leq & 2\int_{(Q_{\epsilon_n\eta})^c \cap \iota(M)}\frac{1}{\epsilon_n^d} \tilde{T}(\frac{\textbf{u}}{\epsilon_n})P\circ \iota^{-1}(\textbf{u})\iota_*dV(\textbf{u}) \label{Q complement} \\
& +\left|\int_{(Q_{\epsilon_n\eta})^c \cap \iota(M)}\frac{1}{\epsilon_n^d} \tilde{T}(\frac{\textbf{u}}{\epsilon_n})P\circ \iota^{-1}(\textbf{u})\iota_*dV(\textbf{u})-\int_{(Q_{\epsilon_n\eta})^c \cap \iota(M)} \frac{1}{\epsilon_n^d}\tilde{T}(\frac{\textbf{u}}{\epsilon_n})d(P\circ \iota^{-1})_n(\textbf{u})\right| \,.\nonumber
\end{align}
The first term in \eqref{Q complement} can be bounded by:
\begin{align} 
& \int_{(Q_{\epsilon_n\eta})^c \cap \iota(M)}\frac{1}{\epsilon_n^d} \tilde{T}(\frac{\textbf{u}}{\epsilon_n}) P\circ \iota^{-1}(\textbf{u})\iota_*dV(\textbf{u}) \nonumber \\
\leq  &  p^{\frac{\alpha}{2}} \int_{\big([-\lambda,\lambda]^p \setminus (Q_{\epsilon_n\eta}) \big) \cap \iota(M)}\frac{1}{\epsilon_n^d} T(\frac{\|\textbf{u}\|_{\mathbb{R}^p}}{\epsilon_n})P\circ \iota^{-1}(\textbf{u})\iota_*dV(\textbf{u})\nonumber \\
&\,+  p^{\frac{\alpha}{2}} \int_{\big([-\lambda,\lambda]^p\big)^c \cap \iota(M)}\frac{1}{\epsilon_n^d} T(\frac{\|\textbf{u}\|_{\mathbb{R}^p}}{\epsilon_n})P\circ \iota^{-1}(\textbf{u})\iota_*dV(\textbf{u}) .\label{Q complement 1}
\end{align}
For $\textbf{v} \in J_x \circ I_x \bigg( \big([-\lambda,\lambda]^p \setminus (Q_{\epsilon_n\eta}) \big) \cap \iota(M)\bigg)$,  there is a unique $\textbf{u} \in \big([-\lambda,\lambda]^p \setminus (Q_{\epsilon_n\eta}) \big) \cap \iota(M)$ such that $\textbf{v}=J_x \circ I_x (\textbf{u})$. Moreover, $\|\textbf{u}\|_{\mathbb{R}^p}=\|\textbf{v}\|_{\mathbb{R}^d}$. Since $\|\textbf{u}\|_{\mathbb{R}^p} \geq \epsilon_n \eta$, we have $\|\textbf{v}\|_{\mathbb{R}^d} \geq \epsilon_n \eta$. Hence, we conclude that $J_x \circ I_x \bigg( \big([-\lambda,\lambda]^p \setminus (Q_{\epsilon_n\eta}) \big) \cap \iota(M)\bigg) \subset \big(B^{\mathbb{R}^d}_{ \epsilon_n \eta}(0)\big)^c$.

Now, the first term in \eqref{Q complement 1} can be bounded as: 
\begin{align}
& \int_{\big([-\lambda,\lambda]^p \setminus (Q_{\epsilon_n\eta}) \big) \cap \iota(M)}\frac{1}{\epsilon_n^d} T(\frac{\|\textbf{u}\|_{\mathbb{R}^p}}{\epsilon_n})P\circ \iota^{-1}(\textbf{u})\iota_*dV(\textbf{u}) \\
= & \int_{J_x \circ I_x \bigg( \big([-\lambda,\lambda]^p \setminus (Q_{\epsilon_n\eta}) \big) \cap \iota(M)\bigg)}\frac{1}{\epsilon_n^d} T(\frac{\|I_x^{-1} \circ J^{-1}_x(\textbf{v})\|_{\mathbb{R}^p}}{\epsilon_n})P\circ \iota^{-1}(I_x^{-1} \circ J^{-1}_x(\textbf{v}))\iota_*dV(I_x^{-1} \circ J^{-1}_x(\textbf{v})) \nonumber \\
\leq & P_{Max} \int_{J_x \circ I_x \bigg( \big([-\lambda,\lambda]^p \setminus (Q_{\epsilon_n\eta}) \big) \cap \iota(M)\bigg)}\frac{1}{\epsilon_n^d} T(\frac{\|\textbf{v}\|_{\mathbb{R}^d}}{\epsilon_n})\iota_*dV(I_x^{-1} \circ J^{-1}_x(\textbf{v}))  \nonumber \\
\leq & P_{Max}(C_{lip} D_{lip})^d \int_{\big(B^{\mathbb{R}^d}_{\epsilon_n \eta}(0)\big)^c}\frac{1}{\epsilon_n^d} T(\frac{\|\textbf{v}\|_{\mathbb{R}^d}}{\epsilon_n})d\textbf{v} \nonumber \\
=&  P_{Max} (C_{lip} D_{lip})^d \int_{\big(B^{\mathbb{R}^d}_{\eta}(0)\big)^c} T(\|\textbf{v}\|_{\mathbb{R}^d})d\textbf{v} \leq  P_{Max} (C_{lip} D_{lip})^d \frac{|S^{d-1}|}{\alpha-d}\gamma^{1-\frac{d}{\alpha}} \,.\nonumber 
\end{align}
Note that we use the fact that $I_x^{-1} \circ J^{-1}_x$ is a bi-Lipschitz homeomorphism in the second last step, and we use the property \eqref{property of T(t)} about $\eta$ in the last step.
Then we bound the second term in \eqref{Q complement 1}. In $\big([-\lambda,\lambda]^p\big)^c \cap \iota(M)$, we have $\frac{1}{\epsilon_n^d} T(\frac{\|\textbf{u}\|_{\mathbb{R}^p}}{\epsilon_n}) \leq \frac{\epsilon_n^{\alpha-d}}{\lambda^\alpha}$, therefore,
\begin{align}
\int_{\big([-\lambda,\lambda]^p\big)^c \cap \iota(M)}\frac{1}{\epsilon_n^d} T(\frac{\|\textbf{u}\|_{\mathbb{R}^p}}{\epsilon_n})P\circ \iota^{-1}(\textbf{u})\iota_*dV(\textbf{u}) \leq \frac{\epsilon_n^{\alpha-d}}{\lambda^\alpha} < \frac{1}{(16p)^{\frac{\alpha}{2}}}\gamma,
\end{align}
where we use the relation \eqref{relation between gamma and epsilon} in the last step.

We sum above two terms to bound the first term in \eqref{Q complement} by
\begin{align}\label{Q complement 2}
& 2\int_{(Q_{\epsilon_n\eta})^c \cap \iota(M)}\frac{1}{\epsilon_n^d} \tilde{T}(\frac{\textbf{u}}{\epsilon_n}) P\circ \iota^{-1}(\textbf{u})\iota_*dV(\textbf{u}) \\
\leq & 2p^{\frac{\alpha}{2}} P_{Max}  (C_{lip} D_{lip})^d \frac{|S^{d-1}|}{\alpha-d}\gamma^{1-\frac{d}{\alpha}} + \frac{2}{4^\alpha}\gamma. \nonumber
\end{align}

To bound the second term in \eqref{Q complement}, we are going to approximate $\tilde{T}(\textbf{u})$ over $(Q_{\eta})^c$ by a step function $\tilde{T}^*(\textbf{u})$. Let $q$ be an integer so that $\frac{1}{\epsilon_n^d} < q <\frac{2}{\epsilon_n^d}$. Let 
\begin{align}\label{def of eta i}
\eta_i=[(1-\frac{i}{q})\gamma]^{-\frac{1}{\alpha}}, 
\end{align}
for $i=0, \dots, q-1$. Note that $\eta=\eta_0$. Define 
\begin{align}
\left\{
\begin{array}{ll}
\tilde{T}^*(\textbf{u})= (1-\frac{i}{q}) \gamma  \quad\quad &\mbox{if $\textbf{u} \in Q_{\eta_i} \setminus  Q_{\eta_{i-1}}$ for $i=1, \cdots, q-1$,} \\
\tilde{T}^*(\textbf{u})=0 & \mbox{if $\textbf{u} \in (Q_{\eta_{q-1}})^c $.}
\end{array}
\right.
\end{align}
By the definition, $\tilde{T}(\textbf{u})=\frac{1}{t^\alpha}$ if and only if $\textbf{u}$ is on the boundary of the cube $Q_{2t}$. Hence, if $\textbf{u} \in (Q_{\eta})^c$, $0 \leq \tilde{T}(\textbf{u})-\tilde{T}^*(\textbf{u}) \leq \frac{\gamma}{q}$. The second term in \eqref{Q complement} can be bounded by:
\begin{align} 
& \left|\int_{(Q_{\epsilon_n\eta})^c \cap \iota(M)}\frac{1}{\epsilon_n^d} \tilde{T}(\frac{\textbf{u}}{\epsilon_n})P\circ \iota^{-1}(\textbf{u})\iota_*dV(\textbf{u})-\int_{(Q_{\epsilon_n\eta})^c \cap \iota(M)} \frac{1}{\epsilon_n^d}\tilde{T}(\frac{\textbf{u}}{\epsilon_n})d(P\circ \iota^{-1})_n(\textbf{u})\right| \nonumber\\
\leq & \int_{(Q_{\epsilon_n\eta})^c \cap \iota(M)}\frac{1}{\epsilon_n^d} |\tilde{T}(\frac{\textbf{u}}{\epsilon_n})-\tilde{T}^*(\frac{\textbf{u}}{\epsilon_n})|P\circ \iota^{-1}(\textbf{u})\iota_*dV(\textbf{u}) \label{Q complement 3} \\
&+ \left|\int_{(Q_{\epsilon_n\eta})^c \cap \iota(M)}\frac{1}{\epsilon_n^d} \tilde{T}^*(\frac{\textbf{u}}{\epsilon_n})P\circ \iota^{-1}(\textbf{u})\iota_*dV(\textbf{u})-\int_{(Q_{\epsilon_n\eta})^c \cap \iota(M)} \frac{1}{\epsilon_n^d}\tilde{T}^*(\frac{\textbf{u}}{\epsilon_n})d(P\circ \iota^{-1})_n(\textbf{u})\right| \nonumber \\
&+ \int_{(Q_{\epsilon_n\eta})^c \cap \iota(M)}\frac{1}{\epsilon_n^d} \left|\tilde{T}(\frac{\textbf{u}}{\epsilon_n})-\tilde{T}^*(\frac{\textbf{u}}{\epsilon_n})\right|d(P\circ \iota^{-1})_n(\textbf{u}) \nonumber \\
\leq & \frac{2}{\epsilon^d} \frac{\gamma}{q}+ \left|\int_{(Q_{\epsilon_n\eta})^c \cap \iota(M)}\frac{1}{\epsilon_n^d} \tilde{T}^*(\frac{\textbf{u}}{\epsilon_n})P\circ \iota^{-1}(\textbf{u})\iota_*dV(\textbf{u})-\int_{(Q_{\epsilon_n\eta})^c \cap \iota(M)} \frac{1}{\epsilon_n^d}\tilde{T}^*(\frac{\textbf{u}}{\epsilon_n})d(P\circ \iota^{-1})_n(\textbf{u})\right|\nonumber \\
\leq & 2\gamma+\left |\int_{(Q_{\epsilon_n\eta})^c \cap \iota(M)}\frac{1}{\epsilon_n^d} \tilde{T}^*(\frac{\textbf{u}}{\epsilon_n})P\circ \iota^{-1}(\textbf{u})\iota_*dV(\textbf{u})-\int_{(Q_{\epsilon_n\eta})^c \cap \iota(M)} \frac{1}{\epsilon_n^d}\tilde{T}^*(\frac{\textbf{u}}{\epsilon_n})d(P\circ \iota^{-1})_n(\textbf{u})\right|\,.\nonumber 
\end{align}

Note that  $\tilde{T}^*(\textbf{u})=0$ outside $Q_{\epsilon_n \eta_{q-1}}$. Moreover, the difference between $\tilde{T}^*(\textbf{u})$ for $\textbf{u}$ in $Q_{\eta_{i+1}} \setminus  Q_{\eta_i}$ and $Q_{\eta_i} \setminus  Q_{\eta_{i-1}}$ is  $\frac{\gamma}{q}$. Hence, we have
\begin{align}
& \int_{(Q_{\epsilon_n\eta})^c \cap \iota(M)}\frac{1}{\epsilon_n^d} \tilde{T}^*(\frac{\textbf{u}}{\epsilon_n})P\circ \iota^{-1}(\textbf{u})\iota_*dV(\textbf{u})\nonumber \\
=& \int_{( Q_{\epsilon_n \eta_{q-1}} \setminus Q_{\epsilon_n\eta} ) \cap \iota(M)}\frac{1}{\epsilon_n^d} \tilde{T}^*(\frac{\textbf{u}}{\epsilon_n})P\circ \iota^{-1}(\textbf{u})\iota_*dV(\textbf{u}) \nonumber \\
=& \frac{\gamma}{q \epsilon_n^d} \sum_{i=1}^{q-1}  \mu \big( (Q_{\epsilon_n \eta_{q-i}} \setminus Q_{\epsilon_n\eta}) \cap \iota(M) \big)\,.\nonumber
\end{align}
Similarly,
\begin{align}
\int_{(Q_{\epsilon_n\eta})^c \cap \iota(M)} \frac{1}{\epsilon_n^d}\tilde{T}^*(\frac{\textbf{u}}{\epsilon_n})d(P\circ \iota^{-1})_n(\textbf{u}) = \frac{\gamma}{q\epsilon_n^d} \sum_{i=1}^{q-1}  \mu_n \big( (Q_{\epsilon_n \eta_{q-i}} \setminus Q_{\epsilon_n\eta}) \cap \iota(M) \big).
\end{align}
Combining the above two equations,
\begin{align}
&\left |\int_{(Q_{\epsilon_n\eta})^c \cap \iota(M)}\frac{1}{\epsilon_n^d} \tilde{T}^*(\frac{\textbf{u}}{\epsilon_n})P\circ \iota^{-1}(\textbf{u})\iota_*dV(\textbf{u})-\int_{(Q_{\epsilon_n\eta})^c \cap \iota(M)} \frac{1}{\epsilon_n^d}\tilde{T}^*(\frac{\textbf{u}}{\epsilon_n})d(P\circ \iota^{-1})_n(\textbf{u})\right|\nonumber \\
\leq & \frac{\gamma}{q \epsilon_n^d} | \sum_{i=1}^{q-1}\Big( \mu \big( (Q_{\epsilon_n \eta_{q-i}} \setminus Q_{\epsilon_n\eta}) \cap \iota(M) \big) -  \mu_n \big( (Q_{\epsilon_n \eta_{q-i}} \setminus Q_{\epsilon_n\eta}) \cap \iota(M) \big) \Big)  |\nonumber \\
\leq & \frac{\gamma}{q \epsilon_n^d} \sum_{i=1}^{q-1}|\mu(Q_{\epsilon_n \eta_{q-i}}\cap \iota(M))-\mu( Q_{\epsilon_n\eta} \cap \iota(M)) -\mu_n(Q_{\epsilon_n \eta_{q-i}}\cap \iota(M))+\mu_n( Q_{\epsilon_n\eta} \cap \iota(M))| \nonumber \\
\leq & \frac{\gamma}{q \epsilon_n^d} \sum_{i=1}^{q-1} \Big( |\mu(Q_{\epsilon_n \eta_{q-i}}\cap \iota(M)) -\mu_n(Q_{\epsilon_n \eta_{q-i}}\cap \iota(M))|+|\mu( Q_{\epsilon_n\eta} \cap \iota(M))- \mu_n( Q_{\epsilon_n\eta} \cap \iota(M))| \Big) \nonumber \\
\leq & \frac{2\gamma}{\epsilon_n^d} \sup_{A \in \mathcal{R}_{2\sqrt{p}\epsilon_n \eta_{q-1}}(\iota(M))} |\mu_n(A)-\mu(A)| . \nonumber
\end{align}

Therefore, \eqref{Q complement 3} which is also the second term in \eqref{Q complement} can be bounded by:
\begin{align}
& \left|\int_{(Q_{\epsilon_n\eta})^c \cap \iota(M)}\frac{1}{\epsilon_n^d} \tilde{T}(\frac{\textbf{u}}{\epsilon_n})P\circ \iota^{-1}(\textbf{u})\iota_*dV(\textbf{u})-\int_{(Q_{\epsilon_n\eta})^c \cap \iota(M)} \frac{1}{\epsilon_n^d}\tilde{T}(\frac{\textbf{u}}{\epsilon_n})d(P\circ \iota^{-1})_n(\textbf{u})\right| \nonumber\\
\leq & 2\gamma+\frac{2\gamma}{\epsilon_n^d} \sup_{A \in \mathcal{R}_{2\sqrt{p}\epsilon_n \eta_{q-1}}(\iota(M))} |\mu_n(A)-\mu(A)| \,.\nonumber
\end{align}
The above equation together with \eqref{Q complement 2} helps us bound \eqref{Q complement} which also bounds the second part in \eqref{inside and outside the cube}:
\begin{align}\label{outside the cube}
& \left|\int_{(Q_{\epsilon_n\eta})^c \cap \iota(M)}\frac{1}{\epsilon_n^d} K(\frac{\|\textbf{u}\|_{\mathbb{R}^p}}{\epsilon_n})P\circ \iota^{-1}(\textbf{u})\iota_*dV(\textbf{u})-\int_{(Q_{\epsilon_n\eta})^c \cap \iota(M)} \frac{1}{\epsilon_n^d}K(\frac{\|\textbf{u}\|_{\mathbb{R}^p}}{\epsilon_n})d(P\circ \iota^{-1})_n(\textbf{u})\right|\nonumber \\
\leq & 2p^{\frac{\alpha}{2}} P_{Max}  (C_{lip} D_{lip})^d \frac{|S^{d-1}|}{\alpha-d}\gamma^{1-\frac{d}{\alpha}} + \frac{2}{4^\alpha}\gamma+ 2\gamma+\frac{2\gamma}{\epsilon_n^d} \sup_{A \in \mathcal{R}_{2\sqrt{p}\epsilon_n \eta_{q-1}}(\iota(M))} |\mu_n(A)-\mu(A)| \nonumber \\
\leq & \mathcal{C}_2 \gamma^{1-\frac{d}{\alpha}}+\frac{2\gamma}{\epsilon_n^d} \sup_{A \in \mathcal{R}_{2\sqrt{p}\epsilon_n \eta_{q-1}}(\iota(M))} |\mu_n(A)-\mu(A)|\,,
\end{align}
where $\mathcal{C}_2:= 2p^{\frac{\alpha}{2}} P_{Max}  (C_{lip} D_{lip})^d \frac{|S^{d-1}|}{\alpha-d}+4$.
\end{proof}

\

\underline{\textbf{Combine Lemma \ref{bound of the local part} and Lemma \ref{bound of the tail part} to complete the proof}}

\

Finally, we plug the results of Lemma \ref{bound of the local part} and Lemma \ref{bound of the tail part} into \eqref{inside and outside the cube} and obtain
\begin{align}\label{not simplified variance term}
& |\mathbb{E}K_n(x)-K_n(x)|  \\
=& |\int_M\frac{1}{\epsilon_n^d} K(\frac{\|\iota(y)-\iota(x)\|_{\mathbb{R}^p}}{\epsilon_n})P(y)dV(y)-\int_M \frac{1}{\epsilon_n^d}K(\|\frac{\iota(y)-\iota(x)\|_{\mathbb{R}^p}}{\epsilon_n})dP_n(y)| \nonumber \\
\leq &(\mathcal{C}_1+\mathcal{C}_2) \gamma^{1-\frac{d}{\alpha}}+\frac{3N(\gamma)+(\ceil{D_{lip}}+1)^d \gamma}{\epsilon_n^d} K_{\sup} \sup_{A \in \mathcal{R}_{2\sqrt{p}\epsilon_n \eta}(\iota(M))} |\mu_n(A)-\mu(A)| \nonumber \\
& +\frac{2\gamma}{\epsilon_n^d} \sup_{A \in \mathcal{R}_{2\sqrt{p}\epsilon_n \eta_{q-1}}(\iota(M))} |\mu_n(A)-\mu(A)| \nonumber 
\end{align}

{By the definition of $\eta_i$ in Lemma\ref{bound of the tail part}, $\eta=\eta_0$, so we have $\epsilon_n \eta=\epsilon_n \eta_0<\epsilon_n \eta_{q-1}$. Hence, we have $\mathcal{R}_{2\sqrt{p}\epsilon_n \eta}(\iota(M)) \subset \mathcal{R}_{2\sqrt{p}\epsilon_n \eta_{q-1}}(\iota(M))$. We conclude that
\begin{align}
\sup_{A \in \mathcal{R}_{2\sqrt{p}\epsilon_n \eta}(\iota(M))} |\mu_n(A)-\mu(A)| \leq \sup_{A \in \mathcal{R}_{2\sqrt{p}\epsilon_n \eta_{q-1}}(\iota(M))} |\mu_n(A)-\mu(A)|. 
\end{align}
Moreover, $\gamma <N(\gamma)$, when $\gamma<1$. So, if we define $\mathcal{C}_3:=(3+(\ceil{D_{lip}}+1)^d) K_{\sup}+2$, then \eqref{not simplified variance term} can be simplified as, 
\begin{align}
& |\mathbb{E}K_n(x)-K_n(x)|  \\
\leq &(\mathcal{C}_1+\mathcal{C}_2) \gamma^{1-\frac{d}{\alpha}}+\frac{\mathcal{C}_3 N(\gamma)}{\epsilon_n^d}\sup_{A \in \mathcal{R}_{2\sqrt{p}\epsilon_n \eta_{q-1}}(\iota(M))} |\mu_n(A)-\mu(A)|. \nonumber
\end{align} }

In order to bound $\sup_{A \in \mathcal{R}_{2\sqrt{p}\epsilon_n \eta_{q-1}}(\iota(M))} |\mu_n(A)-\mu(A)|$ by Lemma \ref{rectangle covering lemma}, we first estimate the upper bound of $\mu(A)$ for $A \in \mathcal{R}_{2\sqrt{p}\epsilon_n \eta_{q-1}}(\iota(M))$. {Since $\eta_{q-1}=(\frac{\gamma}{q})^{-\frac{1}{\alpha}}$, $\frac{1}{\epsilon_n^d} < q <\frac{2}{\epsilon_n^d}$, and $\alpha>d \geq 1$,} we have $4\sqrt{p}\epsilon_n \eta_{q-1} < 8\sqrt{p} \epsilon_n^{1-\frac{d}{\alpha}} \gamma^{-\frac{1}{\alpha}}$. Suppose $A= R\cap \iota(M)$, where $R$ is a rectangle in $\mathbb{R}^p$. Then, $diam(R) \leq  4\sqrt{p} \epsilon_n^{1-\frac{d}{\alpha}} \gamma^{-\frac{1}{\alpha}}$. Choose any $y \in A$.  We can rotate and translate $\iota(M)$ so that $\iota(y)=0 \in \mathbb{R}^p$ and $\iota_*T_y M$ is the subspace $\mathbb{R}^d \subset \mathbb{R}^p$ generated by $e_1, \cdots, e_d$. After the rotation and translation, $R \subset Q_{8 \sqrt{p} \epsilon_n^{1-\frac{d}{\alpha}} \gamma^{-\frac{1}{\alpha}}} \subset Q_{\lambda}$. 
Therefore, if we apply $I_y$ to $Q_{8 \sqrt{p} \epsilon_n^{1-\frac{d}{\alpha}} \gamma^{-\frac{1}{\alpha}}}\cap \iota(M)$, then 
\begin{align}
I_y(A) \subset I_y(Q_{8 \sqrt{p} \epsilon_n^{1-\frac{d}{\alpha}} \gamma^{-\frac{1}{\alpha}}} \cap \iota(M)) = [-8 \sqrt{p} \epsilon_n^{1-\frac{d}{\alpha}} \gamma^{-\frac{1}{\alpha}},8 \sqrt{p} \epsilon_n^{1-\frac{d}{\alpha}} \gamma^{-\frac{1}{\alpha}}]^d. 
\end{align}
We use the bi-Lipschitz property of $I_y$,  \eqref{relation between P and lambda}  and \eqref{relation between gamma and epsilon}, we have
\begin{align}
&\sup_{A \in \mathcal{R}_{4\sqrt{p}\epsilon_n \eta_{q-1}}(\iota(M))} \mu(A) \leq \mu(Q_{8 \sqrt{p} \epsilon_n^{1-\frac{d}{\alpha}} \gamma^{-\frac{1}{\alpha}}} \cap \iota(M)) \\
\leq & b=P_{Max}\big[16 C_{lip}\sqrt{p} \epsilon_n^{1-\frac{d}{\alpha}} \gamma^{-\frac{1}{\alpha}}\big] ^d  \leq P_{Max}[2C_{lip}\lambda]^d \leq \frac{1}{4} \nonumber
\end{align}
Let $\mathcal{C}_4:= P_{Max}\big[16 C_{lip}\sqrt{p}]^d$, then $b=\mathcal{C}_4[\epsilon_n^{1-\frac{d}{\alpha}} \gamma^{-\frac{1}{\alpha}}\big] ^d $. 

If $\delta<b$,  then $\frac{1}{b}<\frac{8b}{\delta^2}$. Lemma \ref{rectangle covering lemma} when it is applied to $A \in \mathcal{R}_{2\sqrt{p}\epsilon_n \eta_{q-1}}(\iota(M))$, can be simplified to the following statement. If $n \geq \frac{8b}{\delta^2}$, then
\begin{align}
P\left\{\sup_{A \in \mathcal{R}_{2\sqrt{p}\epsilon_n \eta_{q-1}}(\iota(M))} |\mu_n(A)-\mu(A)|\geq \delta \right\}  \leq 8(2n)^{2p^2+2p}e^{-\frac{n\delta^2}{68b}} 
\end{align}
Let $\delta= \mathcal{C}_4\frac{\epsilon_n^d}{N(\gamma)} \gamma^{1-\frac{d}{\alpha}}$, then a straightforward calculation shows that $\delta<b$ and we have
\begin{align}
P\left\{\sup_{A \in \mathcal{R}_{2\sqrt{p}\epsilon_n \eta_{q-1}}(\iota(M))} |\mu_n(A)-\mu(A)|\leq \delta \right\}  
\geq 1-8(2n)^{2p^2+2p}e^{-\frac{\mathcal{C}_4}{68} \big(n \epsilon_n^{d+\frac{d^2}{\alpha}}\big)\big(\frac{\gamma^{2-\frac{d}{\alpha}}}{N(\gamma)^2}\big)}. 
\end{align}
In conclusion, suppose $0<\gamma<\min\{\rho^{-\alpha},1\}$.  Let $\epsilon_n \leq \mathcal{D}_3 \gamma^{\frac{1}{\alpha-d}}$. Let $\mathcal{D}_1=\mathcal{C}_1+\mathcal{C}_2+\mathcal{C}_3\mathcal{C}_4$,  and $\mathcal{D}_2=\frac{\mathcal{C}_4}{68}$, then 
\begin{align}
P\left\{\sup_{x \in M}|\mathbb{E}K_n(x)-K_n(x)| \leq \mathcal{D}_1 \gamma^{1-\frac{d}{\alpha}}\right\} \geq 1- 8(2n)^{2p^2+2p}e^{-\mathcal{D}_2 \big(n \epsilon_n^{d+\frac{d^2}{\alpha}}\big)\big(\frac{\gamma^{2-\frac{d}{\alpha}}}{N(\gamma)^2}\big)}.
\end{align}
$\mathcal{D}_1$ and $\mathcal{D}_2$ depend on $p$, $d$, $\alpha$, $P_{Max}$, $K_{sup}$, and the second fundamental form of $\iota(M)$. $\mathcal{D}_3$ depends on $p$, $d$, $\alpha$,$P_{Max}$ and the second fundamental form of $\iota(M)$. 

Note again when $K(t)$ has compact support on $[0, \rho]$, then we only require $\epsilon_n \leq \mathcal{D}_3$ where  $\mathcal{D}_3$ depends on $\rho$, $P_{Max}$ and the second fundamental form of $\iota(M)$.

\begin{remark}
Although we fix $x \in M$,  rotate and translate $\iota(M)$, so that $\iota(x)=0 \in \mathbb{R}^p$ and $\iota_*T_x M$ is the subspace $\mathbb{R}^d \subset \mathbb{R}^p$ generated by $e_1, \cdots, e_d$, our argument is not a pointwise argument. In fact, the argument in the above proof relies on the approximation of $K\left(\frac{\|\textbf{u}-\iota(x)\|_{\mathbb{R}^p}}{\epsilon_n}\right)$ by a simple function $K^*_{\epsilon_n}(\textbf{u}-\iota(x))$ in Lemma \ref{approximation to K}. $K^*_{\epsilon_n}(\textbf{u}-\iota(x))$ is constant on each $A_i$, where $A_i$ is the intersection between a rectangle $R_i$ defined as in Definition \ref{def of rectangle, ball and cube} with $\iota(M)$. Such approximation exists regardless of the rotation or translation of the manifold. The variance on each $A_i$ can always be controlled unformly by Lemma \ref{rectangle covering lemma}, so the argument works for all $x$.  If we rotate and translate the manifold, then $R_i$ is parallel to the axes of $\mathbb{R}^p$ and the proof of Lemma \ref{approximation to K} and the proof of the above theorem can be presented in a simpler way.
\end{remark}

\subsection{\textbf{Proof of Corollary \ref{0-1 kernel}}}
Let $K^+_{j}(t)= c_j \bigchi_{[0, b_j]}(t)$ and $K^-_{j}(t)= c_j \bigchi_{[0, a_j]}(t)$,where $0 \leq a_j \leq b_j \leq \rho$ then $$K(t)=\sum_{j=1}^J(K^+_{j}(t)-K^-_{j}(t)).$$

Let $\tilde{K}(t)= c \bigchi_{[0, b]}(t)$ , where $c \leq K_{\sup}$ and $b \leq \rho$.

For any $x \in M$, 
\begin{align}
& |\mathbb{E}K_n(x)-K_n(x)| \\
\leq  &  \sum_{j=1}^J|\mathbb{E}(K^+_j)_n(x)-(K^+_j)_n(x)|+\sum_{j=1}^J|\mathbb{E}(K^-_j)_n(x)-(K^-_j)_n(x)|.
\end{align}
Hence, it is sufficient to control $|\mathbb{E}\tilde{K}_n(x)-\tilde{K}_n(x)|$. 

Suppose $$\epsilon_n < \frac{\lambda}{8\rho},$$ 
and $\lambda$ is small enough so that 
$$P_{Max}[C_{lip}\lambda]^d \leq \frac{1}{4}.$$
\begin{align} 
& |\mathbb{E}\tilde{K}_n(x)-\tilde{K}_n(x)|  \\
=& \left|\int_M\frac{1}{\epsilon_n^d} \tilde{K} (\frac{\|\iota(y)-\iota(x)\|_{\mathbb{R}^p}}{\epsilon_n})P(y)dV(y)-\int_M \frac{1}{\epsilon_n^d}K(\frac{\|\iota(y)-\iota(x)\|_{\mathbb{R}^p}}{\epsilon_n})dP_n(y)\right| \nonumber \\
\leq & \frac{K_{\sup}}{\epsilon_n ^d}\sup_{A \in \mathcal{B}_{\epsilon_n \rho}(\iota(M))} |\mu_n(A)-\mu(A)|. \nonumber
\end{align}
Therefore, 
\begin{align}
|\mathbb{E}K_n(x)-K_n(x)| \leq  \frac{2J K_{\sup}}{\epsilon_n ^d}\sup_{A \in \mathcal{B}_{\epsilon_n \rho}(\iota(M))} |\mu_n(A)-\mu(A)|..
\end{align}
In order to bound $\sup_{A \in \mathcal{B}_{\epsilon_n \rho}(\iota(M))} |\mu_n(A)-\mu(A)|$ by Lemma \ref{ball covering lemma}, we first estimate the upper bound of $\mu(A)$ for $A \in  \mathcal{B}_{2\epsilon_n\rho}(\iota(M))$. Suppose $A= B^{\mathbb{R}^p}_{2\epsilon_n\rho}\cap \iota(M)$. Choose any $y \in A$.  We can rotate and translate $\iota(M)$ so that $\iota(y)=0 \in \mathbb{R}^p$ and $\iota_*T_y M$ is the subpace $\mathbb{R}^d \subset \mathbb{R}^p$ generated by $e_1, \cdots, e_d$. After the rotation and translation, since $\epsilon_n < \frac{\lambda}{8\rho}$, we have $B^{\mathbb{R}^p}_{2\epsilon_n\rho} \subset Q_{4 \epsilon_n\rho} \subset Q_{\lambda}$. 
Therefore, if we apply $I_y$ to $Q_{4 \epsilon_n\rho}\cap \iota(M)$, then 
\begin{align}
I_y(A) \subset I_y(Q_{4 \epsilon_n\rho}\cap \iota(M)) = [-4 \epsilon_n\rho,4 \epsilon_n\rho]^d. 
\end{align}
We use the bi-Lipschitz property of $I_y$,
\begin{align}
\sup_{A \in  \mathcal{B}_{2\epsilon_n\rho}(\iota(M))} \mu(A) 
\leq b:= P_{Max} \big[8 C_{lip} \rho \epsilon_n\big] ^d  \leq P_{Max}[C_{lip}\lambda]^d \leq \frac{1}{4}. \nonumber
\end{align}
Define $\mathcal{C}_2=P_{Max}(8 C_{lip}\rho )^d$. Then $b=\mathcal{C}_2 \epsilon_n^d$. 

If $\delta<b$,  then $\frac{1}{b}<\frac{8b}{\delta^2}$ and Lemma \ref{ball covering lemma} when it is applied to $A \in \mathcal{B}_{\epsilon_n\rho}(\iota(M))$ can be simplified to the following statement. If $n \geq \frac{8b}{\delta^2}$, then
\begin{align}
P\left\{\sup_{A \in \mathcal{B}_{\epsilon_n \rho}(\iota(M))} |\mu_n(A)-\mu(A)|\geq \delta \right\}  \leq 8(2n)^{p+2}e^{-\frac{n\delta^2}{68b}} 
\end{align}

Let $\delta= \mathcal{C}_2 \epsilon_n^d \gamma$, then $\delta<b$. Hence,
\begin{align}
P\left\{\sup_{A \in \mathcal{B}_{\epsilon_n \rho}(\iota(M))} |\mu_n(A)-\mu(A)|\geq \delta \right\}  \leq 8(2n)^{p+2}e^{-\frac{\mathcal{C}_2}{68} n  \epsilon_n^d \gamma^2}. 
\end{align}
Let $\mathcal{D}_1=2J K_{\sup}  \mathcal{C}_2 =2J K_{\sup} P_{Max}(8 C_{lip})^d$, $\mathcal{D}_2=\frac{\mathcal{C}_2}{68}=\frac{P_{Max}(8 C_{lip} \rho)^d}{68}$ and $\mathcal{D}_3=\min( \frac{\lambda}{8\rho},1)$, then 
\begin{align}
P\left\{\sup_{x \in M}|\mathbb{E}K_n(x)-K_n(x)| \leq \mathcal{D}_1 \gamma\right\} \geq 1-8(2n)^{p+2}e^{-\mathcal{D}_2 n \epsilon_n^d \gamma^2}.
\end{align}

Choose $\epsilon_n= \gamma^2$ and $8(2n)^{p+2}e^{-\mathcal{D}_2 n \epsilon_n^d \gamma^2}=\frac{1}{n^2}$, the last statement in the corollary follows.

\

\section{Proof of Theorem \ref{bias analysis main theorem} and Theorem \ref{density estimation}}\label{Appendix B}

Denote $B^{\mathbb{R}^d}_r(0)$ the open ball of radius $r$ at $0$ in $\mathbb{R}^d$. For $x\in M$, the exponential map at $x$ is denoted as $\exp_{x}(\textbf{v})$. Let $B_r(x) \subset M$ be the open geodesic ball of radius $r$ at $x$. Suppose $r$ is less than $inj(M)$, the injectivity radius of $M$. Then $\exp_{x}(\textbf{v})$ is a diffeomorphism from $B^{\mathbb{R}^d}_r(0)$ to $B_r(x)$, which is the normal coordinates of the manifold around $x$. Hence, for any function $f$ on M, 
\begin{align}
\int_{B_r(x)} f(y)dV(y)= \int_{B^{\mathbb{R}^d}_r(0)} f(\exp_x(\textbf{v})) V_x(\textbf{v})d\textbf{v},
\end{align}
where $V_x(\textbf{v})d\textbf{v}$ is the associated volume form in the normal coordinates. The following lemma is about the volume form. The proof of the lemma can be found in \cite{wu2018think}.

\begin{lemma} \label{volume form in normal coordinate}
Fix $x\in M$. If we use the Cartesian coordinate to parametrize $T_{x}M$, we can write the volume form as $dV=V_x(\textbf{v})d\textbf{v}$. The volume form has the following expansion
\begin{align}
dV=\bigg(&1-\frac{1}{6}\sum_{i,j=1}^d\texttt{Ric}_{x}(i,j) \textbf{v}_i\textbf{v}_j+O(\textbf{v}^3)\bigg) d\textbf{v},
\end{align}
where $\textbf{v}=\sum_{i=1}^d \textbf{v}_ie_i\in T_{x}M$, $\texttt{Ric}_{x}(i,j)=\texttt{Ric}_{x}(e_i,e_j)$.
\end{lemma}

The volume form of $V_x(\textbf{v})$ is smooth function on $B^{\mathbb{R}^d}_r(0) \subset T_xM$. If $r <inj(M)$, then $\sup_{\|\textbf{v}\|_{\mathbb{R}^d} <inj(M) }V_x(\textbf{v})$ can be bounded in terms of  the curvature of $M$ at $x$. Since the manifold is compact, we introduce the universal upper bound of the volume form in the normal coordinates.

\begin{definition}\label{max of volume for on inj ball}
For $\textbf{v} \in T_xM$, we have
\begin{align}
\sup_{x \in M} \sup_{\|\textbf{v}\|_{\mathbb{R}^d} <inj(M) }V_x(\textbf{v}) < V_{max}.
\end{align}
\end{definition}

The next lemma relates the geodesic distance between two points on the manifold and the Euclidean distance between the corresponding points on the embedded submanifold in $\mathbb{R}^p$.  The proof of the lemma can be found in \cite{wu2018think}.

\begin{lemma} \label{geodesic and euclidean distance}
Fix $x\in M$. If $\textbf{v} \in T_{x}M$, when $\|\iota \circ \exp_x(\textbf{v})-\iota(x)\|_{\mathbb{R}^p}$ is sufficiently small, we have
\begin{align*}
\frac{\|\iota \circ \exp_x(\textbf{v})-\iota(x)\|_{\mathbb{R}^p}}{\|\textbf{v}\|_{\mathbb{R}^d}} = &\,   1-\frac{1}{24} \|\Second_x(\theta,\theta)\|^2 \|\textbf{v}\|_{\mathbb{R}^d}^2+O(\|\textbf{v}\|_{\mathbb{R}^d}^3)\,,
\end{align*}
where $\theta=\frac{\textbf{v}}{\|\textbf{v}\|_{\mathbb{R}^d}}$. 
\end{lemma}

\

For any $x \in M$, and $\textbf{v} \in T_xM$ and $0<\|\textbf{v}\|_{\mathbb{R}^d}<inj(M)$, we define a map $\phi_x$:
\begin{align}
\phi_x(\textbf{v})=\frac{\|\iota \circ \exp_x(\textbf{v})-\iota(x)\|_{\mathbb{R}^p}}{\|\textbf{v}\|_{\mathbb{R}^d}} \textbf{v}.
\end{align}

The following lemma describes an important property of $\phi_x$

\begin{lemma}\label{property of phi x}
When $\delta$ is small enough,  for all $x$, $\phi_x$ is a diffeomorphism on $B^{\mathbb{R}^d}_\delta (0) \setminus \{0\}$. We have
\begin{align}
\|\phi_x(\textbf{v})\|_{\mathbb{R}^d} \leq \|\textbf{v}\|_{\mathbb{R}^d},
\end{align}
and 
\begin{align}
|Det(D\phi_x^{-1}(\textbf{v}))|=1+O(\|\textbf{v}\|^2_{\mathbb{R}^d}).
\end{align}
The constant in $O(\|\textbf{v}\|^2_{\mathbb{R}^d})$ depends on the second fundamental form of $M$.
\end{lemma}

\begin{proof}
By Lemma \ref{geodesic and euclidean distance} and a straight forward calculation, when $\textbf{v} \not =0 $, we have $Det(D\phi_x(\textbf{v}))=1+O(\|\textbf{v}\|^2_{\mathbb{R}^d})$. The constant in $O(\|\textbf{v}\|^2_{\mathbb{R}^d})$ depends on the second fundamental form of $M$. Since $M$ is compact, for all $x$, $Det(D\phi_x(\textbf{v}))>0$, if $\|\textbf{v}\|^2_{\mathbb{R}^d}<\delta$ and $\delta$ is small enough. The conclusion follows from the inverse function theorem. 

Since $\|\textbf{v}\|_{\mathbb{R}^d}$ is the geodesic distance between $\iota \circ \exp_x(\textbf{v})$ and $\iota(x)$ in $\iota(M)$, while $\|\phi_x(\textbf{v})\|_{\mathbb{R}^d}$ is equal to the Euclidean distance between $\iota \circ \exp_x(\textbf{v})$ and $\iota(x)$ in $\mathbb{R}^p$, $\|\phi_x(\textbf{v})\|_{\mathbb{R}^d} \leq \|\textbf{v}\|_{\mathbb{R}^d}$ follows.
\end{proof}

We discuss more properties of the map $\phi_x$ in the next lemma.

\begin{lemma}\label{properties of the map phi x}
If $\delta$ is small enough, then we have the following statements:
\begin{enumerate}
\item
Define 
\begin{align}
& M_{\delta/2}(x)=\{y \in M | \|\iota(y)-\iota(x)\|_{\mathbb{R}^p} \leq \frac{\delta}{2}\}, \\
& D_{\delta/2}(x)= \exp_x^{-1}(M_{\delta/2}(x)) \subset T_xM \approx \mathbb{R}^d.
\end{align}
Then,
\begin{align}
B^{\mathbb{R}^d}_{\delta/2}(0) \setminus \{0\} = \phi_x\big(D_{\delta/2}(x)\setminus \{0\} \big) \subset D_{\delta/2}(x) \subset B^{\mathbb{R}^d}_\delta(0).
\end{align}
\item
For all $x$, and all $0<\|\textbf{v}\|_{\mathbb{R}^d} <\delta$
\begin{align}
||Det(D\phi_x^{-1}(\textbf{v}))|-1| \leq C_1 \delta^2,
\end{align}
where $C_1$ depends on the second fundamental form of $\iota(M)$.
\item
Suppose $\|\textbf{v}\|_{\mathbb{R}^d} <\delta$. There is a constant $C_2$ depending on the H\"older constant $C_P$ of $P$, $P_{Max}$ and the curvature of $M$ such that for any $x$ and $y=\exp_x(\textbf{v})$, we have 
\begin{align} \label{holder of property of P exp}
|P(\exp_x(\textbf{v}))V_x(\textbf{v})-P(\exp_x(0))| \leq C_2 \delta^{\kappa},
\end{align}
where $\kappa$ is the H\"older exponent of $P$.
\item
Suppose $0<\|\textbf{v}\|_{\mathbb{R}^d} <\delta$. There is a constant $C_3$ depending on $C_P$, $\kappa$, $ P_{Max}$, the curvature of $M$ and  the second fundamental form of $\iota(M)$ such that
\begin{align}
\left|P(\exp_x(\phi_x^{-1}(\textbf{v})))V_x(\phi_x^{-1}(\textbf{v}))|Det(D\phi_x^{-1}(\textbf{v}))|-P(\exp_x(\textbf{v}))V_x(\textbf{v})\right| \leq C_3 \delta^{2\kappa},
\end{align}
\end{enumerate}
\end{lemma}

\begin{proof}
(1) $B^{\mathbb{R}^d}_{\delta/2}(0) \setminus \{0\} = \phi_x\big(D_{\delta/2}(x)\setminus \{0\} \big)$  follows from the definition of $\phi_x$. $\phi_x\big(D_{\delta/2}(x)\setminus \{0\}\big) \subset D_{\delta/2}(x)$ follows from Lemma \ref{property of phi x}. Suppose $\|\textbf{v}\|_{\mathbb{R}^d}=\delta$. By Lemma \ref{geodesic and euclidean distance}, if $\delta$ is small enough, there is a constant $C$ depending on the second fundamental form of $M$ such that $\|\iota \circ \exp_x(\textbf{v})-\iota(x)\|_{\mathbb{R}^p} \geq \delta-C \delta^2>\frac{\delta}{2}$.  So, we show that the boundary of $ B^{\mathbb{R}^d}_\delta(0)$ is in the complement of $D_{\delta/2}(x)$. Hence, $D_{\delta/2}(x) \subset B^{\mathbb{R}^d}_\delta(0)$ holds whenever $\delta$ is small enough.

(2) By Lemma \ref{property of phi x}, for all $x$ and all $0<\|\textbf{v}\|_{\mathbb{R}^d} <\delta$, there is a constant $C_1$ depending on the second fundamental form of $M$ such that $||Det(D\phi_x^{-1}(\textbf{v}))|-1| \leq C_1 \delta^2$. 

(3) P is H\"older so $|P(\exp_x(\textbf{v}))-P(\exp_x(0))| \leq C_P d(x,y)^{\kappa}<C_P \delta^{\kappa}$. By Lemma \ref{volume form in normal coordinate}, we have $|V_x(\textbf{v})-1|\leq C \delta^2$, where $C$ 
is a constant depending on the curvature of $M$. Hence,
\begin{align}
& |P(\exp_x(\textbf{v}))V_x(\textbf{v})-P(\exp_x(0))| \\
\leq & |P(\exp_x(\textbf{v}))V_x(\textbf{v})-P(\exp_x(0))V_x(\textbf{v})|+|P(\exp_x(0))V_x(\textbf{v})-P(\exp_x(0))| \nonumber \\
\leq & V_{Max}|P(\exp_x(\textbf{v}))-P(\exp_x(0))|+P_{Max}|V_x(\textbf{v})-1| \nonumber \\
\leq & C_P  V_{Max} \delta^{\kappa}+ P_{Max} C \delta^2 \nonumber \\
\leq & C_2 \delta^{\kappa}. \nonumber
\end{align}
where $C_2$ depends on $C_P$, $P_{Max}$ and the curvature of $M$.

(4) Note that $\phi_x^{-1}(\textbf{v})$ is a positive scalar mutiple of $\textbf{v}$. Hence, by Lemma \ref{geodesic and euclidean distance},
\begin{align}
d(\exp_x(\phi_x^{-1}(\textbf{v})),\exp_x(\textbf{v}))=\|\phi_x^{-1}(\textbf{v})-\textbf{v}\|_{\mathbb{R}^d} \leq \tilde{C}_1 \|\textbf{v}\|^3_{\mathbb{R}^d} \leq \tilde{C}_1 \delta^3,
\end{align}
where $\tilde{C}_1$ depends on the second fundamental form of  $M$.
Hence,
\begin{align}
|P(\exp_x(\phi_x^{-1}(\textbf{v})))-P(\exp_x(\textbf{v}))| \leq C_P(\tilde{C}_1 \delta^3)^\kappa=C_P \tilde{C}_1^\kappa \delta^{3\kappa}
\end{align}
By Lemma \ref{geodesic and euclidean distance}, if $\delta$ is small enough
\begin{align}
|V_x(\phi_x^{-1}(\textbf{v}))-V_x(\textbf{v})|\leq  \tilde{C}_2 (\|\phi_x^{-1}(\textbf{v})\|^2_{\mathbb{R}^d}+\|\textbf{v}\|^2_{\mathbb{R}^d}) \leq 5\tilde{C}_2 \delta^2,
\end{align}
where $\tilde{C}_2$ depends on the curvature of  $M$.
From (2), $||Det(D\phi_x^{-1}(\textbf{v}))|-1| \leq C_1 \delta^2$ and $|Det(D\phi_x^{-1}(\textbf{v}))| \leq 2$. Hence,
\begin{align}
& \left|P(\exp_x(\phi_x^{-1}(\textbf{v})))V_x(\phi_x^{-1}(\textbf{v})|Det(D\phi_x^{-1}(\textbf{v}))|-P(\exp_x(\textbf{v}))V_x(\textbf{v})\right| \\
\leq &  |P(\exp_x(\phi_x^{-1}(\textbf{v})))-P(\exp_x(\textbf{v}))| \max_{\textbf{v}} V_x(\phi_x^{-1}(\textbf{v}) \max_{\textbf{v}}|Det(D\phi_x^{-1}(\textbf{v}))| \nonumber  \\
& +\max_{\textbf{v}}P(\exp_x(\textbf{v})) \max_{\textbf{v}} V_x(\phi_x^{-1}(\textbf{v}) ||Det(D\phi_x^{-1}(\textbf{v}))|-1| \nonumber \\
& +\max_{\textbf{v}}P(\exp_x(\textbf{v})) |V_x(\phi_x^{-1}(\textbf{v})-V_x(\textbf{v})|  \nonumber  \\
\leq & 2V_{Max}C_P \tilde{C}_1^\kappa \delta^{3\kappa}+ P_{Max}V_{Max}C_1 \delta^2 +5 P_{Max}\tilde{C}_2 \delta^2 \leq C_3 \delta^{2\kappa}, \nonumber
\end{align}
where $C_3$ depends on $C_P$, $\kappa$, $ P_{Max}$, the curvature of $M$ and  the second fundamental form of $\iota(M)$.
\end{proof}

\underline{\textbf{Proof of Theorem \ref{bias analysis main theorem}}}

Note that based on Assumption \ref{assumptions on kernel 1}, we have
\begin{align}
\tilde{C}_1:=|S^{d-1}|(K_{\sup}\rho^d+ \frac{1}{\alpha-d}\rho^{d-\alpha}) \geq   \int_{\mathbb{R}^d}| K(\|\textbf{v}\|_{\mathbb{R}^d}) | d\textbf{v} = \int_{\mathbb{R}^d} \frac{1}{\epsilon^d} |K(\frac{\|\textbf{v}\|_{\mathbb{R}^d}}{\epsilon})| d\textbf{v}.
\end{align}
Fix $x \in M$, define 
\begin{align}
& M_{\delta/2}(x)=\{y \in M | \|\iota(y)-\iota(x)\|_{\mathbb{R}^p} \leq \frac{\delta}{2}\}, \\
& D_{\delta/2}(x)= \exp_x^{-1}(M_{\delta/2}(x)) \subset T_xM \approx \mathbb{R}^d.
\end{align}

Given $0<\gamma<1$, we choose a $0<\delta<\gamma$ and $\delta$ is small enough so that (1) to (4) in Lemma \ref{properties of the map phi x} are satisfied.  Based on $\delta$, we choose $\epsilon$. We need to find the relation between $\epsilon$ and $\delta$ so that $|K(\frac{t}{\epsilon})| \leq \epsilon^{\frac{\alpha+d}{2}}$ for $t>\frac{\delta}{2}$.
For $t>\frac{\delta}{2}$, if 
\begin{align}
\epsilon<\frac{\delta}{2\rho}, 
\end{align}
then $\frac{t}{\epsilon}>\rho$. Hence,
\begin{align}\label{bound of K t for t greater than delta /2}
|K(\frac{t}{\epsilon})| < (\frac{\epsilon}{t})^\alpha <\epsilon^\alpha (\frac{\delta}{2})^{-\alpha},
\end{align}
where we use $t>\frac{\delta}{2}$ in the last step. Note that if 
\begin{align}
\epsilon<(\frac{\delta}{2})^{\frac{2}{1-d/\alpha}},
\end{align}
then $|K(\frac{t}{\epsilon})|<\epsilon^\alpha (\frac{\delta}{2})^{-\alpha}< \epsilon^{\frac{\alpha+d}{2}}$. 

In conclusion, if
\begin{align}\label{relation between epsilon and delta}
\epsilon <\min\{\frac{\delta}{2\rho}, (\frac{\delta}{2})^{\frac{2}{1-d/\alpha}}\}<\frac{1}{2\rho+1}(\frac{\delta}{2})^{\frac{2}{1-d/\alpha}}=\frac{1}{2\rho+1}(\frac{\delta}{2})^{\frac{2\alpha}{\alpha-d}},
\end{align}
then $|K(\frac{t}{\epsilon})| \leq \epsilon^{\frac{\alpha+d}{2}}$ for $t>\frac{\delta}{2}$. If we substitute $\delta<\gamma$ into \eqref{relation between epsilon and delta}, then
\begin{align}\label{relation between epsilon and gamma}
\epsilon <\frac{1}{2\rho+1}(\frac{\gamma}{2})^{\frac{2\alpha}{(\alpha-d)}}
\end{align}

Since $\epsilon<(\frac{\delta}{2})^{\frac{2}{1-d/\alpha}}$, we have $\frac{2\epsilon}{\delta}<(\frac{\delta}{2})^{\frac{\alpha+d}{\alpha-d}}<(\frac{\gamma}{2})^{\frac{\alpha+d}{\alpha-d}}$. Hence,
\begin{align}\label{integral of kernel outside delta/2  ball}
\int_{\mathbb{R}^d \setminus B^{\mathbb{R}^d}_{\delta/2}(0) } \frac{1}{\epsilon^d} |K(\frac{\|\textbf{v}\|_{\mathbb{R}^d}}{\epsilon}) |d\textbf{v} \leq  \int_{\mathbb{R}^d \setminus B^{\mathbb{R}^d}_{\delta/2}(0) }   \frac{1}{\epsilon^d}(\frac{\epsilon}{\|\textbf{v}\|_{\mathbb{R}^d}})^\alpha d\textbf{v} \leq  \frac{|S^{d-1}|}{\alpha-d}(\frac{2\epsilon}{\delta})^{\alpha-d} \leq \tilde{C}_2 \gamma^{\alpha+d},
\end{align}
where $\tilde{C}_2= \frac{|S^{d-1}|}{\alpha-d} (\frac{1}{2})^{\alpha+d}$. 

Now, we are ready to bound $|\int_{M}\frac{1}{\epsilon^d} K(\frac{\|\iota(y)-\iota(x)\|_{\mathbb{R}^p}}{\epsilon})P(y)dV(y)-P(x)|$. 

\begin{align}
&\left|\int_{M}\frac{1}{\epsilon^d} K(\frac{\|\iota(y)-\iota(x)\|_{\mathbb{R}^p}}{\epsilon})P(y)dV(y)-P(x)\right| \\
= & \left|\int_{M}\frac{1}{\epsilon^d} K(\frac{\|\iota(y)-\iota(x)\|_{\mathbb{R}^p}}{\epsilon})P(y)dV(y)-\int_{\mathbb{R}^d}\frac{1}{\epsilon^d} K(\frac{\|\textbf{v}\|_{\mathbb{R}^d}}{\epsilon})P(x)d\textbf{v}\right| \nonumber \\
\leq & \left|\int_{M_{\delta/2}(x)}\frac{1}{\epsilon^d} K(\frac{\|\iota(y)-\iota(x)\|_{\mathbb{R}^p}}{\epsilon})P(y)dV(y)-\int_{\mathbb{R}^d}\frac{1}{\epsilon^d} K(\frac{\|\textbf{v}\|_{\mathbb{R}^d}}{\epsilon})P(x)d\textbf{v}\right| \nonumber \\
& +\left|\int_{M \setminus M_{\delta/2}(x)}\frac{1}{\epsilon^d} K(\frac{\|\iota(y)-\iota(x)\|_{\mathbb{R}^p}}{\epsilon})P(y)dV(y)\right|\, .\nonumber
\end{align}
If $y \in M \setminus M_{\delta/2}(x)$, then $\|\iota(y)-\iota(x)\|_{\mathbb{R}^p}>\frac{\delta}{2}$. We have $| K(\frac{\|\iota(y)-\iota(x)\|_{\mathbb{R}^p}}{\epsilon}) |< \epsilon^{\frac{\alpha+d}{2}}$. Hence,
\begin{align}\label{inside M delta x}
& \left|\int_{M \setminus M_{\delta/2}(x)}\frac{1}{\epsilon^d} K(\frac{\|\iota(y)-\iota(x)\|_{\mathbb{R}^p}}{\epsilon})P(y)dV(y)\right| \leq  \epsilon^{\frac{\alpha-d}{2}} \left|\int_{M \setminus M_{\delta/2}(x)}P(y)dV(y)\right| \\
\leq & \epsilon^{\frac{\alpha-d}{2}}  <(\frac{1}{2\rho+1} )^{\frac{\alpha-d}{2}}(\frac{\gamma}{2})^\alpha , \nonumber
\end{align}
where we substitute \eqref{relation between epsilon and gamma} in the last step.

Next,
\begin{align}
& \left|\int_{M_{\delta/2}(x)}\frac{1}{\epsilon^d} K(\frac{\|\iota(y)-\iota(x)\|_{\mathbb{R}^p}}{\epsilon})P(y)dV(y)-\int_{\mathbb{R}^d}\frac{1}{\epsilon^d} K(\frac{\|\textbf{v}\|_{\mathbb{R}^d}}{\epsilon})P(x)d\textbf{v}\right| \nonumber \\
=& \left|\int_{D_{\delta/2}(x)}\frac{1}{\epsilon^d} K(\frac{\|\iota \circ \exp_x(\textbf{v})-\iota(x)\|_{\mathbb{R}^p}}{\epsilon})P(\exp_x(\textbf{v}))V_x(\textbf{v})d\textbf{v}-\int_{\mathbb{R}^d}\frac{1}{\epsilon^d} K(\frac{\|\textbf{v}\|_{\mathbb{R}^d}}{\epsilon})P(\exp_x(0))d\textbf{v}\right| \nonumber \\
\leq &  \left|\int_{D_{\delta/2}(x)}\frac{1}{\epsilon^d} \bigg(K(\frac{\|\iota \circ \exp_x(\textbf{v})-\iota(x)\|_{\mathbb{R}^p}}{\epsilon}) - K(\frac{\|\textbf{v}\|_{\mathbb{R}^d}}{\epsilon})\bigg)P(\exp_x(\textbf{v}))V_x(\textbf{v})d\textbf{v} \right| \nonumber \\
&\,+\left|\int_{D_{\delta/2}(x)}\frac{1}{\epsilon^d} K(\frac{\|\textbf{v}\|_{\mathbb{R}^d}}{\epsilon})\big(P(\exp_x(\textbf{v}))V_x(\textbf{v})-P(\exp_x(0))\big) d\textbf{v} \right|\nonumber\\
&\,+\left| \int_{\mathbb{R}^d \setminus D_{\delta/2}(x)} \frac{1}{\epsilon^d} K(\frac{\|\textbf{v}\|_{\mathbb{R}^d}}{\epsilon}) P(\exp_x(0))d\textbf{v}\right| \label{In side and outside D delta in TM}\,.
\end{align}
We bound \eqref{In side and outside D delta in TM} term by term. For the second term in \eqref{In side and outside D delta in TM}, by (1) in Lemma \ref{properties of the map phi x}, $ D_{\delta/2}(x) \subset B^{\mathbb{R}^d}_\delta(0)$,
\begin{align}
&\left|\int_{D_{\delta/2}(x)}\frac{1}{\epsilon^d} K(\frac{\|\textbf{v}\|_{\mathbb{R}^d}}{\epsilon})\big(P(\exp_x(\textbf{v}))V_x(\textbf{v})-P(\exp_x(0))\big) d\textbf{v} \right| \\
\leq & \sup_{\textbf{v} \in D_{\delta/2}(x)  }|P(\exp_x(\textbf{v}))V_x(\textbf{v})-P(\exp_x(0))|\int_{D_{\delta/2}(x)}\frac{1}{\epsilon^d} | K(\frac{\|\textbf{v}\|_{\mathbb{R}^d}}{\epsilon}) | d\textbf{v}  \nonumber\\ 
\leq & \sup_{\textbf{v} \in B^{\mathbb{R}^d}_\delta(0)}|P(\exp_x(\textbf{v}))V_x(\textbf{v})-P(\exp_x(0))|  \int_{\mathbb{R}^d} \frac{1}{\epsilon^d} |K(\frac{\|\textbf{v}\|_{\mathbb{R}^d}}{\epsilon})| d\textbf{v}\leq  \tilde{C}_1 C_2 \delta^{\kappa} \leq  \tilde{C}_1 C_2 \gamma^{\kappa}  \nonumber\,,
\end{align}
where we use \eqref{holder of property of P exp} in the second last step and $\delta \leq \gamma$ in the last step. 

We bound the third term in  \eqref{In side and outside D delta in TM}.  By (1) in Lemma \ref{properties of the map phi x}, $B^{\mathbb{R}^d}_{\delta/2}(0) \subset D_{\delta/2}(x)$. Hence, we have
\begin{align}
& \left| \int_{\mathbb{R}^d \setminus D_{\delta/2}(x)} \frac{1}{\epsilon^d} K(\frac{\|\textbf{v}\|_{\mathbb{R}^d}}{\epsilon}) P(\exp_x(0))d\textbf{v} \right| \leq \int_{\mathbb{R}^d \setminus D_{\delta/2}(x)} \frac{1}{\epsilon^d}| K(\frac{\|\textbf{v}\|_{\mathbb{R}^d}}{\epsilon}) |P(\exp_x(0))d\textbf{v} \\
\leq &\, \int_{\mathbb{R}^d \setminus B^{\mathbb{R}^d}_{\delta/2}(0) } \frac{1}{\epsilon^d}| K(\frac{\|\textbf{v}\|_{\mathbb{R}^d}}{\epsilon}) | P(\exp_x(0))d\textbf{v} \leq P_{Max}\tilde{C}_2 \gamma^{\alpha+d}\,,
\end{align}
where we use \eqref{integral of kernel outside delta/2  ball} in the last step.

At last, for the first term in  \eqref{In side and outside D delta in TM}, 
\begin{align}
 &\left|\int_{D_{\delta/2}(x)}\frac{1}{\epsilon^d} \bigg(K(\frac{\|\iota \circ \exp_x(\textbf{v})-\iota(x)\|_{\mathbb{R}^p}}{\epsilon}) - K(\frac{\|\textbf{v}\|_{\mathbb{R}^d}}{\epsilon})\bigg)P(\exp_x(\textbf{v}))V_x(\textbf{v})d\textbf{v}\right | \\
= &\left|\int_{D_{\delta/2}(x)}\frac{1}{\epsilon^d} K(\frac{\|\phi_x(\textbf{v})\|_{\mathbb{R}^d}}{\epsilon}) P(\exp_x(\textbf{v}))V_x(\textbf{v})d\textbf{v}-\int_{D_{\delta/2}(x)}\frac{1}{\epsilon^d}  K(\frac{\|\textbf{v}\|_{\mathbb{R}^d}}{\epsilon})P(\exp_x(\textbf{v}))V_x(\textbf{v})d\textbf{v} \right| \nonumber \\
=& \Bigg |\int_{\phi_x \big(D_{\delta/2}(x) \setminus \{0\}\big)}\frac{1}{\epsilon^d} K(\frac{\|\textbf{v}\|_{\mathbb{R}^d}}{\epsilon})P(\exp_x(\phi_x^{-1}(\textbf{v})))V_x(\phi_x^{-1}(\textbf{v}))|Det(D\phi_x^{-1}(\textbf{v}))| d\textbf{v} \nonumber \\
&-\int_{D_{\delta/2}(x)}\frac{1}{\epsilon^d}  K(\frac{\|\textbf{v}\|_{\mathbb{R}^d}}{\epsilon})P(\exp_x(\textbf{v}))V_x(\textbf{v})d\textbf{v} \Bigg| \nonumber \\
=& \Bigg |\int_{B^{\mathbb{R}^d}_{\delta/2}(0)\setminus \{0\}}\frac{1}{\epsilon^d} K(\frac{\|\textbf{v}\|_{\mathbb{R}^d}}{\epsilon})P(\exp_x(\phi_x^{-1}(\textbf{v})))V_x(\phi_x^{-1}(\textbf{v}))|Det(D\phi_x^{-1}(\textbf{v}))| d\textbf{v} \nonumber \\
&-\int_{D_{\delta/2}(x)}\frac{1}{\epsilon^d}  K(\frac{\|\textbf{v}\|_{\mathbb{R}^d}}{\epsilon})P(\exp_x(\textbf{v}))V_x(\textbf{v})d\textbf{v} \Bigg| 
\end{align}
By (1) in Lemma \ref{properties of the map phi x}, $B^{\mathbb{R}^d}_{\delta/2}(0)\setminus \{0\} \subset D_{\delta/2}(x) $, hence, we have
\begin{align}
& \Bigg |\int_{B^{\mathbb{R}^d}_{\delta/2}(0)\setminus \{0\}}\frac{1}{\epsilon^d} K(\frac{\|\textbf{v}\|_{\mathbb{R}^d}}{\epsilon})P(\exp_x(\phi_x^{-1}(\textbf{v})))V_x(\phi_x^{-1}(\textbf{v}))|Det(D\phi_x^{-1}(\textbf{v}))| d\textbf{v} \\
& -\int_{D_{\delta/2}(x)}\frac{1}{\epsilon^d}  K(\frac{\|\textbf{v}\|_{\mathbb{R}^d}}{\epsilon})P(\exp_x(\textbf{v}))V_x(\textbf{v})d\textbf{v} \Bigg| \nonumber\\
\leq & \int_{B^{\mathbb{R}^d}_{\delta/2}(0)\setminus \{0\}}\frac{1}{\epsilon^d}|K(\frac{\|\textbf{v}\|_{\mathbb{R}^d}}{\epsilon})| \Big|P(\exp_x(\phi_x^{-1}(\textbf{v})))V_x(\phi_x^{-1}(\textbf{v}))|Det(D\phi_x^{-1}(\textbf{v}))|-P(\exp_x(\textbf{v}))V_x(\textbf{v})\Big| d\textbf{v}  \nonumber \\
&+ \int_{D_{\delta/2}(x) \setminus B^{\mathbb{R}^d}_{\delta/2}(0)}\frac{1}{\epsilon^d} | K(\frac{\|\textbf{v}\|_{\mathbb{R}^d}}{\epsilon}) | P(\exp_x(\textbf{v}))V_x(\textbf{v}) d\textbf{v} \nonumber \\
\leq & \max_{\textbf{v} \in B^{\mathbb{R}^d}_{\delta/2}(0)\setminus \{0\}} \Big|P(\exp_x(\phi_x^{-1}(\textbf{v})))V_x(\phi_x^{-1}(\textbf{v}))|Det(D\phi_x^{-1}(\textbf{v}))|-P(\exp_x(\textbf{v}))V_x(\textbf{v})\Big|  \int_{B^{\mathbb{R}^d}_{\delta/2}(0)}\frac{1}{\epsilon^d}|K(\frac{\|\textbf{v}\|_{\mathbb{R}^d}}{\epsilon})|  d\textbf{v}  \nonumber \\
&+ \int_{D_{\delta/2}(x) \setminus B^{\mathbb{R}^d}_{\delta/2}(0)}\frac{1}{\epsilon^d} | K(\frac{\|\textbf{v}\|_{\mathbb{R}^d}}{\epsilon}) | P(\exp_x(\textbf{v}))V_x(\textbf{v}) d\textbf{v} \nonumber \\
\leq & \tilde{C}_1C_3 \delta^{2\kappa}+P_{Max}V_{Max}\tilde{C}_2 \gamma^{\alpha+d} \nonumber \\
\leq & \tilde{C}_1C_3 \gamma^{2\kappa}+P_{Max}V_{Max}\tilde{C}_2 \gamma^{\alpha+d}, \nonumber
\end{align}
where we use (4) in Lemma \ref{properties of the map phi x} and  \eqref{integral of kernel outside delta/2  ball} in the second last step.  

Therefore,  \eqref{In side and outside D delta in TM} can be bounded by
\begin{align}
\tilde{C}_1 C_2 \gamma^{\kappa} + P_{Max}\tilde{C}_2 \gamma^{\alpha+d}+\tilde{C}_1C_3 \gamma^{2\kappa}+P_{Max}V_{Max}\tilde{C}_2 \gamma^{\alpha+d} .
\end{align}
Now, if we sum up \eqref{inside M delta x}, we have
\begin{align}\label{bias analysis final step}
&\left|\int_{M}\frac{1}{\epsilon^d} K(\frac{\|\iota(y)-\iota(x)\|_{\mathbb{R}^p}}{\epsilon})P(y)dV(y)-P(x)\right| \\
\leq & (\frac{1}{2\rho+1} )^{\frac{\alpha-d}{2}}(\frac{\gamma}{2})^\alpha +\tilde{C}_1 C_2 \gamma^{\kappa} + P_{Max}\tilde{C}_2 \gamma^{\alpha+d}+\tilde{C}_1C_3 \gamma^{2\kappa}+P_{Max}V_{Max}\tilde{C}_2 \gamma^{\alpha+d}  \nonumber \\
\leq & \left[(\frac{1}{2\rho+1} )^{\frac{\alpha-d}{2}}(\frac{1}{2})^\alpha+\tilde{C}_1 C_2 + P_{Max}\tilde{C}_2 +\tilde{C}_1C_3 +P_{Max}V_{Max}\tilde{C}_2 \right] \gamma^{\kappa} \,,\nonumber
\end{align}
where we use $\kappa \leq 1 \leq d <\alpha$  in the last step. 

We set $\omega_1$ and $\omega_2$ to be the constants in \eqref{relation between epsilon and gamma} and \eqref{bias analysis final step} respectively. As a summary, if $0<\gamma<1$ and $\epsilon \leq \omega_1 \gamma^{\frac{2\alpha}{(\alpha-d)}}$, then
\begin{align}
\sup_{x \in M}\left|\int_{M}\frac{1}{\epsilon^d} K(\frac{\|\iota(y)-\iota(x)\|_{\mathbb{R}^p}}{\epsilon})P(y)dV(y)-P(x)\right|\leq \omega_2 \gamma^\kappa.
\end{align}
$\omega_1$ depends $\rho$, $\alpha$ and $d$ and $\omega_2$ depends on $\rho$, $\alpha$, $d$, $\kappa$, $C_P$, $P_{Max}$, $K_{sup}$, the curvature of $M$ and the second fundamental form of $\iota(M)$.

Suppose $K(t)$ has compact support on $[0, \rho]$, then we take $\alpha \rightarrow \infty$ in \eqref{relation between epsilon and gamma} and \eqref{bias analysis final step}.  If $0<\gamma<1$ and $\epsilon \leq \omega_1 \gamma^{2}$, then
\begin{align}
\sup_{x \in M}\left|\int_{M}\frac{1}{\epsilon^d} K(\frac{\|\iota(y)-\iota(x)\|_{\mathbb{R}^p}}{\epsilon})P(y)dV(y)-P(x)\right|\leq \omega_2 \gamma^\kappa.
\end{align}
$\omega_1$   depends on $\rho$. $\omega_2$  depends on $\rho$, $d$, $\kappa$, $C_P$, $P_{Max}$, $K_{sup}$, the curvature of $M$ and the second fundamental form of $\iota(M)$.

\underline{\textbf{Proof of Theorem \ref{density estimation}}}

The proof is a concatenation of the variance and bias analysis.

To derive (1) in Theorem \ref{density estimation}, we use Theorem \ref{EK_n-K_n} and (1) in Theorem \ref{bias analysis main theorem}. Note that $ \gamma^{\frac{2\alpha}{\alpha-d}}<\gamma^{\frac{1}{\alpha-d}}$, when $0<\gamma<1$. Therefore, we have the following conclusion. Suppose $0<\gamma<\min\{\rho^{-\alpha},1\}$ and $\epsilon_n \rightarrow 0$ as $n \rightarrow \infty$. If $\epsilon_n \leq \Omega_3 \gamma^{\frac{2\alpha}{\alpha-d}}$, then 
\begin{align}
P\{\sup_{x \in M}|K_n(x)-P(x)| \leq \Omega_1 (\gamma^{1-\frac{d}{\alpha}}+\gamma^\kappa)\} \geq 1-8(2n)^{2p^2+2p}e^{-\Omega_2 \bigg(n \epsilon_n^{d+\frac{d^2}{\alpha}}\bigg)\bigg(\frac{\gamma^{2-\frac{d}{\alpha}}}{N(\gamma)^2}\bigg)}.
\end{align}
$\Omega_1=\mathcal{D}_1+\omega_2$ which depends on $p$, $\rho$, $d$, $\alpha$, $\kappa$, $C_P$, $P_{Max}$, $K_{sup}$, the curvature of $M$ and the second fundamental form of $\iota(M)$. $\Omega_2=\mathcal{D}_2$ which depends on $p$, $d$, $\alpha$, $P_{Max}$, $K_{sup}$ and the second fundamental form of $\iota(M)$. $\Omega_3=\min\{\omega_1, \mathcal{D}_3\}$ which depends on $p$, $\rho$, $\alpha$, $d$, $P_{Max}$ and the second fundamental form of $\iota(M)$. 

To derive (2) in Theorem \ref{density estimation}, we use Corollary \ref{compact support variance} and (2) in Theorem \ref{bias analysis main theorem}. Note that $\gamma^{2}<1$ and $\gamma\leq \gamma^\kappa$ when $0<\gamma<1$. Therefore, we have the following conclusion. Suppose $K(t)$ has compact support on $[0, \rho]$. Suppose $0<\gamma<1$ and $\epsilon_n \rightarrow 0$, as $n \rightarrow \infty$.  If $\epsilon_n \leq \Omega_3\gamma^{2}$, then
\begin{align}
P\{\sup_{x \in M}|K_n(x)-P(x)| \leq \Omega_1 \gamma^\kappa \} \geq 1-8(2n)^{2p^2+2p}e^{-\Omega_2\frac{n \epsilon_n^{d} \gamma^2}{N(\gamma)^2}}.
\end{align}
$\Omega_1=\mathcal{D}_1+\omega_2$ which depends on $p$, $\rho$, $d$, $\kappa$, $C_P$, $P_{Max}$, $K_{sup}$, the curvature of $M$ and the second fundamental form of $\iota(M)$. $\Omega_2=\mathcal{D}_2$ which depends on $p$, $d$, $P_{Max}$, $K_{sup}$, and the second fundamental form of $\iota(M)$. $\Omega_3=\min\{\omega_1, \mathcal{D}_3\}$ which depends on $\rho$,  $P_{Max}$ and the second fundamental form of $\iota(M)$. 

To derive (3) in Theorem \ref{density estimation}, we use Proposition \ref{0-1 kernel} and (2) in Theorem \ref{bias analysis main theorem}. For $K(t)=\frac{d}{|S^{d-1}| \rho^d}\bigchi_{[0,\rho]}(\textbf{t})$, $K_n(x)=\frac{N_{\epsilon_n \rho}(x)}{n \epsilon_n^d}$.
Note that $\gamma^{2}<1$ and $\gamma\leq \gamma^\kappa$ when $0<\gamma<1$. Therefore, we have the following conclusion. Suppose $0<\gamma<1$ and $\epsilon_n \rightarrow 0$, as $n \rightarrow \infty$.  If $\epsilon_n \leq \Omega_3\gamma^{2}$, then
\begin{align}
P\{\sup_{x \in M}|\frac{N_{\epsilon_n \rho}(x)}{n \epsilon_n^d}-P(x)| \leq \Omega_1 \gamma^\kappa \} \geq 1-8(2n)^{p+2}e^{-\Omega_2  n \epsilon_n^d \gamma^2}.
\end{align}
$\Omega_1=\mathcal{D}_1+\omega_2$ which depends on $\rho$, $d$, $\kappa$, $C_P$, $P_{Max}$, $K_{\sup}$, the curvature of $M$ and the second fundamental form of $\iota(M)$. $\Omega_2=\mathcal{D}_2$ which depends on $d$, $\rho$, $P_{Max}$ and the second fundamental form of $\iota(M)$. $\Omega_3=\min\{\omega_1, \mathcal{D}_3\}$ which depends on $\rho$,  $P_{Max}$ and the second fundamental form of $\iota(M)$. 

Choose $\epsilon_n= \Omega_3 \gamma^2$ and $8(2n)^{p+2}e^{-\Omega_2 n \epsilon_n^d \gamma^2}=\frac{1}{n^2}$, the last statement in the theorem follows.
\

\section{Proof of Theorem \ref{Riemann integrable main theorem}}\label{Appendix D}

We are going to use the following linear algebra lemma:
\begin{lemma}\label{LA lemma 1}
Suppose $\{e_1, \cdots, e_p\}$ is the standard basis of $\mathbb{R}^p$. $V$ is a $d$ dimensional subspace of $\mathbb{R}^p$, such that $V$ is not perpendicular to $e_1$. Then there is a $p-d$ dimensional subspace of $\mathbb{R}^p$, $W$, generated by $\{e_{i_1}, \cdots, e_{i_{p-d}}\}$ such that $W$ is perpendicular to $e_1$ and $W \cap V=\{0\}$.
\end{lemma}

\begin{proof}
Suppose $v_1, v_2, \cdots, v_d$ forms a basis of $V$. We can write them as a $d \times p$ matrix.
\begin{align}
A_1=\begin{bmatrix}
v^\top_{1} \\
\cdots  \\
v^\top_{p-d}.
\end{bmatrix}
\end{align}
Since $V$ is not perpendicular to $e_1$, by using elementary row operations, we can turn $A_1$ into 
\begin{align}
A_2=\begin{bmatrix}
1 & v^\top \\
0 & B  \\
\end{bmatrix},
\end{align}
where $B$ is a $(d-1) \times (p-1)$ matrix. Since the row vectors of $A_2$ also form a basis of $V$, $rank(B)=d-1$. If we enumerate the column vectors of $B$ as $b_2, \cdots, b_{p}$, then we can find $d-1$ column vectors,  $b_{j_1}, \cdots, b_{j_{d-1}}$ of $B$ that are linearly independent. The rest of the column vectors of $B$ are $b_{i_1}, \cdots, b_{i_{p-d}}$. 

We claim that $W$ is generated by  $\{e_{i_1}, \cdots, e_{i_{p-d}}\}$. First, $e_1$ is not in $W$. Hence, $e_1$ is perpendicular to $W$.  To prove that $V \cap W=0$, we prove that $\{e_{i_1}, \cdots, e_{i_{p-d}}\}$ together with the row vectors of $A_2$ generate $\mathbb{R}^p$. In fact, let $B^*$ be $(d-1) \times (d-1)$ matrix whose columns are $b_{j_1}, \cdots, b_{j_{d-1}}$. Then, determinant of $B^*$ is not zero.  Then, by the straightforward calculation, the determinant of 
\begin{align}
\begin{bmatrix}
e^\top_{i_1} \\
\vdots, \\
e_{i_{p-d}} \\
A_2 
\end{bmatrix}
\end{align}
is equal to determinant of $B^*$. Hence, $V$ and $W$ together generate $\mathbb{R}^p$, which implies $V \cap W=0$.
\end{proof}

\begin{lemma}\label{constant lemma}
Let $U \subset M$ be a connected open set. $f$ is a smooth real function on $U$. If $f$ maps $U$ to the set of the critical values , then $f$ is constant on $U$.
\end{lemma}
\begin{proof}
The property is local, hence we can prove the theorem in a chart of $M$. If we write $f$ in the coordinates, then it suffices to prove the following: let $U \subset \mathbb{R}^d$ be a connected open set. $f$ is a smooth real function on $U$. If $f$ maps $U$ to the set of the critical values , then $f$ is constant on $U$.

If $(\nabla f)^{-1}(\mathbb{R}^d \setminus \{(0,\cdots, 0)\})$ is not empty, then it is an open subset of $U$. Then we can find an open ball in $U$ such that the gradient of $f$ in $U$ is not zero. By the fundamental theorem of calculus, if we integrate the gradient along a line in the open ball, then we know that the image of $f$ should at least contain an open interval. Hence the set of the critical values of $f$ contains an open interval. This contradicts to Sard's theorem. Therefore $(\nabla f)^{-1}(\mathbb{R}^d \setminus \{(0,\cdots, 0)\})$ is empty. Since $U$ is connected, $f$ is a constant on $U$.
\end{proof}

\begin{proof}[\textbf{Proof of Theorem \ref{Riemann integrable main theorem}}]
We use $\partial A$ to denote the topological boundary of a set $A$. Fix $x$, we denote the set of critical points of $D_x(\textbf{u})$ on $\iota(M)$ as $\mathcal{C}(x)$ and the set of the critical values of $D_x(\textbf{u})$ as $\mathcal{V}(x)$. 

Note that $D_x(\textbf{u})$ is smooth except at $\textbf{u}=\iota(x)$. Hence $\mathcal{C}(x)$ consists of $\iota(x)$ and the critical points of $D_x(\textbf{u})$ on $\iota(M) \setminus \iota(x)$. The set of critical points of $D_x(\textbf{u})$ on $\iota(M) \setminus \iota(x)$ is the preimage of $0$ of the derivative of $D_x(\textbf{u})$, therefore it is a closed subset of  $\iota(M) \setminus \iota(x)$. $\iota(M)$ is the one point compactification of $\iota(M) \setminus \iota(x)$, so $\mathcal{C}(x)$ is a closed subset of $\iota(M)$, and hence a compact subset of $\iota(M)$. Moreover, it is worth to mention that $\iota(x)$ is an isolated critical point of  $D_x(\textbf{u})$.

\
(1) Suppose that fix $x$, $\mathcal{C}(x)$ is Jordan measurable. Without loss of generality, we assume $\epsilon=1$, and $\iota(x)=0$. We need to prove $K(\|\textbf{u}\|_{\mathbb{R}^p})$ is Riemann integrable on $\iota(M)$ when $\partial \mathcal{C}(x)$ has measure $0$. We prove this in the following three steps.

Step 1. Suppose $\textbf{u}_0 \in \iota(M) \setminus\{0\}$ and it is not a critical point of $D_x(\textbf{u})=\|\textbf{u}\|_{\mathbb{R}^p}$ on $ \iota(M)$. Then we will show that we can construct a local diffeomorphism from an open set around $\textbf{u}_0$ in $\iota(M)$ to an open set in $\mathbb{R}^d$ whose boundary has measure $0$ by inverse function theorem.  

Let $a=\|\textbf{u}_0\|_{\mathbb{R}^p}$. Let $\psi_1(\textbf{u})$ be a rotation of $\mathbb{R}^p$ such that $\psi_1(\textbf{u}_0)=(a,0, \cdots, 0)$. By rotation invariance, $\textbf{u}_0$ is not a critical point of $D_x(\textbf{u})$ on $\iota(M)$ if and only if $(a,0, \cdots, 0)$ is not a critical point of $D_x(\textbf{u})$ on $\psi_1(\iota(M))$. Note that the tangent space $V$ of $\psi_1(\iota(M))$ at $(a,0,\cdots,0)$ is a $d$ dimensional subspace of $\mathbb{R}^p$. The gradient of $D_x(\textbf{u})$ at $(a,0, \cdots, 0)$ is in $e_1$ direction. Since $(a,0, \cdots, 0)$ is not a critical point of $D_x(\textbf{u})$ on $\psi_1(\iota(M))$, $V$ is not perpendicular to $e_1$.  By Lemma \ref{LA lemma 1}, we can find a subspace $W$ generated by $\{e_{i_1}, \cdots, e_{i_{p-d}}\}$ such that $W$ is perpendicular to $e_1$ and $W \cap V=\{0\}$. Let $W^\bot$ be the orthogonal complement of $W$, then $e_1$ is in $W^\bot$. For notation simplicity, we assume $W^\bot$ is generated by $\{e_1, \cdots, e_d\}$. Define 
\begin{align}
\psi_2(\textbf{u})=(\frac{\|\textbf{u}\|_{\mathbb{R}^p}u_1}{\sqrt{u_1^2+\cdots+u_d^2}}, \cdots, \frac{\|\textbf{u}\|_{\mathbb{R}^p}u_d}{\sqrt{u_1^2+\cdots+u_d^2}}).
\end{align}
We need to prove that $\psi_2$ is a local diffeomorphism of $\psi_1(\iota(M))$ around $(a,0,\cdots, 0)$.  Suppose $\tau(\textbf{v})$ is a chart of $\psi_1(\iota(M))$ around $(a,0,\cdots, 0)$ such that $\tau(0)=(a,0,\cdots, 0)$. Then  $[D\psi_2(\tau(0))]=[D\psi_2(a,0,\cdots,0)][D\tau(0)]$. Here $[D\tau(0)]$ is $p \times d$ matrix whose column vectors form a basis of $V$. A straight forward calculation shows that $[D\psi_2(a,0,\cdots,0)]$ is  $d \times p$ matrix whose row vectors are $\{e_1, \cdots, e_d\}$. If the null space of $[D\psi_2(\tau(0))]$ is not $0$, then there is a non zero vector $\textbf{v}$ in $V$ such that $\textbf{v}$ is perpendiuclar to $e_1, \cdots, e_d$. Since $W^\bot$ is generated by $\{e_1, \cdots, e_d\}$, so $\textbf{v}$ is in $W$, and this contradicts to $W \cap V=\{0\}$. Therefore, the null space of $[D\psi_2(\tau(0))]$ is $0$ and $Det ([D\psi_2(\tau(0))])$ is not $0$. By inverse function theorem, $\psi=\psi_2 \circ \psi_1$ is a diffeomorphsim from an open set $\mathcal{O}$ around $\textbf{u}_0$ to an open set in $\mathbb{R}^d$ whose boundary has measure $0$. Moreover, for any $\textbf{u}$, $\|\textbf{u}\|_{\mathbb{R}^p}=\|\psi(\textbf{u})\|_{\mathbb{R}^d}$. Hence, $K(\|\textbf{u}\|_{\mathbb{R}^p})=K(\|\psi(\textbf{u})\|_{\mathbb{R}^d})$.

Step 2.  $K(\|\textbf{u}\|_{\mathbb{R}^p})$ is continuous in the interior of $\mathcal{C}(x)$. In fact, for any open geodesic ball in the interior, $K(\|\textbf{u}\|_{\mathbb{R}^p})$ is a constant over the ball. Hence it is constant on each connected component of the interior of $\mathcal{C}(x)$.

Step 3. For any $\delta>0$, since $\iota(M)$ is compact and $\partial \mathcal{C}(x)$ has measure $0$, $\partial \mathcal{C}(x)$ has content $0$. Therefore, we can cover the boundary by finite many uniform open geodesic balls of radius $r(\delta)$ so that the total volume of the open balls of radius $2r(\delta)$ with the same centers is less than $\frac{\delta}{2}$. The union of these open balls of radius $2r(\delta)$ with the interior of  $\mathcal{C}(x)$ is an open set. Denote the complement of the union in $\iota(M)$ as  $M'$. $M'$ is a compact subset of $\iota(M)$. We can construct a cover of $M'$ by finitely many open sets $\{\mathcal{O}_i\}_{i=1}^N$ satsifying the following properties:
\begin{enumerate}
\item Each $\mathcal{O}_i$ is an open geodesic ball with radius less then $r(\delta)/2$.
\item There is a diffeomorphism  $\psi^i$ associated with $\mathcal{O}_i$ as constructed in Step 1. Each $\mathcal{O}_i$ is diffeomorphic through  
$\psi^i$ to an open set $\mathcal{O}^*_i$ in $\mathbb{R}^d$ whose boundary has measure $0$.
\end{enumerate}
Obviously, the function $K(\|\textbf{v}\|_{\mathbb{R}^d})$ is Riemann integrable on $\mathcal{O}^*_i$, in particular the set where $K(\|\textbf{v}\|_{\mathbb{R}^d})$ is discontinuous is a measure $0$ subset of the interior of $\mathcal{O}^*_i$. Note that by the construction of the open cover of $\partial \mathcal{C}(x)$, we can make sure that each $\mathcal{O}_i$ is at least $r(\delta)/2$ away from the closure of $ \mathcal{C}(x)$. Hence, there is a constant $C(r(\delta))$, such that $|Det(D\psi^i)|>C(r(\delta))>0$. Therefore if we cover the discontinuous set of $K(\|\textbf{v}\|_{\mathbb{R}^d})$ in $\mathcal{O}^*_i$ (including the boundary) by countable open balls with total volume less than $\frac{C(r(\delta))\delta}{2 N}$, then the discontinuous set of $K(\|\textbf{u}\|_{\mathbb{R}^p})=K(\|\psi^i(\textbf{u})\|_{\mathbb{R}^d})$ in $\mathcal{O}_i \cap M'$ can be covered by countable open sets with  total volume less than $\frac{\delta}{2N}$. So the discontinuous set of $K(\|\textbf{u}\|_{\mathbb{R}^p})$ in $M'$ can be covered by countable open sets with  total volume less than $\frac{\delta}{2}$ and  the discontinuous set of $K(\|\textbf{u}\|_{\mathbb{R}^p})$ in $\iota(M)$ can be covered by countable open sets with  total volume less than $\delta$.
Hence, $K(\|\textbf{u}\|_{\mathbb{R}^p})$ is Riemann integrable on $\iota(M)$ and the conclusion follows.

\
(2) If for any $\epsilon>0$, $K(\frac{\|\iota(y)-\iota(x)\|_{\mathbb{R}^p}}{\epsilon})$ is a Riemann integrable function of $y$ on the manifold for all $K(t)$ satisfies the conditions in Assumption \ref{assumptions on kernel 1}, then obviously $K(\|\textbf{u}-\iota(x)\|_{\mathbb{R}^p})$ is a Riemann integrable function of $\textbf{u}$ on $\iota(M)$  for all $K(t)$ satisfies the conditions in Assumption \ref{assumptions on kernel 1}.

Since $D_x(\textbf{u})$ is continuous, $\mathcal{V}(x)$ is compact. Moreover, $\mathcal{V}(x)$ is the union of $0$ and the critical values of $D_x(\textbf{u})$ on $\iota(M) \setminus \iota(x)$. Note that $D_x(\textbf{u})$ is smooth except at $\textbf{u}= \iota(x)$. By Sard's theorem, the set of critical values of $D_x(\textbf{u})$ on $\iota(M) \setminus \iota(x)$ has measure $0$. Hence $\mathcal{V}(x)$ is compact and has measure $0$. Define $K(t)=\bigchi_{\mathcal{V}(x)}(t)$. Since $\mathcal{V}(x)$ is Jordan measurable,  $K(t)$ is bounded and Riemann integrable. We have $K(\|\textbf{u}-\iota(x)\|_{\mathbb{R}^p})=\bigchi_{D^{-1}_x(\mathcal{V}(x))}(\textbf{u})$ is Riemann integrable on $\iota(M)$. Thus, we have $\partial D^{-1}_x(\mathcal{V}(x))$ has measure $0$.  

Next we show that $\partial \mathcal{C}(x) \subset \partial D^{-1}_x(\mathcal{V}(x))$. Since both $\mathcal{C}(x)$ and $D^{-1}_x(\mathcal{V}(x))$ are compact and $\mathcal{C}(x) \subset D^{-1}_x(\mathcal{V}(x))$, if there is a point $\textbf{u}_0$ on $\partial \mathcal{C}(x)$ but not on $\partial D^{-1}_x(\mathcal{V}(x))$, then $\textbf{u}_0$ is in the interior of $D^{-1}_x(\mathcal{V}(x))$. Hence, we can find an open geodesic ball $B \subset D^{-1}_x(\mathcal{V}(x))$ on $\iota(M)$ around $\textbf{u}_0$ so that $D_x$ is smooth on the open set $B \setminus \textbf{u}_0$. By Lemma \ref{constant lemma}, $D_x$ is constant on $B \setminus \textbf{u}_0$.  Hence $B \subset \mathcal{C}(x)$ and this contradicts to $\textbf{u}_0$ on $\partial \mathcal{C}(x)$. We prove the claim  $\partial \mathcal{C}(x) \subset \partial D^{-1}_x(\mathcal{V}(x))$. Since  $\partial D^{-1}_x(\mathcal{V}(x))$ has measure $0$, $\partial \mathcal{C}(x)$ has measure $0$.
\end{proof}

\

\section{Proof of the claim in Example \ref{not VC class}} \label{proof of not VC class}

Let $d_{L^2}$ be the metric on $\mathcal{F}(A)$ defined as in \eqref{general L2 metric}, where $\mathcal{P}$ is the uniform probability measure on $[0,1]$. Then, we have
\begin{align}
N_{pack}(\epsilon, \mathcal{F}(A), d_{L^2}) \leq N_{cov}(\epsilon, \mathcal{F}(A), d_{L^2}) \leq \sup_{\mathcal{P}} N_{cov}(\epsilon, \mathcal{F}(A), d_{L^2(\mathcal{P})}),
\end{align}
where $N_{pack}(\epsilon, \mathcal{F}(A), d_{L^2})$ is the packing number of $\mathcal{F}(A)$ by the $\epsilon$ ball with respect to the metric $d_{L^2}$.
Next, we find a lower bound for $N_{pack}(\epsilon, \mathcal{F}(A), d_{L^2})$. Let $0 \leq a, a+\epsilon, \leq 1$, we need to find a lower bound for $d_{L^2}(K(a+\epsilon - \cdot ), K(a - \cdot ))$. Note that by the Cauchy Schwarz inequality,
\begin{align}
d_{L^2}(K(a+\epsilon - \cdot ), K(a - \cdot ))=& (\int_{0}^1 |\sin(\exp(\exp(\frac{1}{|a+\epsilon-z|})))-\sin(\exp(\exp(\frac{1}{|a-z|})))|^2 dz) ^{\frac{1}{2}} \\
= & (\int_{0}^1 |\sin(\exp(\exp(\frac{1}{|a+\epsilon-z|})))-\sin(\exp(\exp(\frac{1}{|a-z|})))|^2 dz) ^{\frac{1}{2}} (\int_{0}^1 1^2 dz)^{\frac{1}{2}} \nonumber\\
\geq & \int_{0}^1 |\sin(\exp(\exp(\frac{1}{|a+\epsilon-z|})))-\sin(\exp(\exp(\frac{1}{|a-z|})))| dz. \nonumber
\end{align}
Hence, we need to find a lower bound for $\int_{0}^1 |\sin(\exp(\exp(\frac{1}{|a+\epsilon-z|})))-\sin(\exp(\exp(\frac{1}{|a-z|})))| dz$. For notation simplicity, we assume $a=0$ to find the lower bound. The method and the lower bound will be the same for a general $a$. Note that if $\epsilon<\frac{1}{10}$, then $0<-\frac{1}{\log \epsilon}<\frac{1}{2}$ and $0 <\epsilon< \epsilon-\frac{1}{\log \epsilon} <1$. Hence, for $\epsilon$ small enough, we have
\begin{align}
\int_{0}^1 |\sin(\exp(\exp(\frac{1}{|\epsilon-z|})))-\sin(\exp(\exp(\frac{1}{|z|})))|  dz \geq & \int_{\epsilon}^{\epsilon-\frac{1}{\log \epsilon}} |\sin(\exp(\exp(\frac{1}{|\epsilon-z|})))-\sin(\exp(\exp(\frac{1}{|z|})))| dz \\
=& \int_{\epsilon}^{\epsilon-\frac{1}{\log \epsilon}} |\sin(\exp(\exp(\frac{1}{z-\epsilon})))-\sin(\exp(\exp(\frac{1}{z})))|  dz. \nonumber
\end{align}
We compare the function $\sin(\exp(\exp(\frac{1}{z-\epsilon})))$ and $\sin(\exp(\exp(\frac{1}{z})))$.  Let 
$$a_i=\epsilon + \frac{1}{\log(\log((N+i)\pi))},$$ 
where $i$ is an integer (may be negative) and $N$ is the smallest  integer such that $\frac{1}{\log(\log(N\pi))} \leq  -\frac{1}{\log \epsilon}$. In other words,  $N-1 \leq \frac{1}{\pi}e^{\frac{1}{ \epsilon}} \leq N$. Let 
$$a^*_i=\epsilon + \frac{1}{\log((N+i)\pi+\frac{\pi}{2})}.$$
Note that if $z=a_i$, then $\exp(\exp(\frac{1}{z-\epsilon}))=(N+i) \pi$. Intuitively, $[a_{i+1},a_i]$ is a ``half period'' of $\sin(\exp(\exp(\frac{1}{z-\epsilon})))$. In other words, $\sin(\exp(\exp(\frac{1}{z-\epsilon})))$ is completely positive or negative on $(a_{i+1},a_i)$. Let 
$$b_j=\frac{1}{\log(\log((M_1+j)\pi))},$$ where $j \leq M_2$. $M_1$ is the smallest postive integer such that $\frac{1}{\log(\log(M_1\pi))} \leq \epsilon -\frac{1}{\log \epsilon}$. $M_2$ is the largest non negative integer such that $\frac{1}{\log(\log((M_1+M_2)\pi))} \geq \epsilon $. Hence, $M_1+M_2 \leq \frac{1}{\pi}\exp(\exp(\frac{1}{\epsilon})) \leq M_1+M_2+1$.  We have the following observations about $[a_{i+1},a_i]$ and $[b_{j+1},b_{j}]$.

\begin{lemma} \label{lemma on intervals}
If $\epsilon$ is small enough, then we have the following statements.
\begin{enumerate}
\item
$|a_i-a_{i+1}|>|a_{i+1}-a_{i+2}|$ and $|b_j-b_{j+1}|>|b_{j+1}-b_{j+2}|$.
\item
 $1<\frac{|a_{i}-a^*_{i}|}{|a^*_i-a_{i+1}|} <2$. 
\item
$\frac{1}{2} <\frac{|a_{i+6}-a_{i+7}|}{|a_i-a_{i+1}|}<1 $.
\end{enumerate}
\end{lemma}

\begin{proof}
Note that $\frac{1}{\pi}e^{\frac{1}{ \epsilon}} \leq N$ implies that $N \rightarrow \infty$ as $\epsilon \rightarrow 0$.

(1) $a_i-a_{i+1}=\frac{1}{\log(\log((N+i)\pi))}-\frac{1}{\log(\log((N+i+1)\pi))}$. The result follows from $\frac{1}{\log(\log(\pi x))}-\frac{1}{\log(\log(\pi(x+1)))}$ is decreasing for $x>0$. Similarly, we have $|b_j-b_{j+1}|>|b_{j+1}-b_{j+2}|$.

(2) $a_{i}-a^*_{i}= \frac{1}{\log(\log((N+i)\pi))}-\frac{1}{\log(\log((N+i)\pi+\frac{\pi}{2}))}$ and $a^*_i-a_{i+1}=\frac{1}{\log(\log((N+i)\pi+\frac{\pi}{2}))}-\frac{1}{\log(\log((N+i+1)\pi))}$. Hence, 
\begin{align}
\frac{|a_{i}-a^*_{i}|}{|a^*_i-a_{i+1}|}=\frac{\frac{1}{\log(\log((N+i)\pi))}-\frac{1}{\log(\log((N+i+\frac{1}{2})\pi))}}{\frac{1}{\log(\log((N+i+\frac{1}{2})\pi))}-\frac{1}{\log(\log((N+i+1)\pi))}}.
\end{align}
The conclusion follows from   $\frac{\frac{1}{\log(\log(\pi x))}-\frac{1}{\log(\log(\pi(x+\frac{1}{2})))}}{\frac{1}{\log(\log(\pi(x+\frac{1}{2})))}-\frac{1}{\log(\log(\pi(x+1)))}}$ is decreasing to $1$ as $x \rightarrow \infty$.

(3) 
\begin{align}
\frac{|a_{i+1}-a_{i+2}|}{|a_i-a_{i+1}|}=\frac{\frac{1}{\log(\log((N+i+1)\pi))}-\frac{1}{\log(\log((N+i+2)\pi))}}{\frac{1}{\log(\log((N+i)\pi))}-\frac{1}{\log(\log((N+i+1)\pi))}}.
\end{align}
Note that $\frac{\frac{1}{\log(\log(\pi (x+1)))}-\frac{1}{\log(\log(\pi(x+2)))}}{\frac{1}{\log(\log(\pi x))}-\frac{1}{\log(\log(\pi(x+1)))}}$ is increasing to $1$ as $x \rightarrow \infty$. Hence, if $\epsilon$ is small enough, for all $i$, we have $(\frac{1}{2})^\frac{1}{6}<\frac{|a_{i+1}-a_{i+2}|}{|a_i-a_{i+1}|}<1$. The conclusion follows.
\end{proof}

Denote $[\tilde{a}_{k+1}, \tilde{a}_k]$ to be an interval $[a_{i+1},a_i]$ such that $\sin(\exp(\exp(\frac{1}{z-\epsilon})))$ and $\sin(\exp(\exp(\frac{1}{z})))$ have the opposite sign on $[a_{i+1},a_i]$. The existence of $[\tilde{a}_{k+1}, \tilde{a}_k]$ and the total length of $\cup_k [\tilde{a}_{k+1}, \tilde{a}_k]$ are proved in the following lemma.

\begin{lemma}\label{total length of good intervals} If $\epsilon$ is small enough,
\begin{align}
\sum_{k}|\tilde{a}_k-\tilde{a}_{k+1}| >-\frac{1}{20 \log \epsilon}.
\end{align}
\end{lemma}

\begin{proof}
Let $g(z)=\exp(\exp(\frac{1}{z}))$. Consider the interval $[b_{j+1},b_j]$. By the mean value theorem, there is a $b_{j+1} \leq m \leq b_j$ such that
\begin{align}
\frac{d g(m)}{dz}=\frac{g(b_{j+1})-g(b_j)}{b_{j+1}-b_j} = \frac{\pi}{b_{j+1}-b_j}.
\end{align}
Then, we find $a_i$ such that $a_{i+1}<m<a_i$. By the mean value theorem, there is a $m'$ with $\epsilon < a_{i+6} \leq m' \leq a_{i+1}<m$ such that
\begin{align}
\frac{d g(m'-\epsilon)}{dz}=\frac{g(a_{i+6}-\epsilon)-g(a_{i+1}-\epsilon)}{a_{i+6}-\epsilon-(a_{i+1}-\epsilon)}=\frac{5\pi}{a_{i+6}-a_{i+1}}.
\end{align}
Hence, 
\begin{align}
\frac{1}{5}\frac{\frac{d g(m'-\epsilon)}{dz}}{ \frac{dg(m)}{dz}} (a_{i+1}-a_{i+6})=b_{j}-b_{j+1}.
\end{align}
By (1) and (3) in Lemma \ref{lemma on intervals}, if  $\epsilon$ is small enough, $a_{i-5}-a_{i+6}<3(a_{i}-a_{i+6})$. Therefore,
\begin{align}
\frac{1}{15}\frac{\frac{d g(m'-\epsilon)}{dz}}{ \frac{dg(m)}{dz}} (a_{i-5}-a_{i+6})<b_{j}-b_{j+1}.
\end{align}
Observe that for $z>0$, we have $g'(z)<0$ and $g''(z)>0$. Since $m' \leq m$, we have
\begin{align}
\frac{1}{15}\frac{\frac{d g(m'-\epsilon)}{dz}}{ \frac{dg(m')}{dz}} (a_{i-5}-a_{i+6})<b_{j}-b_{j+1}.
\end{align}
Next we show that $\frac{\frac{d g(m'-\epsilon)}{dz}}{ \frac{dg(m')}{dz}}>15$, when $\epsilon$ is small enough. For $z>\epsilon$,
\begin{align}
\frac{\frac{d g(z-\epsilon)}{dz}}{\frac{dg(z)}{dz}}=\frac{g(z-\epsilon) e^{\frac{1}{z-\epsilon}}(z-\epsilon)^{-2}}{g(z) e^{\frac{1}{z}}(z)^{-2}}>\frac{g(z-\epsilon)}{g(z)}.
\end{align}
Note that by the quotient rule $\frac{d\frac{g(z-\epsilon)}{g(z)}}{dz}=\frac{g(z-\epsilon)(z^{-2}e^{\frac{1}{z}}-(z-\epsilon)^{-2}e^{\frac{1}{z-\epsilon}})}{g(z)}<0$. 

Hence,  For $\epsilon<z<\epsilon-\frac{1}{\log \epsilon}$,
\begin{align}
\frac{\frac{d g(z-\epsilon)}{dz}}{\frac{dg(z)}{dz}}>\frac{g(z-\epsilon)}{g(z)}>\frac{g(-\frac{1}{\log \epsilon})}{g(\epsilon-\frac{1}{\log \epsilon})}.
\end{align}
Note that
\begin{align}
\frac{1}{\epsilon-\frac{1}{\log \epsilon}} = -\log \epsilon ( \frac{1}{1-\epsilon\log\epsilon}). 
\end{align}
Since $\epsilon \log \epsilon<0$ and goes to $0$ as $\epsilon$ goes to $0$, if $\epsilon$ is small enough so that $\epsilon \log \epsilon >-1$, then
\begin{align}
\frac{1}{\epsilon-\frac{1}{\log \epsilon}}\leq - (\log\epsilon )(1+\frac{1}{2}\epsilon \log \epsilon)=- \log\epsilon- \frac{1}{2}\epsilon \log^2 \epsilon.
\end{align}
Hence,
\begin{align}
g(\epsilon-\frac{1}{\log \epsilon}) =\exp(\exp(\frac{1}{\epsilon-\frac{1}{\log \epsilon}}))\leq e^{\frac{1}{\epsilon}e^{-\frac{1}{2}\epsilon \log^2 \epsilon}},
\end{align}
At last, note that $1-e^{-\frac{1}{2}x}>\frac{1}{4}x$ when $x<3$. Hence, if $\epsilon$ is small enough, then $\epsilon \log^2 \epsilon<3$ and 
\begin{align}
\frac{g(-\frac{1}{\log \epsilon})}{g(\epsilon-\frac{1}{\log \epsilon})}\geq e^{\frac{1}{\epsilon}[1-e^{-\frac{1}{2}\epsilon \log^2 \epsilon}]} \geq e^{\frac{1}{4} \log^2 \epsilon}.
\end{align}
We require $\epsilon$ to be even smaller, we have
\begin{align}
\frac{g(-\frac{1}{\log \epsilon})}{g(\epsilon-\frac{1}{\log \epsilon})} \geq e^{\frac{1}{4} \log^2 \epsilon}>15.
\end{align}
Therefore,
\begin{align}
a_{i-5}-a_{i+6}<b_{j}-b_{j+1}.
\end{align}
Since $m \in [a_{i+1},a_i] \cap [b_{j+1},b_j] \not= \emptyset$, either $[a_{i+6},a_{i+1}]$ or $[a_i,a_{i-5}]$ is in $[b_{j+1},b_j]$. Hence, each $[b_{j+1},b_j]$ contains at least five consecutive intervals  $[a_{i+1},a_{i}]$, while $[\epsilon, b_{M_2}]$ contains infinitely many $[a_{i+1},a_{i}]$. For $\epsilon$ small enough, choose $L$ to be the largest postive integer such that $\epsilon- \frac{1}{2 \log \epsilon}<b_1<a_L$. Then, for any $i \geq L$, there is a there is a $[\tilde{a}_{k+1}, \tilde{a}_k]$ in any five consecutive intervals $ [a_{i+5},a_{i+4}],\cdots, [a_{i+1},a_{i}] $. Hence, by  (1) and (3) in Lemma \ref{lemma on intervals}
\begin{align}
\sum_{[\tilde{a}_{k+1}, \tilde{a}_k]=[a_{i+1},a_i], i \geq L }|\tilde{a}_k-\tilde{a}_{k+1}| \geq \frac{1}{10}  \sum_{i \geq L}|a_i-a_{i+1}| =\frac{1}{10}|a_L-\epsilon| \geq \frac{1}{10}|b_1-\epsilon| = -\frac{1}{20 \log \epsilon}.
\end{align}
Therefore, 
\begin{align}
\sum_{k}|\tilde{a}_k-\tilde{a}_{k+1}| \geq \sum_{[\tilde{a}_{k+1}, \tilde{a}_k]=[a_{i+1},a_i], i \geq L }|\tilde{a}_k-\tilde{a}_{k+1}|  \geq  -\frac{1}{20 \log \epsilon}.
\end{align}
\end{proof}

Next, we find the lower bound for $\int_{\tilde{a}_{k+1}}^{\tilde{a}_k}|\sin(\exp(\exp(\frac{1}{z-\epsilon})))| dz$. 

\begin{lemma}\label{lower bound on good interval}
$\int_{\tilde{a}_{k+1}}^{\tilde{a}_k}|\sin(\exp(\exp(\frac{1}{z-\epsilon})))| dz >\frac{1}{8}|\tilde{a}_k-\tilde{a}_{k+1}|$. 
\end{lemma}

\begin{proof}
Let $g(z)=\exp(\exp(\frac{1}{z}))$, then the second derivative for $\sin(\exp(\exp(\frac{1}{z-\epsilon})))$ is 
\begin{align}
-\sin(g(z-\epsilon)) (g'(z-\epsilon))^2+\cos(g(z-\epsilon))g''(z-\epsilon).
\end{align}
Note that $g''(z-\epsilon)>0$ for $z >\epsilon$. Thus, if $\sin(\exp(\exp(\frac{1}{z-\epsilon})))$  is negative on $[a_{i+1},a_i]$, then $\cos(\exp(\exp(\frac{1}{z-\epsilon})))$ is positive on $[a_{i+1},a^*_i]$ and $\sin(\exp(\exp(\frac{1}{z-\epsilon})))$ is convex on $[a_{i+1},a^*_i]$. If $\sin(\exp(\exp(\frac{1}{z-\epsilon})))$  is positive on $[a_{i+1},a_i]$, then $\cos(\exp(\exp(\frac{1}{z-\epsilon})))$ is negative on $[a_{i+1},a^*_i]$ and $\sin(\exp(\exp(\frac{1}{z-\epsilon})))$ is concave on $[a_{i+1},a^*_i]$. By (2) in Lemma \ref{lemma on intervals},  $|a^*_i-a_{i+1}|>\frac{1}{4}|a_i-a_{i+1}|$. Therefore,
\begin{align}
\int_{\tilde{a}_{k+1}}^{\tilde{a}_k}|\sin(\exp(\exp(\frac{1}{z-\epsilon})))| dz >\frac{1}{2} \cdot \frac{1}{4}|\tilde{a}_k-\tilde{a}_{k+1}|=\frac{1}{8}|\tilde{a}_k-\tilde{a}_{k+1}|.
\end{align}
\end{proof}

Hence, 
\begin{align}
&\int_{\epsilon}^{\epsilon-\frac{1}{\log \epsilon}} |\sin(\exp(\exp(\frac{1}{z-\epsilon})))-\sin(\exp(\exp(\frac{1}{z})))|  dz  \\
 >& \int_{\cup_k [\tilde{a}_{k+1}, \tilde{a}_k]}|\sin(\exp(\exp(\frac{1}{z-\epsilon})))-\sin(\exp(\exp(\frac{1}{z})))|  dz  \\
 >& \int_{\cup_k[\tilde{a}_{k+1}, \tilde{a}_k]}|\sin(\exp(\exp(\frac{1}{z-\epsilon})))| dz  . \nonumber \\
 >& \frac{1}{8}\sum_{k}|\tilde{a}_k-\tilde{a}_{k+1}|>-\frac{1}{160 \log \epsilon}. \nonumber
\end{align}
Note that we apply Lemma \ref{lower bound on good interval} in the second last step and we apply Lemma \ref{total length of good intervals} in the last step.

Therefore, $d_{L^2}(K(a+\epsilon - \cdot ), K(a - \cdot )) >-\frac{1}{160 \log \epsilon}$. Equivalently, if $|x-y| > e^{-\frac{1}{160 \epsilon}}$, then $d_{L^2}(K(x - \cdot ), K(y - \cdot ))>\epsilon$. We conclude that $N_{pack}(\epsilon, \mathcal{F}(A), d_{L^2}) $ is greater than the packing number of $[0,1]$ by balls of radius $e^{-\frac{1}{160 \epsilon}}$. Hence, 
\begin{align}
\sup_{\mathcal{P}} N_{cov}(\epsilon, \mathcal{F}(A), d_{L^2(\mathcal{P})}) \geq N_{pack}(\epsilon, \mathcal{F}(A), d_{L^2}) \geq \frac{1}{2}e^{\frac{1}{160 \epsilon}}.
\end{align}
In other words, it is impossible that
\begin{align}
\sup_{\mathcal{P}} N_{cov}(\epsilon, \mathcal{F}(A), d_{L^2(\mathcal{P})}) \leq C \epsilon^{-b},
\end{align}
when $\epsilon$ is small enough.

\

\section{Proof of Propositions \ref{variance L1 case} and \ref{bias L1 case} }\label{Appendix C}
\begin{proof}[Proof of Proposition \ref{variance L1 case}]

We show that $\int_{M} |\frac{1}{n}\sum_{i=1}^n K_{\epsilon}(\iota(x),\iota(x_i))-\mathbb{E}K_{\epsilon}(x)| dV(x) \rightarrow 0$, a.s. as $n \rightarrow \infty$.

\underline{\textbf{Pointwise convergence of $\frac{1}{n}\sum_{i=1}^n K_{\epsilon}(\iota(x),\iota(x_i))$ to $\mathbb{E}K_{\epsilon}(x)$}}

\

Suppose $P$ is the density function of the random variable $X$ on $M$. Fix $x \in M$. Define a random variable $F:=K_{\epsilon}(\iota(x),\iota(X))$. Then, $F_i:=K_{\epsilon}(\iota(x),\iota(x_i))$ can be regarded as i.i.d samples from $F$. Note that if we use Lemma \ref{volume form in normal coordinate}, so that $Vol(B^{\mathbb{R}^p}_{\epsilon}(\iota(x))\cap \iota(M)) \leq C \epsilon^d$ for some constant $C$ depending on the manifold $M$, then, we have
\begin{align}
& b=\|K_{\epsilon}\|_{\infty} \leq   \frac{K_{\sup}}{\epsilon^\alpha}, \nonumber \\
& \mathbb{E}[F] \leq P_{max}, \nonumber \\
& \mathbb{E}[F^2] \leq \frac{K^2_{\sup}}{\epsilon^{2\alpha}}P_{max}Vol(B^{\mathbb{R}^p}_{\epsilon}(\iota(x))\cap \iota(M)) \leq \frac{C K^2_{\sup}P_{max}}{\epsilon^{2\alpha-d}}. \nonumber 
\end{align}
Thus,
\begin{align}
\sigma^2:=\text{Var}(F)  \leq \frac{C K^2_{\sup}P_{max}}{\epsilon^{2\alpha-d}}-P^2_{max} \leq \frac{C_1}{\epsilon^{2\alpha-d}},
\end{align} 
when $\epsilon$ is small enough. We apply Bernstein's inequality to provide a large deviation bound. Recall Bernstein's inequality
\begin{equation}
\Pr \left\{\left|\frac{1}{n}\sum_{i=1}^n F_i- \mathbb{E}[F]\right| > \beta\right\} \leq e^{-\frac{n\beta^2}{2\sigma^2 + \frac{2}{3}b\beta}}.
\end{equation}
For any $\beta<1$ and $\epsilon$ small, there is a constant $C_2$ such that $2\sigma^2 + \frac{2}{3}b\beta<\frac{1}{C_2\epsilon^{2\alpha-d}}$.
Hence,
\begin{align}
e^{-\frac{n\beta^2}{2\sigma^2 + \frac{2}{3}b\beta}}<e^{-C_2 n\beta^2\epsilon^{2\alpha-d}}.
\end{align}
Note that $e^{-C_2 n\beta^2\epsilon^{2\alpha-d}}<\frac{1}{n^2}$ if and only if $\beta^2>\frac{2\log n}{C_2 n \epsilon^{2\alpha-d}}$. Take $\beta \rightarrow 0$, we conclude that for the fixed $x$, if $\frac{\log n}{n \epsilon^{2\alpha-d}} \rightarrow 0$ as $n \rightarrow \infty$, then, a.s.
\begin{align}
\left|\frac{1}{n}\sum_{i=1}^n K_{\epsilon}(\iota(x),\iota(x_i))-\mathbb{E}K_{\epsilon}(x)\right| \rightarrow 0\,.
\end{align}

\underline{\textbf{Pointwise convergence of $\int_{M} |\frac{1}{n}\sum_{i=1}^n K_{\epsilon}(\iota(x),\iota(x_i))| dV(x)$ to $\int_{M} |\mathbb{E}K_{\epsilon}(x)|dV(x)$}}

Suppose $P$ is the density function of the random variable $X$ on $M$. We define a random variable $G=\int_{M} K_{\epsilon}(\iota(x),\iota(X)) dV(x)$. Then, $G(i)=\int_{M} K_{\epsilon}(\iota(x),\iota(x_i)) dV(x)$ can be regarded as i.i.d samples from $F$. Note that we use Lemma \ref{volume form in normal coordinate}, so that $Vol(B^{\mathbb{R}^p}_{\epsilon}(\iota(x))\cap \iota(M)) \leq C \epsilon^d$ for some constant $C$ depending on the manifold $M$. Hence, we have
\begin{align}
b=\|G\|_{\infty} \leq   \frac{C K_{\sup}}{\epsilon^{\alpha-d}}.
\end{align}
If we apply Fubini's theorem, we have
\begin{align}
\mathbb{E}[G] =& \int_{M} \int_{M} K_{\epsilon}(\iota(x),\iota(y)) dV(x) P(y) dV(y)  \\
=& \int_{M} \int_{M} K_{\epsilon}(\iota(x),\iota(y))  P(y) dV(y) dV(x) \nonumber \\
\leq & P_{max}  \int_{M} \int_{M} K_{\epsilon}(\iota(x),\iota(y)) dV(y) dV(x) \nonumber\\
= &   P_{max}  \int_{M} 1 dV(x) = P_{max} Vol(M). \nonumber
\end{align}
If we apply the Cauchy Schwarz inequality and the Fubini's Theorem,
\begin{align}
\mathbb{E}[G^2] =& \int_{M} |\int_{M} K_{\epsilon}(\iota(x),\iota(y)) dV(x)|^2 P(y) dV(y) \\
\leq & Vol(M) \int_{M}\int_{M} \Big[ K_{\epsilon}(\iota(x),\iota(y)) \big]^2 dV(x)P(y) dV(y) \nonumber \\
\leq & \frac{C K^2_{\sup}P_{max}(Vol(M))^2}{\epsilon^{2\alpha-d}}. \nonumber
\end{align}
Thus,
\begin{align}
\sigma^2:=\text{Var}(G)  \leq\frac{C K^2_{\sup}P_{max}(Vol(M))^2}{\epsilon^{2\alpha-d}}- (P_{max} Vol(M))^2 \leq \frac{C_3}{\epsilon^{2\alpha-d}},
\end{align} 
when $\epsilon$ is small enough. Note that
\begin{align}
\int_{M} \left|\frac{1}{n}\sum_{i=1}^n K_{\epsilon}(\iota(x),\iota(x_i))\right| dV(x) =\int_{M} \frac{1}{n}\sum_{i=1}^n K_{\epsilon}(\iota(x),\iota(x_i)) dV(x)=\frac{1}{n}\sum_{i=1}^n G(i),
\end{align}
and
\begin{align}
\int_{M} |\mathbb{E}K_{\epsilon}(x)|dV(x)=& \int_{M} \int_{M} K_{\epsilon}(\iota(x),\iota(y))  P(y) dV(y) dV(x)  \\
=& \int_{M} \int_{M} K_{\epsilon}(\iota(x),\iota(y)) dV(x) P(y) dV(y)  =\mathbb{E}[G] \nonumber
\end{align}
We apply Bernstein's inequality
\begin{align}
& \Pr \left\{\left|\int_{M} \Big|\frac{1}{n}\sum_{i=1}^n K_{\epsilon}(\iota(x),\iota(x_i))\Big| dV(x)- \int_{M} |\mathbb{E}K_{\epsilon}(x)|dV(x)\right| > \beta\right\} \\
 =& \Pr \left\{\left|\frac{1}{n}\sum_{i=1}^n G(i)- \mathbb{E}[G]\right| > \beta\right\} \leq e^{-\frac{n\beta^2}{2\sigma^2 + \frac{2}{3}b\beta}}. \nonumber
\end{align}
For any $\beta<1$ and $\epsilon$ small, there is a constant $C_4$ such that $2\sigma^2 + \frac{2}{3}b\beta<\frac{1}{C_4\epsilon^{2\alpha-d}}$.
Hence,
\begin{align}
e^{-\frac{n\beta^2}{2\sigma^2 + \frac{2}{3}b\beta}}<e^{-C_4 n\beta^2\epsilon^{2\alpha-d}}.
\end{align}
Note that $e^{-C_4 n\beta^2\epsilon^{2\alpha-d}}<\frac{1}{n^2}$ if and only if $\beta^2>\frac{2\log n}{C_4 n \epsilon^{2\alpha-d}}$. Take $\beta \rightarrow 0$, we have the following conclusion. 
If $\frac{\log n}{n \epsilon^{2\alpha-d}} \rightarrow 0$ as $n \rightarrow \infty$, then, a.s.
\begin{align}
\int_{M} \left|\frac{1}{n}\sum_{i=1}^n K_{\epsilon}(\iota(x),\iota(x_i))\right| dV(x) \rightarrow \int_{M} |\mathbb{E}K_{\epsilon}(x)|dV(x)\,.
\end{align}

By Scheffe's Lemma, if $\frac{\log n}{n \epsilon^{2\alpha-d}} \rightarrow 0$ as $n \rightarrow \infty$, then, a.s.
\begin{align}
\int_{M} \left|\frac{1}{n}\sum_{i=1}^n K_{\epsilon}(\iota(x),\iota(x_i))-\mathbb{E}K_{\epsilon}(x)\right| dV(x) \rightarrow 0. 
\end{align}

\end{proof}

\begin{proof}[Proof of Proposition \ref{bias L1 case}]

Next, we show that $\int_{M} |\mathbb{E}K_{\epsilon}(x)-P(x)| dV(x) \rightarrow 0$, as $\epsilon \rightarrow 0$. Note that
\begin{align}
|\mathbb{E}K_{\epsilon}(x)|=|\int_{M}K_{\epsilon}(\iota(x),\iota(y))P(y)dV(y)| = \int_{M}K_{\epsilon}(\iota(x),\iota(y))P(y)dV(y) \leq P_{max}.
\end{align}
Hence, if we can show that $\mathbb{E}K_{\epsilon}(x) \rightarrow P(x)$ for almost every $x$, then $\int_{M} |\mathbb{E}K_{\epsilon}(x)-P(x)| dV(x) \rightarrow 0$ follows from bounded convergence Theorem.

\underline{\textbf{$\mathbb{E}K_{\epsilon}(x) \rightarrow P(x)$, when $P$ is continuous}}
\begin{align}
|\mathbb{E}K_{\epsilon}(x)- P(x)| = & |\int_{M}K_{\epsilon}(\iota(x),\iota(y))P(y)dV(y)-\int_{M}K_{\epsilon}(\iota(x),\iota(y))dV(y)P(x)| \\
\leq &  \int_{M}K_{\epsilon}(\iota(x),\iota(y))|P(y)-P(x)|dV(y) \nonumber \\
\leq & \sup_{y: \iota(y) \in B^{\mathbb{R}^p}_{\epsilon}(\iota(x))} |P(y)-P(x)|. \nonumber
\end{align}
As $\epsilon \rightarrow 0$, we have $\iota(y) \rightarrow \iota(x)$ and $y \rightarrow x$. Hence, we have $ |P(y)-P(x)| \rightarrow 0$ and $ |\mathbb{E}K_{\epsilon}(x)- P(x)| \rightarrow 0$. 

\underline{\textbf{$\mathbb{E}K_{\epsilon}(x) \rightarrow P(x)$ for almost every $x$, when $\alpha=d$ and $P$ is bounded}}

For $\epsilon$ small enough, since $M$ is compact, we can cover $M$ by finite many open sets $\{\mathcal{O}_i, z_i\}$ with the following properties.
\begin{enumerate}
\item
$z_i \in \mathcal{O}_i \subset M$. 
\item
Let $inj(z_i)$ be the injectivity radius at $z_i$, then any for  $x \in \mathcal{O}_i$, we have $ B^{\mathbb{R}^p}_{\epsilon}(\iota(x)) \cap \iota(M) \subset \iota(B_{inj(z_i)/2}(z_i))$, where $B_r(z_i)$ is a geodesic ball of radius $r>0$.
\end{enumerate}
Since $\exp_{z_i}$ is a diffeomorphism, we have a bi-Lipschitz constant $r_i>1$ for $\exp_{z_i}$ on the closure of $B_{inj(z_i)/2}(z_i)$. Let $r=\max_i r_i$.

For any $x \in \mathcal{O}_i$,
\begin{align}
|\mathbb{E}K_{\epsilon}(x)- P(x)| = & \Big|\int_{M}K_{\epsilon}(\iota(x),\iota(y))P(y)dV(y)-\int_{M}K_{\epsilon}(\iota(x),\iota(y))dV(y)P(x)\Big| \\
\leq &  \int_{M}K_{\epsilon}(\iota(x),\iota(y))|P(y)-P(x)|dV(y). \nonumber \\
= &  \int_{B^{\mathbb{R}^p}_{\epsilon}(\iota(x)) \cap \iota(M) } K_{\epsilon}(\iota(x),\iota(y))|P(y)-P(x)|dV(y) \nonumber
\end{align}
Suppose $x=\exp_{z_i}(\textbf{v}_1)$ and $y=\exp_{z_i}(\textbf{v})$ for $\textbf{v}, \textbf{v}_1 \in T_{z_i}M \approx \mathbb{R}^d$. Denote 
\begin{align}
D(\textbf{v}_1)=(\iota \circ \exp_{z_i})^{-1} \big(B^{\mathbb{R}^p}_{\epsilon}(\iota(x)) \cap \iota(M) \big).
\end{align}
Note that we have $D(\textbf{v}_1) \subset B^{\mathbb{R}^d}_{inj(z_i)/2}(z_i)$. Suppose $V_{z_i}(\textbf{v})$ is the volume form of $\exp_{z_i}$ at $\textbf{v}$. Then, by Lemma \ref{volume form in normal coordinate} and Definition \ref{max of volume for on inj ball} for any $\textbf{v} \in  B^{\mathbb{R}^d}_{inj(z_i)/2}(z_i)$, $V_{z_i}(\textbf{v})< V_{max}$.

By Lemma \ref{geodesic and euclidean distance}, there is a constant $C_5$ such that the diameter of $B^{\mathbb{R}^p}_{\epsilon}(\iota(x)) \cap \iota(M)$ as a subset of $\iota(M)$ is bounded above by $C_5 \epsilon$, hence $D(\textbf{v}_1) \subset Q_{2r C_5 \epsilon}(\textbf{v}_1) \subset \mathbb{R}^d$. where $ Q_{2r C_5 \epsilon}(\textbf{v}_1)$ is an open cube of side length $2r C_5 \epsilon$ centered at $\textbf{v}_1$.

Hence, 
\begin{align}
& \int_{ B^{\mathbb{R}^p}_{\epsilon}(\iota(x)) \cap \iota(M) } K_{\epsilon}(\iota(x),\iota(y))|P(y)-P(x)|dV(y) \\
=& \int_{D(\textbf{v}_1)} K_{\epsilon}(\iota \circ \exp_{z_i}(\textbf{v}_1), \iota \circ \exp_{z_i}(\textbf{v}))|P( \exp_{z_i}(\textbf{v}))-P( \exp_{z_i}(\textbf{v}_1))|V_{z_i}(\textbf{v}) d\textbf{v}. \nonumber \\
\leq  &  V_{max} \int_{D(\textbf{v}_1)} K_{\epsilon}(\iota \circ \exp_{z_i}(\textbf{v}_1), \iota \circ \exp_{z_i}(\textbf{v}))|P( \exp_{z_i}(\textbf{v}))-P( \exp_{z_i}(\textbf{v}_1))|d\textbf{v} \nonumber \\
\leq &  V_{max}\frac{K_{\sup}}{\epsilon^{d}} \int_{Q_{2r C_5 \epsilon}(\textbf{v}_1)} |P( \exp_{z_i}(\textbf{v}))-P( \exp_{z_i}(\textbf{v}_1))| \bigchi_{D(\textbf{v}_1)}(\textbf{v} ) d\textbf{v} \nonumber \\
=& V_{max}\frac{K_{\sup}}{\epsilon^{d}} Vol(Q_{2r C_5 \epsilon}(\textbf{v}_1)) \frac{\int_{Q_{2r C_5 \epsilon}(\textbf{v}_1)} |P( \exp_{z_i}(\textbf{v}))-P( \exp_{z_i}(\textbf{v}_1))| \bigchi_{D(\textbf{v}_1)}(\textbf{v} ) d\textbf{v} }{Vol(Q_{2r C_5 \epsilon}(\textbf{v}_1))}\nonumber \\
=& V_{max} (2 C_5)^d  r^{d} K_{\sup} \frac{\int_{Q_{2r C_5 \epsilon}(\textbf{v}_1)} |P( \exp_{z_i}(\textbf{v}))-P( \exp_{z_i}(\textbf{v}_1))| \bigchi_{D(\textbf{v}_1)}(\textbf{v} ) d\textbf{v}}{Vol(Q_{2r C_5 \epsilon}(\textbf{v}_1))}  \nonumber \\
\leq &  V_{max}  (2 C_5)^d  r^{d} K_{\sup} \frac{\int_{Q_{2r C_5 \epsilon}(\textbf{v}_1)} |P( \exp_{z_i}(\textbf{v}))\bigchi_{D(\textbf{v}_1)}(\textbf{v} )-P( \exp_{z_i}(\textbf{v}_1))|  d\textbf{v}}{Vol(Q_{2r C_5 \epsilon}(\textbf{v}_1))}. \nonumber
\end{align}
By \cite{zygmund1989certain},  as $\epsilon \rightarrow 0$,
\begin{align}
 \frac{\int_{Q_{2r C_5 \epsilon}(\textbf{v}_1)} |P( \exp_{z_i}(\textbf{v}))\bigchi_{D(\textbf{v}_1)}(\textbf{v} )-P( \exp_{z_i}(\textbf{v}_1))|  d\textbf{v}}{Vol(Q_{2r C_5 \epsilon}(\textbf{v}_1))}\rightarrow 0,
\end{align}
for almost every $\textbf{v}_1$ in $\exp^{-1}(\mathcal{O}_i)$.  Hence, $\mathbb{E}K_{\epsilon}(x) \rightarrow P(x)$ for almost every $x \in \mathcal{O}_i$. There are finite many $\mathcal{O}_i$, thus we conclude that $\mathbb{E}K_{\epsilon}(x) \rightarrow P(x)$ for almost every $x$.
\end{proof}

\end{document}